\documentclass[12pt,showkeys]{amsart}
\def\makeheadbox{{%
\hbox to0pt{\vbox{\baselineskip=10dd\hrule\hbox
to\hsize{\vrule\kern3pt\vbox{\kern3pt
\hbox{\bfseries Draft for discussion }
\hbox{Date of this version: 13.04.21}
\kern3pt}\hfil\kern3pt\vrule}\hrule}%
\hss}}}

\usepackage{color}
\usepackage{graphicx}
\usepackage{subfigure}
\usepackage{amssymb}
\usepackage{amsmath}
\usepackage{multirow}
\usepackage{enumerate}
\usepackage{epstopdf}
\usepackage{cite,stmaryrd,txfonts}
\usepackage{tikz}
\usetikzlibrary{snakes}

\usepackage{epic}%% package for dash line \dashline
\usepackage{empheq}
\usepackage{cases}

\def\cequiv{\raisebox{-1.5mm}{$\;\stackrel{\raisebox{-3.9mm}{=}}{{\sim}}\;$}}

\newtheorem{theorem}{Theorem}[section]
\newtheorem{remark}[theorem]{Remark}\newtheorem{proposition}[theorem]{Proposition}
\newtheorem{lemma}[theorem]{Lemma}
\newtheorem{corollary}[theorem]{Corollary}

\newcounter{mnote}
\let\oldmarginpar\marginpar
\renewcommand\marginpar[1]{\-\oldmarginpar[\raggedleft\footnotesize #1]
  {\raggedright\footnotesize #1}}
  
\usepackage{footnote}
\usepackage{booktabs}
\newcommand{\ud}{\,d}

\newcommand{\jump}[1]{\llbracket {#1} \rrbracket}
\usepackage{rotating}

\numberwithin{equation}{section}

\usepackage{cite}

\renewcommand{\dfrac}[2]{{
\renewcommand{\arraystretch}{1.375}
\begingroup\displaystyle
\rule[0pt]{0pt}{11pt}#1\endgroup
\over\displaystyle\rule[-3pt]{0pt}{11pt}#2
}} 

\usepackage[colorlinks,linkcolor=black,
anchorcolor=black,
citecolor=black]{hyperref}
\usepackage[shortlabels]{enumitem} 
\setlist[enumerate]{nosep}
\usepackage{color, soul}

\input{undertilde}
\def\uu{\undertilde{u}}
\def\uw{\undertilde{w}}
\def\uv{\undertilde{v}}
\def\uV{\undertilde{V}}
\def\uL{\undertilde{L}}
\def\uH{\undertilde{H}}

\def\uf{\undertilde{f}}

\def\uy{\undertilde{y}}
\def\uz{\undertilde{z}}

\def\curl{{\rm curl}}
\def\dv{{\rm div}}

\begin{document}

\title[RRM element for Poisson]{Optimal quadratic element on rectangular grids for $H^1$ problems}

\author{Huilan Zeng}
\address{LSEC, ICMSEC, Academy of Mathematics and System Sciences, Chinese Academy of Sciences, Beijing 100190, China; School of Mathematical Sciences, University of Chinese Academy of Sciences, Beijing 100049, China}
\email{zhl@lsec.cc.ac.cn}
\author{Chensong Zhang}
\address{LSEC, ICMSEC and NCMIS, Academy of Mathematics and System Sciences, Chinese Academy of Sciences, Beijing 100190, China}
\email{zhangcs@lsec.cc.ac.cn}
\author{Shuo Zhang}
\address{LSEC, ICMSEC and NCMIS, Academy of Mathematics and System Sciences, Chinese Academy of Sciences, Beijing 100190, China}
\email{szhang@lsec.cc.ac.cn}
\thanks{Zeng and C.-S. Zhang are partially supported by National Key Research and Development Program 2016YFB0201304, China. C.-S. Zhang is also supported by the Key Research Program of Frontier Sciences of CAS. S. Zhang is partially supported by National Natural Science Foundation,  11471026 and 11871465, China}

%\thanks{}



\maketitle

\begin{abstract}
In this paper, a piecewise quadratic finite element method on rectangular grids for $H^1$ problems is presented. The proposed method can be viewed as a reduced rectangular Morley (RRM) element. For the source problem, the convergence rate of this scheme is proved to be $O(h^2)$ in the energy norm on uniform grids over a convex domain. A lower bound of the $L^2$-norm error is also proved, which makes the capacity of this scheme more clear. For the eigenvalue problem, the computed eigenvalues by this element are shown to be the lower bounds of the exact ones. Some numerical results are presented to verify the theoretical findings.
\end{abstract}

%\tableofcontents

%
%
%
\section{Introduction}
\label{sec:intro}

The design and capacity analysis of the discretization schemes for the source problem (say, the boundary value problem) and the eigenvalue problem are key issues in numerical analysis and, in general, of approximation theory. When the approximation of functions in Sobolev spaces is performed using piecewise polynomials defined on a domain partition, lower-degree polynomials are often preferred in order to achieve a simpler interior structure. A finite element scheme with polynomials of the total degree no more than $k$, denoted by $P_{k}$, is called optimal if it achieves $\mathcal{O}(h^{k+1-m})$ accuracy in the energy norm for $H^{m}$ elliptic problems.
In this paper, we present an optimal quadratic element scheme for the $H^1$ problems, including the source problem and the eigenvalue problem, on rectangular grids, and present its error analysis. 

The study of optimal finite element schemes has been attracting wide interests. For the case wherein the grid comprises simplexes, there are already some systematic results. It is known that the Lagrange finite elements of arbitrary degree on domains of arbitrary dimension are optimal conforming elements for second-order elliptic problems. At the same time, a systematic family of minimal-degree nonconforming finite elements is proposed by \cite{Wang.M;Xu.J2012}, where $m$-th degree polynomials work for $2m$-th order elliptic problems in $R^n$ for any $n\geqslant m$. Known as the Wang-Xu or Morley-Wang-Xu family, these elements are constructed based on the perfect matching between the dimension of $m$-th degree polynomials and the dimension of $(n-k)$-subsimplexes with $1\leqslant k\leqslant m$. The generalisation to the cases where $n<m$ is attracting increasing research interest (see, e.g., \cite{Wu.S;Xu.J2019}). These spaces can be naturally used for both the source problem and the eigenvalue problem. On the other hand, to clarify the capacity of the schemes clearly, some kinds of extremal analysis have also been conducted, including, e.g., lower bounds of the error estimates and guaranteed bounds of the computed eigenvalues. We refer to, e.g., \cite{Lin.Q;Xie.H;Xu.J2014} for a general analysis of the lower bounds of the discretization error for piecewise polynomials, and \cite{XY.Meng;XQ.Yang;S.Zhang2016,J.Hu;XQ.Yang;S.Zhang2015,J.Hu;ZC.Shi2013,Hu.J;Shi.Z-C2012} for specific analysis with certain finite element schemes. We refer to, e.g.,\cite{Q.Lin;JF.Lin2007,M.Armentano;R.Duran2004,ZM.Zhang;YD.Yang2007,Y.Yang;Z.Zhang2010,Y.Yang;J.Han;H.Bi;Y.Yu2015,FS.Luo;Q.Lin;Xie2012,YA.Li2011,J.Hu;Y.Huang;Q.Lin2014,Carstensen.C;Gallistl.D2014,Carstensen.C;Gedicke.J2014} for the computed guaranteed bounds of certain eigenvalue problems. The extremal analysis is naturally used on or ready to be generalized to optimal schemes. 

When the grid comprises  shapes other than simplexes, the design of optimal schemes becomes more complicated. We would like to recall that, $Q_k$ (rather than $P_k$) polynomials are used for $2k$-th order problems on $\mathbb{R}^n$ rectangular grids by \cite{Hu.J;Zhang.Sy2015}, which form minimal \emph{conforming} element spaces. For biharmonic equation, some low-degree rectangular elements have been designed, including the rectangular Morley element and incomplete $P_3$ element. Very recently, a space, consisting of exactly piecewise quadratic polynomials, is constructed and shown convergent for the biharmonic equation on general quadrilateral grids, which forms a convergent scheme of the minimal degree \cite{Shuo.Zhang2017}. Also, there have been several rectangle elements for $H^1$ problems in the literature \cite{Lee.H;Sheen.D2006,Hu.J;Zhang.S2013,SheenKimLuoMengNamPark2013,Wilson.E;Taylor.R;Doherty.W;Ghaboussi.J1973}. In \cite{Lee.H;Sheen.D2006}, an enriched quadratic nonconforming element on rectangles is introduced, and second-order error is shown, which is generalized to higher orders by \cite{Hu.J;Zhang.S2013}. Another second-order quadratic element is given by \cite{SheenKimLuoMengNamPark2013}, where the spline technique is used, but the shape function space on a cell is not exactly $P_2$. The Wilson element \cite{Wilson.E;Taylor.R;Doherty.W;Ghaboussi.J1973} is the first quadratic quadrilateral nonconforming element. Despite its superior performance in practice, as shown in \cite{ZC.Shi1986}, its global asymptotic convergence rate is the same as that of the bilinear element, due to low internal continuity. Generally, this deficiency can be compensated by equipping the piecewise quadratic polynomial with second order moment-continuity across the internal edges. In this way, the moment-continuous (MC) element space is defined. However, it is proved in Appendix \ref{sec:appA}, that the MC element space possesses essentially the same accuracy as that of the bilinear element space, and thus it fails to reach second-order convergence rate. To our best knowledge, it remains open whether an optimal scheme can be constructed with degrees higher than minimal even on rectangular grids and for $H^1$ problems.

In this paper, we study the optimal finite element construction for the $H^1$ problems, and present a finite element space comprising piecewise quadratic polynomials on uniform rectangular grids that can provide $\mathcal{O}(h^2)$ convergence in energy norm for the source and eigenvalue problems. The computed eigenvalues are lower bounds of the exact ones, which can be proved theoretically and verified numerically.  Only rectangular grids are taken into consideration herein, but if a quadrilateral grid is only a sufficiently small perturbation of a uniform one, then an optimal convergence rate could be expected on it. Moreover, the finite element functions cannot be described with free rein cell by cell. Similar to the elements described in \cite{Fortin.M;Soulie.M1983, Park.C;Sheen.D2003,Shuo.Zhang2017} and in many spline-type methods, the number of continuity restrictions of the finite element function is greater than  the dimension of the local polynomial space. We believe this difficulty is not abnormal for low-degree schemes. In general, these cells can share interfaces with more neighbour cells, and more continuity restrictions will strengthen the requirement for higher-degree polynomials, generally higher than the order of the underlying Sobolev space. Thus, constructing consistent finite elements in the formulation of Ciarlet's triple is difficult with $m$-th degree polynomials for $H^m$ problems even on rectangular grids. Here we utilize some non-standard technical approaches to overcome the difficulty for both implementation and especially theoretical analysis. 

The main difficulty is that the local interpolation is too difficult, if ever possible, to be established, which plays a fundamental role in the approximation error analysis for the source problem and the guaranteed bounds analysis for the eigenvalue problem. Notice that the space constructed herein can be viewed as a reduced rectangular Morley element space. Similar to the approach in \cite{Shuo.Zhang2017} but with technical modifications, we can determine that the finite element functions are discrete stream functions of the discrete divergence-free functions constructed in a study \cite{Park.C;Sheen.D2003}; using this exact relation, we can perform the approximation estimation indirectly.  Also, the discretization of the eigenvalue problem can be viewed as an inner approximation of the rectangular Morley element scheme, and this helps avoid the direct dependence on an interpolation. This newly-designed routine method of theoretical analysis can be potentially used to find out other optimal schemes. 

Finally, we remark that, two examples, namely the rectangular Morley (RM) element and the reduced rectangular Morley (RRM) element, are reported in this paper that when used for the eigenvalue problem, errors of the eigenvalues and eigenfunctions are of the same order. This unusual performance is due to the fact that no nontrivial conforming finite element subspace can be found contained in these two spaces.

The rest of the paper is organized as follows. In Section \ref{sec:pre}, some preliminaries are given and some related low-degree rectangle elements are reviewed. In Section \ref{sec:rmrev}, the rectangular Morley element is revisited. In Section \ref{sec:rrmscheme}, a reduced rectangular Morley element scheme is presented for both source problem and eigenvalue problem. In Section \ref{sec:analysisrrm}, the convergence analysis and lower bound properties are shown for the RRM element scheme. In Section \ref{sec:concdisc}, some concluding remarks and discussions are given. In contrast to a general implementation approach in Section \ref{sec:analysisrrm},  concise sets of basis functions of the MC element and the RRM element are presented in the appendix.

\section{Preliminaries} 
\label{sec:pre}

\subsection{Notations}
Throughout this paper, we use $\Omega$ for a simply-connected polygonal domain in~$\mathbb{R}^{2}$.  We use $\nabla$, $\curl$, $\dv$, and $\nabla^2$ for the gradient operator, curl operator, divergence operator, and Hessian operator, respectively. As usual, we use $H^2(\Omega)$, $H^2_0(\Omega)$, $H^1(\Omega)$, $H^1_0(\Omega)$,  and $L^2(\Omega)$ for certain Sobolev spaces. Specifically, we denote $\displaystyle L^2_0(\Omega):=\Big\{w\in L^2(\Omega):\int_\Omega w dx=0\Big\}$, $\undertilde{H}{}^1_0(\Omega):=\big(H^1_0(\Omega)\big)^2$, and $\uH{}^1_{\bf n}(\Omega):=\Big\{v \in \big(H^1(\Omega)\big)^{2}:v\cdot {\bf n}\big|_{\partial \Omega}=0\Big\}$. Denote, by $\uH{}^{-1}(\Omega)$ and $H^{-1}(\Omega)$, the dual spaces of $\uH{}^1_0(\Omega)$ and $H^1_0(\Omega)$, respectively.  We use $``\undertilde{~}"$ for vector valued quantities in the present paper, and $\uv{}^1$ and $\uv{}^2$ for the two components of the function $\uv$. We utilize the subscript $``\cdot_h"$ to indicate the dependence on grids. Particularly, an operator with the subscript '$``\cdot_h"$ implies the operation is done cell by cell. Finally, $\lesssim$, $\gtrsim$, and $\cequiv$ respectively denote $\leqslant$, $\geqslant$, and $=$ up to a generic positive constant, which  might depend on the shape-regularity of subdivisions, but not on the mesh-size $h$ \cite{J.Xu1992}.

Let $\mathcal{G}_h$ be in a regular family of quadrilateral grids of domain $\Omega$. Let $\mathcal{N}_h$ be the set of all vertices, $\mathcal{N}_h=\mathcal{N}_h^i\cup\mathcal{N}_h^b$, with $\mathcal{N}_h^i$ and $\mathcal{N}_h^b$ comprising the interior vertices and the boundary vertices, respectively. Similarly, let $\mathcal{E}_h=\mathcal{E}_h^i\bigcup\mathcal{E}_h^b$ be the set of all the edges, with $\mathcal{E}_h^i$ and $\mathcal{E}_h^b$ comprising the interior edges and the boundary edges, respectively. For an edge $e$, $\mathbf{n}_e$ is a unit vector normal to $e$ and $\boldsymbol{\tau}_e$ is a unit tangential vector of $e$ such that $\mathbf{n}_e\times \boldsymbol{\tau}_e>0$. On the edge $e$, we use $\llbracket\cdot\rrbracket_e$ for the jump across $e$. If $e\subset\partial\Omega$, then $\llbracket\cdot\rrbracket_e$ is the evaluation on $e$. The subscript ${\cdot}_e$ can be dropped when there is no ambiguity brought in.

\subsection{Some rectangular finite element spaces}

Suppose that $K\subset \mathbb{R}^{2}$ represents a rectangle with sides parallel to the two axis respectively. Let $(x_{K},y_{K})$ be the barycenter of $K$. Let $h_{x,K}$, $h_{y,K}$ be the length of $K$ in the $x$ and $y$ directions, respectively.  Let $a_{i}$ and $e_{i}$ denote an vertex and an edge of K, respectively.
Let $h := \max\limits_{K \in \mathcal{G}_{h}}\max\{h_{x,K},h_{y,K} \}$ be the mesh size of $\mathcal{G}_{h}$. When $\mathcal{G}_{h}$ is uniform, we denote $h_{x} : = h_{x,K}$ and $h_{y} : = h_{y,K}$. Let $P_l(K)$ denote the space of polynomials on $K$ of total degree no bigger than~$l$. Let $Q_l(K)$ denote  the space of polynomials of degree no bigger than~$l$ in each variable. Similarly, we define spaces $P_l(e)$ and $Q_l(e)$ on an edge~$e$.

\subsubsection{The $Q_{1}$ element}The $Q_{1}$ element is defined by 
$(K,P_{K}^{\text{BL}},D_{K}^{\text{BL}})$ with the following properties:
\begin{itemize}
\item[(a)] $K$ is a rectangle;
\item[(b)] $P_{K}^{\text{BL}} = Q_{1}(K)$;
\item[(c)] for any $v\in H^{1}(K)$, $D_{K}^{\text{BL}} =\big\{v(a_{i})\big\}_{i=1:4}$.
\end{itemize}

\noindent Define the $Q_{1}$ element space as
$$
V_{h}^{\text{BL}} :=\Big\{w_{h}\in H^{1}(\Omega) : w_{h}|_{K} \in Q_{1}(K),\ \forall K\in \mathcal{G}_{h} \Big\}.
$$

\noindent Associated with $H_{0}^{1}(\Omega)$, we define
$
V_{h0}^{\text{BL}} :=\Big\{w_{h}\in V_{h}^{\text{BL}} :  w_{h}(a) = 0, \ \forall a\in \mathcal{N}_{h}^{b}\Big\}.
$

\subsubsection{The Park--Sheen (PS) element}The PS element \cite{Park.C;Sheen.D2003}  is a piecewise linear nonconforming
finite element space for $H^{1}$ problems. It is defined as
$$
V_{h}^{\rm{PS}} :=\Big\{w_{h}\in L^{2}(\Omega) : w_{h}|_{K} \in P_{1}(K),\ \forall K\in \mathcal{G}_{h}, \mbox{ and }\fint_{e} \jump{w_{h}} \ud s  =0, \ \forall e \in \mathcal{E}_{h}^{i}  \Big\}.
$$
Associated with $H_{0}^{1}(\Omega)$, we define
$
V_{h0}^{\rm{PS}} :=\Big\{w_{h} \in V_{h}^{\rm{PS}} :  \fint_{e} w_{h} \ud s =0, \ \forall e\in \mathcal{E}_{h}^{b}\Big\}.
$

\subsubsection{The rotated $Q_{1}$ ($Q_{1}^{\rm{rot}}$) element}The $Q_{1}^{\rm{rot}}$ element is defined by 
$(K,P_{K}^{\rm{rQ}},D_{K}^{\rm{rQ}})$ with the following properties:
\begin{itemize}
\item[(a)] $K$ is a rectangle; 
\item[(b)] $P_{K}^{\rm{rQ}} = P_{1}(K) + \text{span}\{x^{2}-y^{2}\}$;
\item[(c)] for any $v\in H^{1}(K) $, $D_{K}^{\rm{rQ}} = \Big\{\fint_{e_{i}}v \ud s \Big\}_{i=1:4}$.
\end{itemize}

\noindent Define the $Q_{1}^{\rm{rot}}$ element space as
$$
V_{h}^{\text{rQ}} := \Big\{w_{h}\in L^{2}(\Omega) : w_{h}|_{K} \in P_{K}^{\rm rQ},\ \forall K \in \mathcal{G}_{h}, \mbox{ and } \fint_{e} \jump{w_{h}}=0 \ud s, \ \forall e \in \mathcal{E}_{h}^{i}  \Big\}.
$$
Associated with $H_{0}^{1}(\Omega)$, we define
$
V_{h0}^{\text{rQ}} :=\Big\{w_{h} \in V_{h}^{\text{rQ}} :  \fint_{e} w_{h} \ud s =0, \ \forall e\in \mathcal{E}_{h}^{b}\Big\}.
$

\subsubsection{The Lin--Tobiska--Zhou (LTZ) element}The LTZ element(\cite{Lin.Q;Tobiska.L;Aihui.Zhou2005, Zhang.S2016nm}) is defined by 
$(K,P_{K}^{\rm{LTZ}},D_{K}^{\rm{LTZ}})$ with the following properties:
\begin{itemize}
\item[(a)] $K$ is a rectangle; 
\item[(b)] $P_{K}^{\rm{LTZ}} = P_{1}(K) + \text{span}\{x^{2}, y^{2}\}$;
\item[(c)]  for any $v\in H^{1}(K) $, $D_{K}^{\rm{LTZ}} =  \Big\{\fint_{K} v \ud x d y, \ \fint_{e_{1}}v \ud s, \ \ldots,\ \fint_{e_{4}}v \ud s \Big\}$.
\end{itemize}
Define the LTZ element space as
$$
V_{h}^{\rm{LTZ}} :=\Big\{w_{h}\in L^{2}(\Omega) : w_{h}|_{K} \in P_{K}^{\rm{LTZ}}, \ \forall K\in \mathcal{G}_{h}, \mbox{ and } \fint_{e} \jump{w_{h}} \ud s =0, \ \forall e \in \mathcal{E}_{h}^{i}  \Big\}.
$$
Associated with $H_{0}^{1}(\Omega)$, we define
$
V_{h0}^{\rm{LTZ}} := \Big\{w_{h} \in V_{h}^{\rm{LTZ}} :  \fint_{e}w_{h} \ud s =0, \ \forall  e\in \mathcal{E}_{h}^{b} \Big\},
$ and associated with $\uH{}_\mathbf{n}^1(\Omega)$, we define $
\uV{}_{h{\bf n}}^{\rm LTZ}:= \Big\{\uv{}_h\in \big(V_h^{\rm LTZ}\big)^2:\int_e\uv{}_h\cdot\mathbf{n}=0, \ \forall e\in\mathcal{E}_h^b \Big\}.$

\subsubsection{The Wilson element} 
The Wilson element is defined by 
$(K,P_{K}^{\rm{W}} ,D_{K}^{\rm{W}})$ with with the following properties:
\begin{itemize}
\item[(a)] $K$ is a rectangle; 
\item[(b)] $P_{K}^{\rm{W}} = P_{2}(K)$;
\item[(c)]  for any $v\in H^{2}(K) $, $D_{K}^{\rm{W}} = \Big\{ v(a_{1}),\ \ldots, v(a_{4}),\ \fint_{K}\partial_{xx}v \ud x d y, \ \fint_{K}\partial_{yy}v \ud x d y \Big\}$.
\end{itemize}
Define the Wilson element space as
$$
V_{h}^{\rm{W}} :=\Big\{w_{h}\in L^{2}(\Omega) : w_{h}|_{K} \in P_{K}^{\rm{W}},\ \forall K\in \mathcal{G}_{h}, \mbox{ and } w_{h} \mbox{ is continuous at any } a\in \mathcal{N}_{h}^{i}\Big\}.
$$
Associated with $H_{0}^{1}(\Omega)$, we define
$
V_{h0}^{\rm{W}} :=\Big\{w_{h}\in V_{h}^{\rm{W}} :  w_{h}(a)=0,\ \forall a\in \mathcal{N}_{h}^{b}\Big\}.
$

\subsubsection{The moment-continuous (MC) element}
Associated with  $\mathcal{G}_h$, the MC element space is defined as
$$
V_h^{\rm MC}:=\Big\{w_h\in L^2(\Omega):w_h|_K\in P_2(K), \mbox{ and } w_h \mbox{ is\ moment-continuous on }\mathcal{G}_h\Big\}.
$$
Associated with $H_{0}^{1}(\Omega)$, we define
$$
V_{h0}^{\rm MC}:=\Big\{w_h\in V_h^{\rm MC}: w_h \mbox{ is\ moment-homogeneous on } \mathcal{G}_h\Big\}.
$$

\noindent A piecewise quadratic polynomial function $w$ is moment-continuous of second-order, if
$$ 
\int_{e} \jump{w}_{e} v \ud s = 0,\; \forall v \in P_{1}(e),  \; e \in \mathcal{E}_{h}^{i}.
$$
\noindent Moreover, $w$ is moment-homogeneous of second-order, if
$
\int_{e} w v \ud s = 0, \; \forall v \in P_{1}(e), \; e \in \mathcal{E}_{h}^{b}.
$

A piecewise quadratic function $v_{h}$ belongs to $V_{h0}^{\rm MC}$ if and only if $v_{h}$ is continuous at the second-order Gauss points of any $e \in \mathcal{E}_{h}^{i}$ and vanishes on the Gauss points of any $e \in \mathcal{E}_{h}^{b}$.

\begin{theorem}\label{thm:hombc} 
If $\mathcal{G}_{h}$ is a $m \times n$ rectangular subdivision of $\Omega$, then ${\rm dim} (V_{h0}^{\rm MC} )= 2mn-m-n+1$.
\end{theorem}

\begin{theorem}\label{thm:nobc}
If $\mathcal{G}_{h}$ be a $m \times n$ rectangular subdivision of $\Omega$, then ${\rm dim} (V_{h}^{\rm MC} )= 2mn+2m+2n$. 
\end{theorem}

\noindent Detailed proof of Theorems \ref{thm:hombc} and \ref{thm:nobc} are put in Appendix \ref{sec:appA}, and available sets of basis functions of $V_{h0}^{\rm MC}$ and $V_{h}^{\rm MC}$ are also presented there.

\subsection{Some technical lemmas}
In addition to these spaces above, we denote
\begin{align*}
& \mathcal{L}^0_{h}:=\Big\{q\in L^2(\Omega):q|_K\in P_0(K)\Big\}, \quad \mathcal{L}^0_{h0}:=\mathcal{L}^0_{h}\cap L^2_0(\Omega), \\
& \uV{}^{\rm BL}_{h\bf n}:=\Big\{\uv\in (V^{\rm BL}_h)^2:\uv\cdot\mathbf{n}|_{\partial \Omega}=0\Big\}, \quad
\uV{}^{\rm rQ}_{h\bf n}:=\Big\{\uv\in (V^{\rm rQ}_h)^2:\int_e\uv\cdot\mathbf{n}=0,\ \forall\,e\in\mathcal{E}_h^b\Big\}, \\
& \uV{}^{\rm PS}_{h\bf n}:=(V^{\rm PS}_h)^{2}\cap \uV{}^{\rm rQ}_{h\bf n},  \quad  \uV{}_{h\bf n}^{\rm LTZ}:= \Big\{\uv{}\in (V_h^{\rm LTZ})^2:\int_e\uv{}\cdot\mathbf{n}=0,\ e\in\mathcal{E}_h^b \Big\}.
\end{align*}

Let $\Pi^{\rm rQ}_h$ and $\undertilde{\Pi}{}^{\rm rQ}_h$ be the nodal interpolation associated with $V^{\rm rQ}_h$ and $(V{}^{\rm rQ}_{h})^2$, respectively.
\begin{lemma}{\rm(\cite[Lemma 1]{Rannacher.R;Turek.S1992},\cite[Lemma 6]{Shuo.Zhang2017})}\label{lem:converRQ}
For the $Q_{1}^{\rm{rot}}$ element, we have
\begin{itemize}
\item[(1)] $\big|\Pi^{\rm rQ}_h v \big|_{1,h} \lesssim |v|_{1,h}$, $\forall v \in H_{0}^{1}(\Omega)\cap H^{2}(\Omega);$
\item[(2)]$\big\| \Pi^{\rm rQ}_h v - v \big\|_{0,\Omega} + 
                 h\big|\Pi^{\rm rQ}_h v - v\big|_{1,h}
          \lesssim h^{2}|v|_{2,\Omega}$, $\forall v \in H_{0}^{1}(\Omega)\cap H^{2}(\Omega)$.
\end{itemize}
\end{lemma}

\begin{lemma}{\rm (\cite[Lemma 7]{Shuo.Zhang2017})} \label{lem:relationBLPS}The following relationships hold.
$$
\Pi^{\rm rQ}_h V^{\rm BL}_{h}=V^{\rm PS}_h \ \mbox{ and }\ \undertilde{\Pi}{}^{\rm rQ}_h\uV{}^{\rm BL}_{h\bf n}=\uV{}^{\rm PS}_{h\bf n}.
$$
\end{lemma}

\subsection{$H^{1}$ elliptic problems and nonconforming finite element approximation}
In this paper, we consider the following model problems:

{$\bullet$} Source problem: with $f\in Q:= L^2(\Omega)$, $\rho \in L^{\infty}(\Omega)$, and $\rho \geqslant c_0>0$,
\begin{equation}\label{eq:Poisson} 
\left\{
\begin{split}
-\Delta u &=  \rho f & \text{in} &\; \Omega, \\
u & = 0 & \text{on} & \; \partial\Omega.
\end{split}
\right.
\end{equation}

\noindent Its weak form is given by: Find $u \in V:=H_{0}^{1}(\Omega)$ satisfying
 \begin{equation}\label{eq:varP}
 a(u,v) = b(f,v), \quad \forall v\in V,
 \end{equation}
where $ a(u,v) = \int_{\Omega} \nabla u \cdot \nabla v \ud \emph{xdy}$ and $b(f,v) = \int_{\Omega} \rho f v 
\ud \emph{xdy}.$

{$\bullet$}  Eigenvalue problem: with
 $\rho \in L^{\infty}(\Omega)$ and $\rho \geqslant c_0>0$,
\begin{equation}\label{eq:eigen}
\left\{
\begin{split}
-\Delta u &= \lambda \rho u & \text{in} &\; \Omega, \\
u & = 0 & \text{on} & \; \partial\Omega.
\end{split}
\right.
\end{equation}

\noindent Its weak form is given by: Find $(\lambda, u)\in \mathbb{R}\times V$ with $\|u\|_{0,\rho} = 1$, such that
 \begin{equation}\label{eq:varEigen}
 a(u,v) = \lambda b(u,v), \quad \forall v\in V,
 \end{equation}
where $\|v\|_{0,\rho} := b(v,v)^{\frac{1}{2}}$ defines a norm over $V$ equivalent to the usually $L^{2}$ norm.

From \cite{Babuska1991}, the eigenvalue problem \eqref{eq:eigen} has a sequence of eigenvalues 
\begin{align*}
0 < \lambda_{1} \leqslant \lambda_{2}  \leqslant \cdots \leqslant \lambda_{k}  \leqslant \cdots, \text{ satisfying }  \lim_{k\to \infty} \lambda_{k} = \infty, 
\end{align*}
and corresponding eigenfunctions
\begin{align*}
u_{1}, \ u_{2},\ \cdots, \ u_{k},\cdots,  \text{ satisfying }  b(u_{i},u_{j})=\delta_{ij} .
\end{align*}
For a certain eigenvalue $\lambda_{j}$ of \eqref{eq:varEigen}, we define
$$
M(\lambda_{j}) = 􏰛\{w \in V : w \text{ is an eigenfunction of }\eqref{eq:varEigen} \text{ corresponding to }\lambda_{j}\}.
$$

Given an discrete space $V_{h}$ defined on $\mathcal{G}_{h}$, the discretization schemes are

{$\bullet$} for the source problem: Find $u_{h} \in V_{h}$, such that
\begin{align}\label{eq:dvarP}
 a_{h}(u_{h},v_{h}) =  b(f,v_{h}), \quad \forall v_{h}\in V_{h},
\end{align}

{$\bullet$} for the eigenvalue problem:
Find $(\lambda_{h},u_{h}) \in  \mathbb{R} \times V_{h} $ with $||u_{h}||_{0,\rho} = 1$, such that
\begin{align}\label{eq:dEigen}
 a_{h}(u_{h},v_{h}) =  \lambda_{h}b(u_{h},v_{h}), \quad \forall v_{h}\in V_{h}.
\end{align}

 Let $\text{dim}V_{h}=N$. The discrete eigenvalue problem \eqref{eq:dEigen} has a sequence of eigenvalues
\begin{equation*}
0 < \lambda_{1,h} \leqslant \lambda_{2,h}  \leqslant \cdots \leqslant \lambda_{N,h},
\end{equation*}
and corresponding eigenfunctions
\begin{equation*}
u_{1,h}, \ u_{2,h},\ \cdots, \ u_{N,h}, \text{ satisfying }  b(u_{i,h},u_{j,h})=\delta_{ij} .
\end{equation*}
\begin{lemma}
{\rm (\cite[Theorem 4.1.7]{C.Chen2001}) }\label{lem:EGregularity}
Suppose that $\Omega$ is a rectangular region and $\rho$ is smoothing enough. If $(\lambda_{j},u_{j})$ is an eigen-pair of \eqref{eq:eigen}, then $u_{j}\in C^{5,\alpha}(\Omega)$.
\end{lemma}

\section{The rectangular Morley (RM) element revisited}
\label{sec:rmrev}

\subsection{The RM element space}The RM element is defined by 
 $(K,P_{K}^{\rm{M}},D_{K}^{\rm{M}})$ with the following properties:
\begin{itemize}
\item[(1)] $K$ is a rectangle;
\item[(2)] $P_{K}^{\rm{M}} = P_{2}(K) + \text{span}\{x^{3},y^{3}\}$;
\item[(3)] for any $v\in H^{2}(K)$, $D_{K}^{\rm{M}} =\big\{ v(a_{i}), \fint_{e_{i}}\partial_{n_{e_{i}}}v \ud s \big\}_{i=1:4}$.
\end{itemize}
Define the RM element space as
\begin{equation*}
\begin{split}
V_{h}^{\rm{M}} := \Big\{w_{h}\in L^{2}(\Omega) : w_{h}|_{K} \in P_{K}^{\rm{M}}, \ & w_{h}(a)\mbox{ is continuous at any } a\in \mathcal{N}_{h}^{i}, \\
& \mbox{and} \fint_{e}\partial_{n_{e}} w_{h} \ud s \mbox{ is continuous across any } e \in \mathcal{E}_{h}^{i} \Big\}.
\end{split}
\end{equation*}
Associated with $H_{0}^{1}(\Omega)$, we define $V_{hs}^{\rm{M}} :=\Big\{w_{h}\in V_{h}^{\rm{M}} :  w_{h}(a)=0, \forall a\in \mathcal{N}_{h}^{b} \Big\}.$ 

\begin{lemma}{\rm(\cite[Lemmas 3.2 and 3.5]{XY.Meng;XQ.Yang;S.Zhang2016})}\label{lem:consisRM} Denote
$E_{h}(w,v_{h}) := a_{h}(w,v_{h}) + (\Delta w,v_{h})$ with $w\in V \text{ and } v_{h}\in V_{h}.$ The following estimates hold.
\begin{itemize}
\item[(a)] For any shape-regular rectangular grid, it holds for any $v_{h} \in V_{hs}^{\rm{M}}$ that
\begin{align*}
| E_{h}(v, v_{h}) | \lesssim \sum_{K\in \mathcal{G}_{h}} h_{K}^{2} |v|_{2,K} |v_{h}|_{2,K}\lesssim h |v|_{2,\Omega} |v_{h}|_{1,h}, \quad \forall v\in H^{2}(\Omega)\cap H_{0}^{1}(\Omega).
\end{align*}

\item[(b)] For any uniform rectangular grid, it holds for any $v_{h} \in V_{hs}^{\rm{M}}$ that
\begin{align*}
| E_{h}(v, v_{h}) | \lesssim h^{k-1} |v|_{k,\Omega} |v_{h}|_{1,h}, \quad \forall v\in H^{k}(\Omega)\cap H_{0}^{1}(\Omega), \quad k = 2,3.
\end{align*}
\end{itemize}
\end{lemma}

For the RM element, there is a refined property of the interpolation operator $\Pi_{h}^{\rm{M}}: V \mapsto V_{hs}^{\rm{M}}$.
\begin{lemma}{\rm (\cite[Lemma 3.17]{XY.Meng;XQ.Yang;S.Zhang2016})}
\label{lem:propertyPiRM}
Assume that $\mathcal{G}_{h}$ is uniform. For any $w \in H_{0}^{1}(\Omega)\cap H^{3}(\Omega)$ with $\big\| \frac{\partial^{2}w}{\partial x\partial y } \big\|_{0,\rho} \ne 0$,  if $h$ is small enough, then 
\begin{align}\label{eq:propertyPiRM}
a_{h}(w-\Pi_{h}^{\rm{M}}w,\Pi_{h}^{\rm{M}}w) \geqslant \alpha h^{2},
\end{align}
where $\alpha>0$ is a constant independent of $h$.
\end{lemma}

\begin{corollary}\label{cor:expansion2}
{\rm Under the conditions in Lemma \ref{lem:propertyPiRM}, there exists $\alpha_{1}>0$, such that
\begin{align}\label{eq:property1PiRM}
a_{h}(w-\Pi_{h}^{\rm{M}}w,w)\geqslant \alpha_{1} h^{2}.\end{align}
}
\end{corollary}

\begin{proof}
It follows from Lemma \ref{lem:propertyPiRM} and $\big|w-\Pi_{h}^{\rm{M}}w\big|_{1,h} \lesssim h^{2}$ immediately.
\end{proof}
\noindent Hence we obtain an interesting and intuitive conclusion:
\begin{equation}
a_{h}(u-\Pi_{h}^{\rm M}u,u) \cequiv \big|u-\Pi_{h}^{\rm M}u\big|_{1,h}  \big|u\big|_{1,h}, \text{ when } h \text{ is small enough}.
\end{equation} 

By standard argument, we can prove the exact sequence which reads
\begin{equation}
\begin{array}{ccccccccc}
0 & \longrightarrow & V_{hs}^{\rm{M}} & \xrightarrow{\curl_h} & \uV{}_{h\bf n}^{\rm LTZ} & \xrightarrow{\dv_h} & \mathcal{L}_h^{1,-1}  & \longrightarrow & 0,
\end{array}
\end{equation}
where $\mathcal{L}_h^{1,-1}=\{q\in L^2_0(\Omega):q|_K\in P_1(K)\}$.

\subsection{The RM element scheme for the $H^{1}$ eigenvalue problem}

\subsubsection{Expanded representation of the difference between energy of states}\label{sub:commonExpan}
Simple calculations yield
\begin{equation}\label{eq:xpyh}
\begin{split}
a(\upsilon_1,\upsilon_1) + a_{h}(\upsilon_2,\upsilon_2) &= 2a_{h}(\upsilon_1,\upsilon_2) + a_{h}(\upsilon_1-\upsilon_2,\upsilon_1-\upsilon_2)\\
& = 2a_{h}(\upsilon_3,\upsilon_2) + 2a_{h}(\upsilon_1-\upsilon_3,\upsilon_3)\\
& \quad + 2a_{h}(\upsilon_1-\upsilon_3,\upsilon_2-\upsilon_3) + a_{h}(\upsilon_1-\upsilon_2,\upsilon_1-\upsilon_2),
\end{split}
\end{equation} 

\begin{equation}{\label{eq:yhmx}}
\begin{split}
a_{h}(\upsilon_2,\upsilon_2) - a(\upsilon_1,\upsilon_1)  & = \left[2a_{h}(\upsilon_3,\upsilon_2) - 2a(\upsilon_1,\upsilon_1)\right] + 2a_{h}(\upsilon_1-\upsilon_3,\upsilon_3)\\
& \quad + 2a_{h}(\upsilon_1-\upsilon_3,\upsilon_2-\upsilon_3) + a_{h}(\upsilon_1-\upsilon_2,\upsilon_1-\upsilon_2),
\end{split}
\end{equation} 

\begin{equation}\label{eq:xmyh}
\begin{split}
a(\upsilon_1,\upsilon_1) - a_{h}(\upsilon_2,\upsilon_2)  & = 2a_{h}(\upsilon_3-\upsilon_2,\upsilon_2) + 2a_{h}(\upsilon_1-\upsilon_3,\upsilon_3)\\
& \quad + 2a_{h}(\upsilon_1-\upsilon_3,\upsilon_2-\upsilon_3) + a_{h}(\upsilon_1-\upsilon_2,\upsilon_1-\upsilon_2).
\end{split}
\end{equation} 
Let $u$ be the solution of the source problem \eqref{eq:varP} or the eigenvalue problem \eqref{eq:varEigen} and $u_{h}$ be its approximation.  Let $\upsilon_1 =u $, $\upsilon_2 = u_{h}$ and $\upsilon_3= \Pi_{h}u$, where $\Pi_{h}: V\mapsto V_{h0}$ is an interpolation operator. We use \eqref{eq:yhmx} to obtain an expansion of $b(-f, u-u_{h})$ and \eqref{eq:xmyh} to obtain an expansion of $\lambda - \lambda_{h}$.

%\begin{enumerate}
%\item[(a)] For the source problem: 
{$\bullet$} For the source problem:

\noindent Let $u$ and $u_{h}$ be the solutions of \eqref{eq:varP} and \eqref{eq:dvarP}, respectively. It holds that 
$$b(-f,u-u_{h}) = a_{h}(u_{h},u_{h}) - a(u,u).$$ 
From 
$\left[2a_{h}(\Pi_{h}u,u_{h}) - 2a(u,u)\right] = 2b(f,\Pi_{h}u - u),$ the formula \eqref{eq:yhmx} becomes
\begin{equation}\label{eq:uhmu}
\begin{split}
a_{h}(u_{h},u_{h}) - a(u,u)  & = 2b(f,\Pi_{h}u - u) + 2a_{h}(u-\Pi_{h}u,\Pi_{h}u)\\
& \quad + 2a_{h}(u-\Pi_{h}u,u_{h}-\Pi_{h}u) + a_{h}(u-u_{h},u-u_{h}),
\end{split}
\end{equation} 
Analyze the items on the right-hand-side of \eqref{eq:uhmu}. Suppose that $u \in H^{3}(\Omega)\cap H_{0}^{1}(\Omega)$. With the second term $2a_{h}(u-\Pi_{h}u,\Pi_{h}u)$ not considered, the rest items in \eqref{eq:uhmu} are of high order than $|u-u_{h}|_{1,h}$. Therefore, $2a_{h}(u-\Pi_{h}u,\Pi_{h}u)$ becomes the dominant factor to determine whether $b(-f, u-u_{h})$ is of higher order than $|u-u_{h}|_{1,h}$.
 
%\item[(b)] For the eigenvalue problem: 
{$\bullet$} For the eigenvalue problem:

\noindent Let $u$ and $u_{h}$ be the solutions of \eqref{eq:varEigen} and \eqref{eq:dEigen}, which satisfy $b(u,u) = b(u_{h},u_{h})=1$. 
It holds that 
$$\lambda - \lambda_{h} = a(u,u) - a_{h}(u_{h},u_{h}).$$

\noindent Notice that $2b(u_{h}, u- u_{h}) =2b(u_{h}, u)- b(u_{h},u_{h}) -b(u,u)= - b(u-u_{h},u-u_{h})$. Thus we can obtain
\begin{equation*}
\begin{split}
2a_{h}(\Pi_{h}u-u_{h},u_{h}) & = 2\lambda_{h}b(u_{h}, \Pi_{h}u-u_{h}) 
 =  2\lambda_{h}b(u_{h}, \Pi_{h}u-u) + 2\lambda_{h}b(u_{h}, u- u_{h})\\
& = 2\lambda_{h}b(u_{h}, \Pi_{h}u-u) - \lambda_{h}b(u-u_{h}, u- u_{h}).
\end{split}
\end{equation*}
Based on these above, \eqref{eq:xmyh} becomes 
\begin{equation}\label{eq:umuh}
\begin{split}
a(u,u) - a_{h}(u_{h},u_{h})  & = 2\lambda_{h}b(u_{h}, \Pi_{h}u-u) - \lambda_{h}b(u-u_{h}, u- u_{h}) + 2a_{h}(u-\Pi_{h}u,\Pi_{h}u)\\
& \quad + 2a_{h}(u-\Pi_{h}u,u_{h}-\Pi_{h}u) + a_{h}(u-u_{h},u-u_{h}).
\end{split}
\end{equation} 
Analyze the items on the right-hand-side of \eqref{eq:umuh}. Similarly, $2a_{h}(u-\Pi_{h}u,\Pi_{h}u)$ is also the dominant factor to determine whether $\lambda-\lambda_{h}$ is of higher order than $|u-u_{h}|_{1,h}$.
%\end{enumerate}

%
 \subsubsection{Analysis of the RM element for the eigenvalue problem.}
Based on the error estimates of the rectangular Morley element scheme for the source problem (see \cite{XY.Meng;XQ.Yang;S.Zhang2016}), the following estimates for the eigenvalue problem follows by standard argument.
\begin{theorem}\label{thm:EerrorRM}
Let $\lambda_{j}$ be the $j$-th eigenvalue of \eqref{eq:varEigen}, and $(\lambda_{j,h}^{\rm{M}},u_{j,h}^{\rm{M}}) \in \mathbb{R} \times V_{hs}^{\rm{M}}$ be the $j$-th eigen-pair of \eqref{eq:dEigen} with $\|u_{j,h}^{\rm{M}}\|_{0,\rho} =1$. It holds that
\begin{itemize}
\item[(a)] if $M(\lambda_{j}) \subset H_{0}^{1}(\Omega) \cap H^{2}(\Omega)$, then there exists $u_{j} \in M(\lambda_{j})$ with $\|u_{j}\|_{0,\rho} = 1$, such that
\begin{align*}
\big|\lambda_{j}-\lambda_{j,h}^{\rm{M}}\big| \lesssim h^{2}, \quad
\big\|u_{j}-u_{j,h}^{\rm{M}}\big\|_{0,\rho}\lesssim h^{2}, \mbox{ and }
\big|u_{j}-u_{j,h}^{\rm{M}}\big|_{1,h}\lesssim h ;
\end{align*}
\item[(b)] if the mesh is uniform and $M(\lambda_{j}) \subset H_{0}^{1}(\Omega) \cap H^{3}(\Omega)$, then there exists $u_{j} \in M(\lambda_{j})$ with $\| u_{j}  \|_{0,\rho} = 1$, such that
$
\big|u_{j}-u_{j,h}^{\rm{M}}\big|_{1,h} \lesssim h^{2} .
$
\end{itemize}
\end{theorem}

\noindent Moreover, we obtain the lower-bound property of eigenvalue approximations by the RM element.
\begin{theorem}\label{thm:lowerM}
Let $\lambda_{j}$ and $ \lambda_{j,h}^{\rm{M}} $ be an exact eigenvalue and its approximation by the RM element. Suppose that $u_{j}\in H_{0}^{1}(\Omega)\cap H^{3}(\Omega)$ and the mesh is uniform. When $h$ is small enough, we have
\begin{align}\label{eq:lowerM}
\lambda_{j} - \lambda_{j,h}^{\rm{M}} \geqslant C_{\rm M}h^{2}. 
\end{align}
where $C_{\rm M}$ is a positive constant independent of $h$.
\end{theorem}

\begin{proof}
We have the basic expansion by \cite{Y.Yang;Z.Zhang2010,Y.Yang;J.Han;H.Bi;Y.Yu2015}, which generalizes the identity introduced by \cite{M.Armentano;R.Duran2004},
\begin{equation}\label{eq:eigenExpanM}
\begin{split}
\lambda_{j} - \lambda_{j,h}^{\rm{M}} =& \big|u_{j}-u_{j,h}^{\rm{M}} \big|^{2}_{1,h} - \lambda_{j,h}^{\rm{M}}\big\|u_{j}-u_{j,h}^{\rm{M}}\big\|^{2}_{0,\rho} - 2\lambda_{j,h}^{\rm{M}}b(u_{j} - \Pi_{h}^{\rm{M}}u_{j},u_{j,h}^{\rm{M}}) \\
&+ 2a_{h}(u_{j} - \Pi_{h}^{\rm{M}}u_{j},\Pi_{h}^{\rm{M}}u_{j}) + 2a_{h}(u_{j} - \Pi_{h}^{\rm{M}}u_{j},u_{j,h}^{\rm{M}}-\Pi_{h}^{\rm{M}}u_{j}).
\end{split}
\end{equation}
From Theorem \ref{thm:EerrorRM}, the first two terms  can be bounded as 
$$
\big\|u_{j}-u_{j,h}^{\rm{M}}\big\|_{0,\rho}^{2} \lesssim \big|u_{j}-u_{j,h}^{\rm{M}}\big|_{1,h}^{2} \lesssim h^{4}.
$$ 
From a standard interpolation theory in \cite{Wang.M;Shi.Z2013mono}, the assumption $\big\|u_{j,h}^{\rm{M}}\big\|_{0,\rho} = 1$, and Theorem \ref{thm:EerrorRM} (b), the third and last terms have the  estimates below
\begin{align*}
b(u_{j}-\Pi_{h}^{\rm{M}}u_{j},u_{j,h}^{\rm{M}}) & \lesssim \big\|u_{j}-\Pi_{h}^{\rm{M}}u_{j}\big\|_{0,\rho} \big\|u_{j,h}^{\rm{M}}\big\|_{0,\rho} \lesssim h^{3},
\\
 a_{h}(u_{j}-\Pi_{h}^{\rm{M}}u_{j} , u_{j,h}^{\rm{M}}-\Pi_{h}^{\rm{M}}u_{j})
 & \lesssim \big|u_{j}-\Pi_{h}^{\rm{M}}u_{j}\big|_{1,h}\big|u_{j,h}^{\rm{M}} - \Pi_{h}^{\rm{M}}u_{j}\big|_{1,h} \lesssim h^{4}.
 \end{align*}
 
\noindent When $h$ is small enough, it follows from Lemma \ref{lem:propertyPiRM} that
$$
a_{h}(u_{j}-\Pi_{h}^{\rm{M}}u_{j},\Pi_{h}^{\rm{M}}u_{j})\geqslant \alpha h^{2}.
$$ 
Thus,  $a_{h}(u_{j}-\Pi_{h}^{\rm{M}}u_{j},\Pi_{h}^{\rm{M}}u_{j})$ becomes the dominant term on the right-hand-side of \eqref{eq:eigenExpanM}. Hence the result.
\end{proof}

\section{Reduced rectangular Morley element space for $H^1$ problems}
\label{sec:rrmscheme}

\subsection{Reduced rectangular Morley element space}
We introduce an reduced rectangular Morley (RRM) element space by
\begin{equation}
\begin{split}
V_{h}^{\rm{R}} := \Big\{w_{h}\in L^{2}(\Omega) : w_{h}|_{K} \in P_{2}(K), & \; w_{h}(a)\mbox{ is continuous at any } a\in \mathcal{N}_{h}^{i}, \\
& \mbox{and }\fint_{e}\partial_{n_{e}} w_{h} \ud s \mbox{ is continuous across any } e \in \mathcal{E}_{h}^{i}\Big\},
\end{split}
\end{equation}
and, associated with $H_{0}^{1}(\Omega)$, define
\begin{equation}
V_{hs}^{\rm{R}} :=\big\{w_{h}\in V_{h}^{\rm{R}} :  w_{h}(a)=0,\forall a\in \mathcal{N}_{h}^{b} \big\}.
\end{equation}

\begin{theorem}\label{thm:basisofRRM}
If $\mathcal{G}_{h}$ is a $m\times n$ rectangular subdivision of $\Omega$, then ${\rm dim} (V_{hs}^{\rm {R}} )= mn+1$.
\end{theorem}
\noindent Detailed proof of Theorems \ref{thm:basisofRRM} and an available set of basis functions of $V_{hs}^{\rm {R}}$ are put in Appendix~\ref{sec:appB}. 

For any function $v_{h}$ in the RRM element space, the number of continuity restrictions across internal edges is greater than ${\rm dim}\big(P_{2}(K)\big)$, which makes it a nontrivial task to find out a set of basis functions of $V_{hs}^{\rm {R}}$, and it is not even easy to tell if the space contains non-zero functions. Actually, the proof of Theorem \ref{thm:basisofRRM} in Appendix \ref{sec:appB} ensures that the RRM element space is non-zero. From the analysis therein, the supports of the basis functions in $V_{hs}^{\rm {R}}$ are not completely local, making it complicated to construct an interpolation operator from $V$ to $V_{hs}^{\rm {R}}$, which, however, plays a fundamental role in the approximation error analysis.

\begin{remark}\rm{
Since a non-convex domain which can be covered by a rectangular subdivision can be considered as a combination of several rectangular regions, a nontrivial RRM element space can still be expected on it.}
\end{remark}

\subsection{Approximation property of the RRM element space} 
The main result of this subsection is the theorem below. 
\begin{theorem}\label{thm:approRRM}
Given $u\in H^1_0(\Omega)\cap H^3(\Omega)$, we have
$$
\inf_{v_h\in V_{hs}^{\rm R}}|u-v_h|_{1,h}\lesssim h^{\alpha}|u|_{1+\alpha,\Omega},\ \ \alpha=1,2.
$$
\end{theorem}
We postpone the proof of Theorem \ref{thm:approRRM} after some technical lemmas. Let $\uf\in \big(\uH{}^1_{\bf n}(\Omega)\big)'$. We firstly consider the regularity of the Stokes problem: Find $(\uu,p)\in \uH{}^1_{\bf n}(\Omega)\times L^2_0(\Omega)$, such that 
\begin{equation}\label{eq:stokes}
\left\{
\begin{array}{lll}
(\nabla\uu,\nabla \uv)+(p,\dv\uv)&=(\uf,\uv), &\forall\,\uv\in \uH{}^1_{\bf n}(\Omega),
\\
(q,\dv\uu)&=0,&\forall\,q\in L^2_0(\Omega).
\end{array}
\right.
\end{equation}

\begin{lemma}\label{lem:regstokes}
Let $\Omega$ be a rectangle. If $\uf\in \uL^2(\Omega)$, then $(\uu,p)\in \uH^2(\Omega)\times H^1(\Omega)$.
\end{lemma}
\begin{proof}
As $\dv\uu=0$, there exists a unique $\varphi\in H^2(\Omega)\cap H^1_0(\Omega)$, such that $\curl\varphi=\uu$. Moreover, $\varphi$ solves the biharmonic equation:
\begin{align*}
(\nabla\curl\varphi,\nabla\curl\psi)=(\uf,\curl\psi),\quad\forall\,\psi\in H^2(\Omega)\cap H^1_0(\Omega).
\end{align*}
By the regularity theory of the biharmonic equation (see \cite[Theorem 2 ]{Blum.H;Rannacher.R1980}),  we have $\varphi\in H^3(\Omega)$, and $\|\varphi\|_{3,\Omega} \lesssim \sup\limits_{\psi\in H^1_0(\Omega)\setminus\{0\}} {\dfrac{( \uf,\curl \psi)}{\|\psi\|_{1,\Omega}} }\lesssim \|\uf\|_{0,\Omega}$. Furthermore, $\nabla p=\uf+\Delta\uu$, and $\|p\|_{1,\Omega}\lesssim \|\uf+\Delta\uu\|_{0,\Omega}\lesssim \|\uf\|_{0,\Omega}$.  The proof is completed.
\end{proof}

A related finite element problem is to find $(\uu{}_h,q_h)\in \uV{}^{\rm BL}_{h\bf n}\times\mathcal{L}^0_{h0}$, such that 
\begin{equation}\label{eq:stokesq1p0}
\left\{
\begin{array}{ll}
(\nabla\uu{}_h,\nabla\uv{}_h)+(p_h,\dv\uv{}_h)&=(\uf,\uv{}_h)
\\
(q_h,\dv\uu{}_h)&=0.
\end{array}
\right.
\end{equation}
To ensure the convergence of the finite element scheme in Theorem \ref{thm:stokes}, we need the following hypothesis:
\paragraph{\bf Hypothesis RT} A rectangular grid $\mathcal{G}_h$ is called to satisfy the hypothesis {\bf Hypothesis RT} if and only if it is generated by refining a grid $\mathcal{G}_{4h}$ twice. 

\begin{theorem}\label{thm:stokes}
Let $\mathcal{G}_h$ be a grid that satisfies {\bf Hypothesis RT}. Let $(\uu,p)$ and $(\uu{}_h,p_h)$ be the solutions of \eqref{eq:stokes} and \eqref{eq:stokesq1p0}, respectively. If $(\uu,p)\in \uH^2(\Omega)\times H^1(\Omega)$, then
\begin{align}
|\uu-\uu{}_h|_{1,\Omega}\lesssim h(|u|_{2,\Omega}+|p|_{1,\Omega}),
\end{align}
and further
\begin{align}
\|\uu-\uu{}_h\|_{0,\Omega}\lesssim h^2(|u|_{2,\Omega}+|p|_{1,\Omega}).
\end{align}
\end{theorem}
Based on Lemma \ref{lem:regstokes}, the proof of Theorem \ref{thm:stokes} is just a duplication of the proofs of \cite[Theorems~3.4--3.5 and Corollary~3.2]{GiraultRaviart1986}, and we omit the details here. 

\begin{theorem}\label{thm:approdf}
Let $\mathcal{G}_h$ be a grid that satisfies {\bf Hypothesis RT}. Given $\uw\in \uH{}^1_{\bf n}(\Omega)\cap \uH^2(\Omega)$ satisfying $\dv\uw=0$. It holds that 
\begin{align}
\inf_{\uw{}_h\in{\undertilde{V}}{}^{\rm PS}_{h\bf n},\ \dv_h\uw{}_h=0}h|\uw-\uw{}_h|_{1,\Omega}+\|\uw-\uw{}_h\|_{0,\Omega}\lesssim h^2|\uw|_{2,\Omega}.
\end{align}
\end{theorem}
\begin{proof}
Let $\uu$ be the exact velocity of~\eqref{eq:stokes}. Denote $p_{0}\equiv0$. It can be verified directly that the pair $(\uu,p_{0}\equiv0)\in \uH{}^1_{\bf n}(\Omega)\times L^2_0(\Omega)$ solves the equation
\begin{equation}
\left\{
\begin{array}{ll}
(\nabla\uu,\nabla\uv)+(p_{0},\dv\uv)&=(-\Delta \uu,\uv)
\\
(q,\dv\uu)&=0,
\end{array}
\right.
\end{equation}
Let $(\uy{}_h,p_{0h})\in \uV{}^{\rm BL}_{h\bf n}\times\mathcal{L}^0_{h0}$ solve
\begin{equation}
\left\{
\begin{array}{ll}
(\nabla\uy{}_h,\nabla\uv{}_h)+(p_{0h},\dv\uv{}_h)&=(-\Delta \uu,\uv{}_h)
\\
(q_h,\dv\uy{}_h)&=0,
\end{array}
\right.
\end{equation}
then 
$
h|\uw-\uy{}_h|_{1,\Omega}+\|\uw-\uy{}_h\|_{0,\Omega}\leqslant Ch^2|\uu|_{2,\Omega}.
$
Set $\uw{}_h:=\undertilde{\Pi}{}^{\rm rQ}_{h}\uy{}_h$, then, from Lemma \ref{lem:relationBLPS}, $\uw{}_h\in \uV{}^{\rm PS}_{h\bf n}$. Moreover, it is easy to verify that
$$
(\dv_h\uw{}_h,q_h)=(\dv\uy{}_h,q_h)=0,\quad \forall\,q_h\in \mathcal{L}^0_{h0},
$$
namely $\dv_h\uw{}_h=0$. Furthermore,
$$
|\uw-\uw{}_h|_{1,h}\leqslant \big|\uw-\undertilde{\Pi}{}^{\rm rQ}_{h}\uw\big|_{1,\Omega} + \big|\undertilde{\Pi}{}^{\rm rQ}_{h}(\uw-\uy{}_h)\big|_{1,h} \lesssim h |\uw|_{2,\Omega},
$$
and
$$
\|\uw-\uw{}_h\|_{0,\Omega}=\big\|\uw-\uy{}_h+\big(\undertilde{\rm Id}-\undertilde{\Pi}{}^{\rm rQ}_{h}\big)(\uy{}_h-\uw)+\big(\undertilde{\rm Id}-\undertilde{\Pi}{}^{\rm rQ}_{h}\big)\uw\big\|_{0,\Omega}\lesssim h^2|\uw|_{2,\Omega}.
$$
Hence the result.
\end{proof}

\begin{lemma}\label{lem:strdf} 
For the $\curl$ of space $V_{hs}^{\rm {R}}$, it can be depicted as a special subspace of the vector Park--Sheen element space, i.e.,
$$\curl_h V_{hs}^{\rm {R}}=\Big\{\uz{}_h\in \uV{}^{\rm PS}_{h\bf n}:\dv_h\uz{}_h=0\Big\}.$$
\end{lemma}
\begin{proof}
Firstly, by standard argument, we can prove the exact sequence which reads
\begin{equation}
\begin{array}{ccccccccc}
0 & \longrightarrow & V_{hs}^{\rm{M}} & \xrightarrow{\curl_h} & \uV{}_{h\bf n}^{\rm LTZ} & \xrightarrow{\dv_h} & \mathcal{L}_h^{1,-1}  & \longrightarrow & 0,
\end{array}
\end{equation}
where $\mathcal{L}_h^{1,-1}=\{q\in L^2_0(\Omega):q|_K\in P_1(K)\}$. This way, given $\uv{}_h\in\uV{}^{\rm PS}_{h\mathbf n}\subset \uV{}_{h\bf n}^{\rm LTZ}$ with $\dv_h\uv{}_h=0$, there exists $w_h\in V_{hs}^{\rm{M}}$, such that $\curl_hw_h=\uv{}_h$. Since $\uv{}_h$ is piecewise linear polynomial, $w_h$ is piecewise quadratic, namely $w_h\in V_{hs}^{\rm {R}}$. On the other hand, it is evident that $\curl_hV_{hs}^{\rm {R}}\subset\big\{\uz{}_h\in \uV{}^{\rm PS}_{h\bf n}:\dv_h\uz{}_h=0\big\}$. Hence the result.
\end{proof}

\paragraph{\bf Proof of Theorem \ref{thm:approRRM}}
By Theorem \ref{thm:approdf} and Lemma \ref{lem:strdf},
\begin{multline}
\quad\inf_{v_h\in V_{hs}^{\rm{R}}}|u-v_h|_{1,h}=\inf_{v_h\in V_{hs}^{\rm{R}}}\|\curl u-\curl v_h\|_{0,\Omega}
\\
=\inf_{\uw{}_h\in\undertilde{V}{}^{\rm PS}_{h\bf n},\ \dv_h\uw{}_h=0}\|\curl u-\uw{}_h\|_{0,h}\lesssim h^\alpha|\curl u|_{\alpha,\Omega}\lesssim h^{\alpha}|u|_{1+\alpha,\Omega},\ \ \alpha=1,2.\quad
\end{multline}
The proof is completed. \qed

\section{Convergence analysis of the RRM element schemes}

\label{sec:analysisrrm}

\subsection{Optimal convergence for the source problem}
For the RM element space $V_{hs}^{\rm{M}}$ and the RRM element space $V_{hs}^{\rm{R}}$,  the discrete source problems are given as:

 Find $u_{h}^{\rm{M}} \in V_{hs}^{\rm{M}}$, such that
\begin{align}\label{eq:sourceRM}
 a_{h}(u_{h}^{\rm{M}},v_{h}) =  b(f,v_{h}), \quad \forall v_{h}\in V_{hs}^{\rm{M}};
\end{align}

 Find $u_{h}^{\rm{R}} \in V_{hs}^{\rm{R}}$, such that
\begin{align}\label{eq:sourceRRM}
 a_{h}(u_{h}^{\rm{R}},v_{h}) =  b(f,v_{h}), \quad \forall v_{h}\in V_{hs}^{\rm{R}}.
\end{align}

It is obvious that $V_{hs}^{\rm{R}} = V_{hs}^{\rm{M}} \cap  V_{h0}^{\rm{W}}$, and we infer that the RRM element is a quadratic nonconforming element on rectangles with a second-order  convergence rate in the energy norm. We will verify this assertion strictly in this section.

\begin{theorem}\label{thm:errorRRM} Let $\mathcal{G}_h$ be a grid that satisfies {\bf Hypothesis RT}. Let $u$ and $u_{h}^{\rm{R}}$ be the solutions of \eqref{eq:varP} and \eqref{eq:sourceRRM}, respectively. It holds that
\begin{itemize}
\item[(a)] 
if $u \in H_{0}^{1}(\Omega) \cap H^{2}(\Omega)$, then $\big|u-u_{h}^{\rm{R}}\big|_{1,h} \lesssim h
|u|_{2,\Omega} \; \text{and} \; \big\|u-u_{h}^{\rm{R}}\big\|_{0,\rho} \lesssim h^{2} |u|_{2,\Omega} ;$
\item[(b)]
 if $u \in H_{0}^{1}(\Omega)\cap H^{3}(\Omega)$ and the mesh is uniform, then $\big|u-u_{h}^{\rm{R}}\big|_{1,h} \lesssim h^{2}|u|_{3,\Omega}.$
\end{itemize}
\end{theorem}

\begin{proof}
{(a)} By the Strang lemma, we have
\begin{align}\label{eq:strangRRM}
\big|u-u_{h}^{\rm R}\big|_{1,h} \lesssim \inf_{v_{h}\in V_{hs}^{\rm{R}}} |u - v_{h}|_{1,h}+ \sup_{w_{h}\in V_{hs}^{\rm{R}},w_{h} \ne 0}  \frac{E_{h}(u,w_{h})}{|w_{h}|_{1,h}}.
\end{align}
For the first term in the right hand side of \eqref{eq:strangRRM}, we have from Theorem  \ref{thm:approRRM} that
\begin{align}\label{eq:approRRM}
\inf_{v_{h}\in V_{hs}^{\rm{R}}} |u - v_{h}|_{1,h} \lesssim h|u|_{2,\Omega}.
\end{align}
For the second term, we have from Lemma \ref{lem:consisRM} and $V_{hs}^{\rm{R}} \subset V_{hs}^{\rm{M}}$ that
\begin{align}\label{eq:consisRRM}
 |E_{h}(u,w_{h})|  \lesssim \sum_{K\in \mathcal{G}_{h}} h_{K}^{2} |u|_{2,K} |w_{h}|_{2,K}\lesssim h |u|_{2,\Omega}|w_{h}|_{1,h}, \quad \forall w_{h} \in V_{h}^{\rm{R}}.
\end{align}
 
\noindent Submit \eqref{eq:approRRM} and \eqref{eq:consisRRM} into \eqref{eq:strangRRM}, and we obtain 
$
\big|u-u_{h}^{\rm R}\big|_{1,h}\lesssim h |u|_{2,\Omega}, \mbox{ where } u \in H_{0}^{1}(\Omega) \cap H^{2}(\Omega).
$

Given $g \in Q$, let $\phi_{g} \in V$ and $\phi_{gh} \in V_{hs}^{\rm{R}}$ be the solutions of the two problems below, respectively, 
\begin{align*}
a(v,\phi_{g})   = b(g,v),\quad \forall v \in V; \quad 
a_{h}(v_{h},\phi_{gh})  = b(g,v_{h}), \quad \forall v_{h} \in V_{hs}^{\rm{R}}.
\end{align*}
By the Nitsche-Lascaux-Lesaint lemma (see e.g., \cite[Theorem 5.3.1]{Wang.M;Shi.Z2013mono}), it holds that
\begin{align}\label{eq:NitscheRRM}
\big\|u-u_{h}^{\rm R}\big\|_{0,\rho} \lesssim  \big|u-u_{h}^{\rm R}\big|_{1,h} \sup_{g\in Q, g\ne 0} \left\{ \frac{1}{||g||_{0,\rho}} | \phi_{g} - \phi_{gh} |_{1,h} \right \}
 + \sup_{g\in Q, g\ne 0}\left\{ \frac{1}{||g||_{0,\rho}} \Big(E_{h}(\phi_{g},u_{h}^{\rm R}) +E_{h}(u,\phi_{gh})\Big) \right\}.
\end{align}
For the first term in the right side of \eqref{eq:NitscheRRM}, we have
\begin{equation}\label{eq:term1}
\begin{split}
\big|u-u_{h}^{\rm R}\big|_{1,h} \sup_{g\in Q, g\ne 0}\left\{ \frac{1}{\|g\|_{0,\rho}} |\phi_{g} - \phi_{gh}|_{1,h}\right\} \lesssim h |u|_{2,\Omega} \frac{h | \phi_{g} |_{2,\Omega}}{\|g\|_{0,\rho}} 
\lesssim h^{2}|u|_{2,\Omega} \frac{ \| g \|_{0,\rho}}{\|g\|_{0,\rho}}
\lesssim h^{2}|u|_{2,\Omega},
\end{split}
\end{equation}
where we utilize the regularity of solution on a convex domain, namely, $|\phi_{g} | _{2,\Omega} \lesssim \|g\|_{0,\rho}$. 

\noindent For the second term, we notice that
\begin{align}\label{eq:E(phig,uR)}
 \big| E_{h}(\phi_{g} ,u_{h}^{\rm R}) \big| \lesssim \sum_{K\in \mathcal{G}_{h}} h_{K}^{2} |\phi_{g}|_{2,K} \big|u_{h}^{\rm R}\big|_{2,K}  \lesssim h |\phi_{g}|_{2,\Omega}\Big( \sum_{K\in \mathcal{G}_{h}} h_{K}^{2} \big|u_{h}^{\rm R}\big|_{2,K}^{2} \Big)^{1/2}.
\end{align}
 
\noindent Let $\Pi_{h0}^{\rm p} : L^{2}(\Omega) \mapsto V_{hs}^{\rm{M}}$ be an average projection operator defined in \cite{Wang.M;Shi.Z2013mono}. From \cite[Theorem~3.5.4]{Wang.M;Shi.Z2013mono}, we have $\big|u- \Pi_{h0}^{\rm p}u\big|_{1,K} \lesssim h_{K}|u|_{2,K}$ and $\big|\Pi_{h0}^{\rm p}u\big|_{2,K} \lesssim |u|_{2,K}$. Thus we obtain 
\begin{equation}\label{eq:|huR|_{2}}
\begin{split}
\sum_{K \in \mathcal{G}_{h}} h_{K}^{2} \big|u_{h}^{\rm R}\big|_{2,K}^{2} & \leqslant 2\sum_{K\in \mathcal{G}_{h}} h_{K}^{2}\Big( \big|u_{h}^{\rm R} - \Pi_{h0}^{\rm p}u\big|_{2,K}^{2} + |\Pi_{h0}^{\rm p}u|_{2,K}^{2} \Big) 
 \lesssim \sum_{K\in \mathcal{G}_{h}} \Big( \big|u_{h}^{\rm R} - \Pi_{h0}^{\rm p}u\big|_{1,K}^{2} + h_{K}^{2}|u|_{2,K}^{2} \Big)
\\
& \lesssim \sum_{K\in \mathcal{G}_{h}} \Big( \big|u - u_{h}^{\rm R}\big|_{1,K}^{2} + \big|u- \Pi_{h0}^{\rm p}u\big|_{1,K}^{2} + h_{K}^{2}|u|_{2,K}^{2} \Big)
 \lesssim \big|u - u_{h}^{\rm R}\big|_{1,h}^{2} + h^{2}|u|_{2,\Omega}^{2} \lesssim  h^{2}|u|_{2,\Omega}^{2}.
\end{split}
\end{equation}
The last inequality holds due to the fact that $\big|u-u_{h}^{\rm R}\big|_{1,h}\lesssim h |u|_{2,\Omega}$. Submitting \eqref{eq:|huR|_{2}} into \eqref{eq:E(phig,uR)}, it yields 
$ \big| E_{h}(\phi_{g} ,u_{h}^{\rm R}) \big| \lesssim h^{2}|\phi_{g}|_{2,\Omega} |u|_{2,\Omega} \lesssim h^{2}||g||_{0,\rho} |u|_{2,\Omega} $. Similarly, it holds that $ \big| E_{h}(u ,\phi_{gh}) \big| \lesssim h^{2}|u|_{2,\Omega} ||g||_{0,\rho}$. Thus we have
\begin{align}\label{eq:term2}
\sup_{g\in Q, g\ne 0} \left\{ \frac{1}{||g||_{0,\rho}} \Big(E_{h}(u,\phi_{gh}) + E_{h}(\phi_{g},u_{h}^{\rm R}) \Big)\right\} \lesssim h^{2}|u|_{2,\Omega}.
\end{align}
Submit \eqref{eq:term1} and \eqref{eq:term2} into \eqref{eq:NitscheRRM}, and we obtain
$
\big\|u-u_{h}^{\rm R}\big\|_{0,\rho} \lesssim h^{2} |u|_{2,\Omega}, \mbox{ where } u \in H_{0}^{1}(\Omega) \cap H^{2}(\Omega).
$

\noindent{(b)} If $u \in H_{0}^{1}(\Omega) \cap H^{3}(\Omega)$ and the mesh is uniform, then $
\inf\limits_{v_{h}\in V_{hs}^{\rm{R}}} |u - v_{h}|_{1,h} \lesssim h^{2}|u|_{3,\Omega}.
$
From Lemma ~ \ref{lem:consisRM},  
$|E_{h}(u,v_{h})| \lesssim h^{2} |u|_{3,\Omega}|v_{h}|_{1,h}, \forall v_{h} \in V_{h}^{\rm{R}}.$ 
Hence we have $|u-u_{h}|_{1,h}\lesssim h^{2} |u|_{3,\Omega}.$
\end{proof}

It shows that the error estimate in the energy norm can be $\mathcal{O}(h^{2})$ when the mesh is uniform. However, the convergence rate in $L^{2}$-norm can not be improved on uniform grids. Actually, there exists a lower bound stated in the following Theorem \ref{thm:lowerboundRRM}. 

\begin{theorem}\label{thm:lowerboundRRM}
Let $\mathcal{G}_{h}$ be a uniform grid satisfying {\bf Hypothesis RT}. Let $u$ and $u_{h}^{\rm{R}}$ be the solutions of \eqref{eq:varP} and \eqref{eq:sourceRRM}, respectively. If $\|f\|_{0,\rho} \ne 0 \text{ and } u \in H_{0}^{1}(\Omega)\cap H^{3}(\Omega) $, then $||u-u_{h}^{\rm R }||_{0,\rho} \gtrsim h^{2}$, provided that $h$ is small enough.
\end{theorem}
\begin{proof}
Let $u_{h}^{\rm{M}}$ be the solution of \eqref{eq:sourceRM}. Then 
\begin{align} \label{eq:fequalityRM}
b(-f, u-u_{h}^{\rm{M}}) \geqslant \delta h^{2},
\end{align}
given that $h$ is small enough, where $\delta>0$ is a constant independent of $h$; see \cite[Lemma 3.18]{XY.Meng;XQ.Yang;S.Zhang2016}. Noticing that $V_{hs}^{\rm{R}} \subset V_{hs}^{\rm{M}}$, we obtain 
\begin{equation}
a_{h}(u_{h}^{\rm{M}}-u_{h}^{\rm{R}}, v_{h}) = 0, \quad \forall v_{h}\in V_{hs}^{\rm{R}}.
\end{equation}
A simple division yields
\begin{align}\label{eq:fRRM1}
b(-f, u-u_{h}^{\rm{R}}) = b(-f, u-u_{h}^{\rm{M}}) + b(-f, u_{h}^{\rm{M}}-u_{h}^{\rm{R}}).
\end{align}
Owing to the orthogonality and $\big|u_{h}^{\rm{M}}-u_{h}^{\rm{R}}\big|_{1,h} \leqslant \big|u_{h}^{\rm{M}}-u\big|_{1,h} +\big|u-u_{h}^{\rm{R}}\big|_{1,h}  \lesssim h^2$, it holds that
\begin{equation}\label{eq:fRRM2}
\begin{split}
b(-f, u_{h}^{\rm{M}}-u_{h}^{\rm{R}})  = -a_{h}(u_{h}^{\rm{M}}, u_{h}^{\rm{M}}-u_{h}^{\rm{R}})
 = -a_{h}(u_{h}^{\rm{M}}-u_{h}^{\rm{R}}, u_{h}^{\rm{M}}-u_{h}^{\rm{R}})
 = - \big|u_{h}^{\rm{M}}-u_{h}^{\rm{R}}\big|_{1,h}^{2} \lesssim h^{4}.
\end{split}
\end{equation}
A combination of \eqref{eq:fequalityRM}, \eqref{eq:fRRM1} and \eqref{eq:fRRM2} leads to the following lower bound, provided that $h$ is small enough,
\begin{align*}
b(-f, u-u_{h}^{\rm{R}}) \geqslant \gamma h^{2},
\end{align*}

\noindent  where $\gamma > 0$ is a constant independent of  $h$. 
Therefore we have
\begin{align*}
\big\|u-u_{h}^{\rm{R}}\big\|_{0,\rho} = \sup_{g\in Q, g\ne 0} \frac{b(g,u-u_{h}^{\rm{R}})}{\|g\|_{0,\rho}} \geqslant \frac{b(-f,u-u_{h}^{\rm{R}})}{\|-f\|_{0,\rho}} \geqslant \frac{\gamma}{\|f\|_{0,\rho}}h^{2}.
\end{align*}
Hence the result.
\end{proof}

\begin{remark}{\rm\cite[Remark 3.16]{XY.Meng;XQ.Yang;S.Zhang2016}} 
{\rm For rectangle domain, the condition $\|f\|_{0,\rho} \ne 0$ implies that $\big\|\frac{\partial^{2}u}{\partial x \partial y}\big\|_{0,\rho} \ne 0$. In fact, if $\big\|\frac{\partial^{2}u}{\partial x \partial y}\big\|_{0,\rho} = 0$, then $u$ is of the form $u = h(x)+g(y)$, for some function $h(x)$ with respect to $x$ and $g(y)$ with respect to $y$. Then, the boundary condition, i.e., $u = 0$ on $\partial \Omega$, indicates $u  \equiv 0$, which contradicts  $\|f\|_{0,\rho} \ne 0$.}
\end{remark}

\subsection{Analysis of the scheme for the eigenvalue problem} 
With the associated spaces $V_{hs}^{\rm{M}}$ and $V_{hs}^{\rm{R}}$,  the discrete eigenvalue problems are given as:

 Find $(\lambda_{h}^{\rm{M}},u_{h}^{\rm{M}}) \in \mathbb{R} \times V_{hs}^{\rm{M}} $ with $\big\|u_{h}^{\rm{M}}\big\|_{0,\rho} = 1$, such that
\begin{align}\label{eq:eigenRM}
 a_{h}(u_{h}^{\rm{M}},v_{h}) =  \lambda_{h}^{\rm{M}} b(u_{h}^{\rm{M}},v_{h}), \quad \forall v_{h}\in V_{hs}^{\rm{M}};
\end{align}

Find $(\lambda_{h}^{\rm{R}}, u_{h}^{\rm{R}}) \in \mathbb{R} \times V_{hs}^{\rm{R}}$ with $\big\|u_{h}^{\rm{R}}\big\|_{0,\rho} = 1$, such that
\begin{align}\label{eq:eigenRRM}
 a_{h}(u_{h}^{\rm{R}},v_{h}) =  \lambda_{h}^{\rm{R}} b(u_{h}^{\rm{R}},v_{h}), \quad \forall v_{h}\in V_{hs}^{\rm{R}}.
\end{align}

From Theorem \ref{thm:EerrorRRM}, the convergence results of the eigenvalue problem is obtained by standard argument (see, e.g., \cite{J.Xu;A.Zhou2001,J.Hu;Y.Huang;Q.Lin2014,Y.Yang;Z.Zhang2010,YiDU.Yang;Zhen.Chen2008}).

\begin{theorem}\label{thm:EerrorRRM}
Let $\mathcal{G}_h$ be a grid satisfying {\bf Hypothesis RT}. Let $\lambda_{j}$ be the $j$-th eigenvalue of  \eqref{eq:varEigen}, and $(\lambda_{j,h}^{\rm{R}},u_{j,h}^{\rm{R}}) \in \mathbb{R} \times V_{hs}^{\rm{R}}$ be the $j$-th eigen-pair of  \eqref{eq:eigenRRM} with $\big\|u_{j,h}^{\rm{R}}\big\|_{0,\rho} =1$. It holds that

\begin{itemize}
\item[(a)] if $M(\lambda_{j}) \subset H_{0}^{1}(\Omega) \cap H^{2}(\Omega)$, then there exists $u_{j} \in M(\lambda_{j})$ with $\| u_{j}  \|_{0,\rho} = 1$, such that
\begin{align*}
\big|\lambda_{j}-\lambda_{j,h}^{\rm{R}}\big| \lesssim h^{2}, \quad
\big\|u_{j}-u_{j,h}^{\rm{R}}\big\|_{0,\rho} \lesssim h^{2}, \mbox{ and }
\big|u_{j}-u_{j,h}^{\rm{R}}\big|_{1,h} \lesssim h ;
\end{align*}
\item[(b)] if the mesh is uniform and $M(\lambda_{j}) \subset H_{0}^{1}(\Omega) \cap H^{3}(\Omega)$, then there exists $u_{j} \in M(\lambda_{j})$ with $\| u_{j}  \|_{0,\rho} = 1$, such that
$
\big|u_{j}-u_{j,h}^{\rm{R}}\big|_{1,h} \lesssim h^{2}.
$
\end{itemize}
\end{theorem}

Similar to the basic relation between $V$ and its conforming approximation $V_{h}^{C}$ for eigenvalue problems in \cite{Babuska1991}, the following relation holds. 
\begin{lemma}\label{lem:conformRelation} Let $(\lambda_{h}^{\rm{M}}, u_{h}^{\rm{M}})$ be an eigen-pair of \eqref{eq:eigenRM} with $\big\|u_{h}^{\rm{M}}\big\|_{0,\rho} =1$. Denote, the Rayleigh quotient, by $R(v_{h}) := \dfrac{a_{h}(v_{h},v_{h})}{b(v_{h},v_{h})}$.  For any $v_{h} \in V_{hs}^{\rm{M}}$ with $\|v_{h}\|_{0,\rho} = 1$, it holds that
\begin{equation}
R( v_{h} ) - \lambda_{h}^{\rm{M}} = \big|v_{h} - u_{h}^{ \rm{M}}\big|_{1,h}^{2} - \lambda_{h}^{ \rm{M} }\big\|v_{h}-u_{h}^{\rm{M}}\big\|_{0,\rho}^{2}.
\end{equation}
\end{lemma}
\begin{proof}
From $\|v_{h}\|_{0,\rho} = \big\|u_{h}^{\rm{M}}\big\|_{0,\rho} =1$ and $\lambda_{h}^{\rm{M}} = a_{h}(u_{h}^{\rm{M}},u_{h}^{\rm{M}})$, it is equivalent to prove that
\begin{equation*}
\begin{split}
a_{h}(v_{h},v_{h}) - a_{h}(u_{h}^{\rm{M}},u_{h}^{\rm{M}}) & =a_{h}(v_{h}+u_{h}^{\rm{M}},v_{h}-u_{h}^{\rm{M}})\\
& = a_{h}(v_{h}-u_{h}^{\rm{M}},v_{h}-u_{h}^{\rm{M}})+2a_{h}( u_{h}^{\rm{M}}, v_{h}-u_{h}^{\rm{M}})\\
& = \big|v_{h}-u_{h}^{\rm{M}} \big|_{1,h}^{2} + 2\lambda_{h}^{\rm{M}}b(u_{h}^{\rm{M}}, v_{h}-u_{h}^{\rm{M}})\\
& = \big|v_{h}-u_{h}^{\rm{M}}\big|_{1,h}^{2} - \lambda_{h}^{\rm{M}}\left[ -2b(u_{h}^{\rm{M}}, v_{h}) + b(u_{h}^{\rm{M}},u_{h}^{\rm{M}}) + b(v_{h},v_{h})\right ]\\
& = \big|v_{h}-u_{h}^{\rm{M}} \big|_{1,h}^{2} - \lambda_{h}^{\rm{M}}b(v_{h}-u_{h}^{\rm{M}}, v_{h}-u_{h}^{\rm{M}})\\
& =  \big|v_{h}-u_{h}^{\rm{M}} \big|_{1,h}^{2} - \lambda_{h}^{\rm{M}}\big\|v_{h}-u_{h}^{\rm{M}}\big\|_{0,\rho}^{2},
\end{split}
\end{equation*}
where we utilize again the assumption: $b(v_{h},v_{h})=b(u_{h}^{\rm{M}},u_{h}^{\rm{M}})$.\end{proof}
\begin{theorem}\label{thm:lowerR}
Let $(\lambda_{j},u_{j})$, $ (\lambda_{j,h}^{\rm{R}}, u_{j,h}^{\rm{R}})$ and $(\lambda_{j,h}^{\rm{M}}, u_{j,h}^{\rm{M}})$ be the $j$-th exact eigen-pair and its discrete approximations with $\|u_{j}\|_{0,\rho} = \big\|u_{j,h}^{\rm{R}}\big\|_{0,\rho} = \big\|u_{j,h}^{\rm{M}}\big\|_{0,\rho} =1$.  Assume that $u_{j}\in H_{0}^{1}(\Omega)\cap H^{3}(\Omega)$ and the mesh is uniform. Provided that $h$ is small enough, we have
\begin{align}
C_{\rm R}h^{2} \leqslant \lambda_{j} - \lambda_{j,h}^{\rm{R}} \leqslant \lambda_{j} - \lambda_{j,h}^{\rm{M}},
\end{align}
where $C_{\rm R}$ is a positive constant independent of $h$.
\end{theorem}

\begin{proof}
Since $V_{hs}^{\rm{R}} \subset V_{hs}^{\rm{M}}$, the second inequality, or $\lambda_{j,h}^{\rm{R}} \geqslant \lambda_{j,h}^{\rm{M}}$, holds from the minimum-maximum principle \cite{Babuska1991}.  Let $v_{h} = u_{j,h}^{\rm{R}}$ in Lemma \ref{lem:conformRelation}. We obtain
\begin{align}\label{eq:RM1}
\lambda_{j,h}^{\rm{R}} - \lambda_{j,h}^{\rm{M}} = R( u_{j,h}^{\rm{R}} ) - \lambda_{j,h}^{\rm{M}} = \big|u_{j,h}^{\rm{R}} - u_{j,h}^{ \rm{M}}\big|_{1,h}^{2} - \lambda_{j,h}^{ \rm{M} }\big\|u_{j,h}^{\rm{R}}-u_{j,h}^{\rm{M}}\big\|_{0,\rho}^{2}.
\end{align}

\noindent From Theorem \ref{thm:EerrorRM}, Theorem \ref{thm:EerrorRRM}, and the triangle inequality, it holds that
\begin{align}\label{eq:RM2}
\big|u_{j,h}^{\rm{R}} - u_{j,h}^{ \rm{M}}\big|_{1,h}^{2} - \lambda_{j,h}^{\rm{M}} 
\big\|u_{j,h}^{\rm{R}} - u_{j,h}^{ \rm{M}}\big\|_{0,\rho}^{2} \lesssim h^{4}.
\end{align}
A combination of \eqref{eq:lowerM}, \eqref{eq:RM1}, and \eqref{eq:RM2} yields that
\begin{align*}
\lambda_{j} - \lambda_{j,h}^{\rm{R}}= \lambda_{j} - \lambda_{j,h}^{\rm{M}}  + \lambda_{j,h}^{\rm{M}}  - \lambda_{j,h}^{\rm{R}} \geqslant \alpha h^{2} + o(h^{2}).
\end{align*}
\noindent Hence the result.
\end{proof}

\begin{remark}{\rm To the best of our knowledge, the RM element and the RRM element are the only two elements, by which eigenvalue approximations have the same convergence rates as that of eigenfunction approximations in the energy norm.}
\end{remark}

\subsection{Implementation and numerical results}
\subsubsection{Implementation}
Since constructing clearly a basis functions of $V_{hs}^{\rm{R}}$ on an arbitrary grid is sophisticated, we now present an available approach how \eqref{eq:sourceRRM} and \eqref{eq:eigenRRM} can be implemented. We start with the fact that $V_{hs}^{\rm{R}} = \Big\{w_{h} \in V_{h0}^{\rm{W}}: \int_{e} \jump{\partial_{n_{e}}w_{h}}  =0,\forall e \in \mathcal{E}_{h}^{i} \Big\}$. Define $P_{0}(\mathcal{E}_{h}^{i})$ the space of piecewise constant functions defined on $\mathcal{E}_{h}^{i}$.

An equivalent formulation of  \eqref{eq:sourceRRM} is to find $(u_{h}, \delta_{h}) \in V_{h0}^{\rm{W}} \times  P_{0}(\mathcal{E}_{h}^{i})$, such that
\begin{equation}\label{eq:eqformulP}
\left\{
\begin{split}
&  a_{h}(u_{h},v_{h}) + \sum_{e\in \mathcal{E}_{h}^{i}} \fint_{e}\jump{\partial_{n_{e}}v_{h}}\delta_{h}  =  b(f, v_{h}),\\
& \sum_{e\in \mathcal{E}_{h}^{i}} \fint_{e}\jump{\partial_{n_{e}}u_{h}}\mu_{h}   \qquad \qquad \;=  0.
\end{split}
\right.
\end{equation}

 An equivalent formulation of  \eqref{eq:eigenRRM} is to find $(\lambda_{h}, u_{h}, \delta_{h}) \in \mathbb{R}  \times V_{h0}^{\rm{W}} \times P_{0}(\mathcal{E}_{h}^{i})$ with $\|u_{h}\|_{0,\rho} = 1$, such that

\begin{equation}\label{eq:eqformulE}
\left\{
\begin{split}
&  a_{h}(u_{h},v_{h}) + \sum_{e\in \mathcal{E}_{h}^{i}} \fint_{e}\jump{\partial_{n_{e}}v_{h}}\delta_{h}  =  \lambda_{h}b(u_{h}, v_{h}),\\
& \sum_{e\in \mathcal{E}_{h}^{i}} \fint_{e}\jump{\partial_{n_{e}}u_{h}}\mu_{h}   \qquad \qquad \;=  0.
\end{split}
\right.
\end{equation}

Problem \eqref{eq:eqformulP} admits a solution $(u_{h}, \delta_{h})$, where $u_{h}$ solves problem \eqref{eq:sourceRRM}, and problem \eqref{eq:eqformulE} admits a solution $(\lambda_{h}, u_{h}, \delta_{h})$, where $(\lambda_{h}, u_{h})$ solves problem \eqref{eq:eigenRRM}.

\begin{remark}{\rm
It's worth mentioning that formulations \eqref{eq:eqformulP} and \eqref{eq:eqformulE} can be applied to an arbitrary quadrilateral grid \cite{Shuo.Zhang2017}. For the case wherein the grid comprises rectangles, a detailed construction process of basis functions of $V_{hs}^{\rm R}$ is given in Appendix \ref{sec:appB}, based on which \eqref{eq:sourceRRM} and \eqref{eq:eigenRRM} can be implemented directly in elliptic formulation.}
\end{remark}

\subsubsection{Numerical experiments}
Let $\Omega=(0,1)^{2}$. We consider nonuniform meshes (see Figure \ref{fig:nonuniformgrid}) with $\frac{h_{x,K}}{h_{y,K}} \in \Big\{ \frac{0.35}{0.65},\frac{0.65}{0.35},\frac{1}{1}\Big\}$, and uniform meshes with $\frac{h_{x,K}}{h_{y,K}} = \frac{1}{2}$. Numerical examples of the source problem and the eigenvalue problem are given below. 
\begin{figure}
\begin{minipage}[t]{0.5\linewidth}
\centering
\includegraphics[width=2.8in]{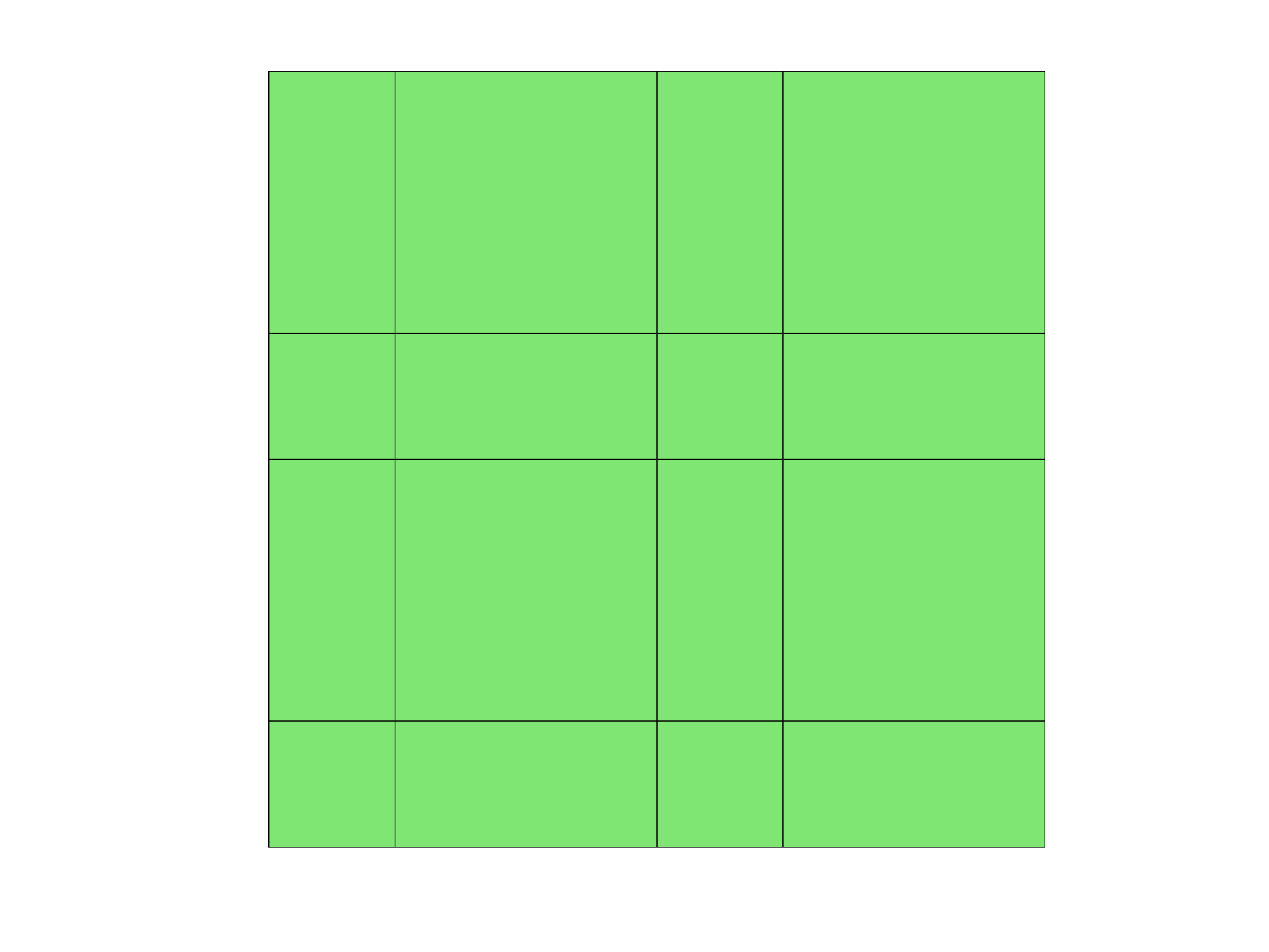}
\end{minipage}%
\begin{minipage}[t]{0.5\linewidth}
\centering
\includegraphics[width=2.8in]{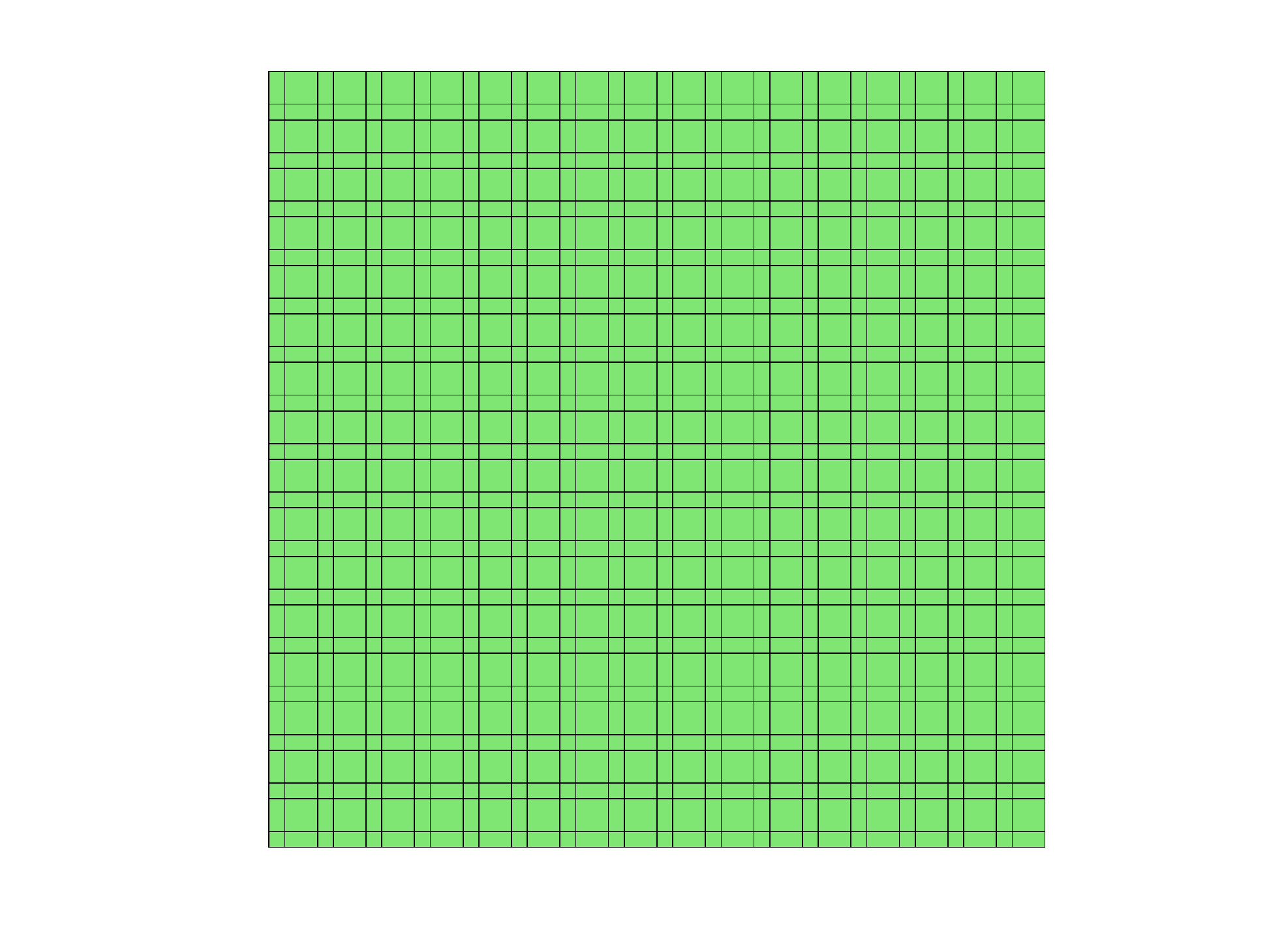}
\end{minipage}
\caption{Illustration of a nonuniform shape regular subdivision. The partition in the right is a combination of small patterns as the left one.}\label{fig:nonuniformgrid}
\end{figure}

\noindent{\bf Example 1 for the source problem.} 

\noindent Consider \eqref{eq:Poisson} with $f = 2 \pi^2 sin(\pi x) sin(\pi y)$. The exact solution $u$ is computed as $u = sin(\pi x) sin(\pi y)$. Apply \eqref{eq:eqformulP} to get the discrete solutions on uniform and nonuniform meshes. From Figure~\ref{fig:errorPoisson}, the convergence rate is $\mathcal{O}(h)$ in the energy norm, and $\mathcal{O}(h^{2})$ in $L^{2}$-norm, on a nonuniform mesh. Both rates reach $\mathcal{O}(h^{2})$ order on uniform grids.
%%%%%%
\begin{figure}
\begin{minipage}[t]{0.5\linewidth}
\centering
\includegraphics[width=3in]{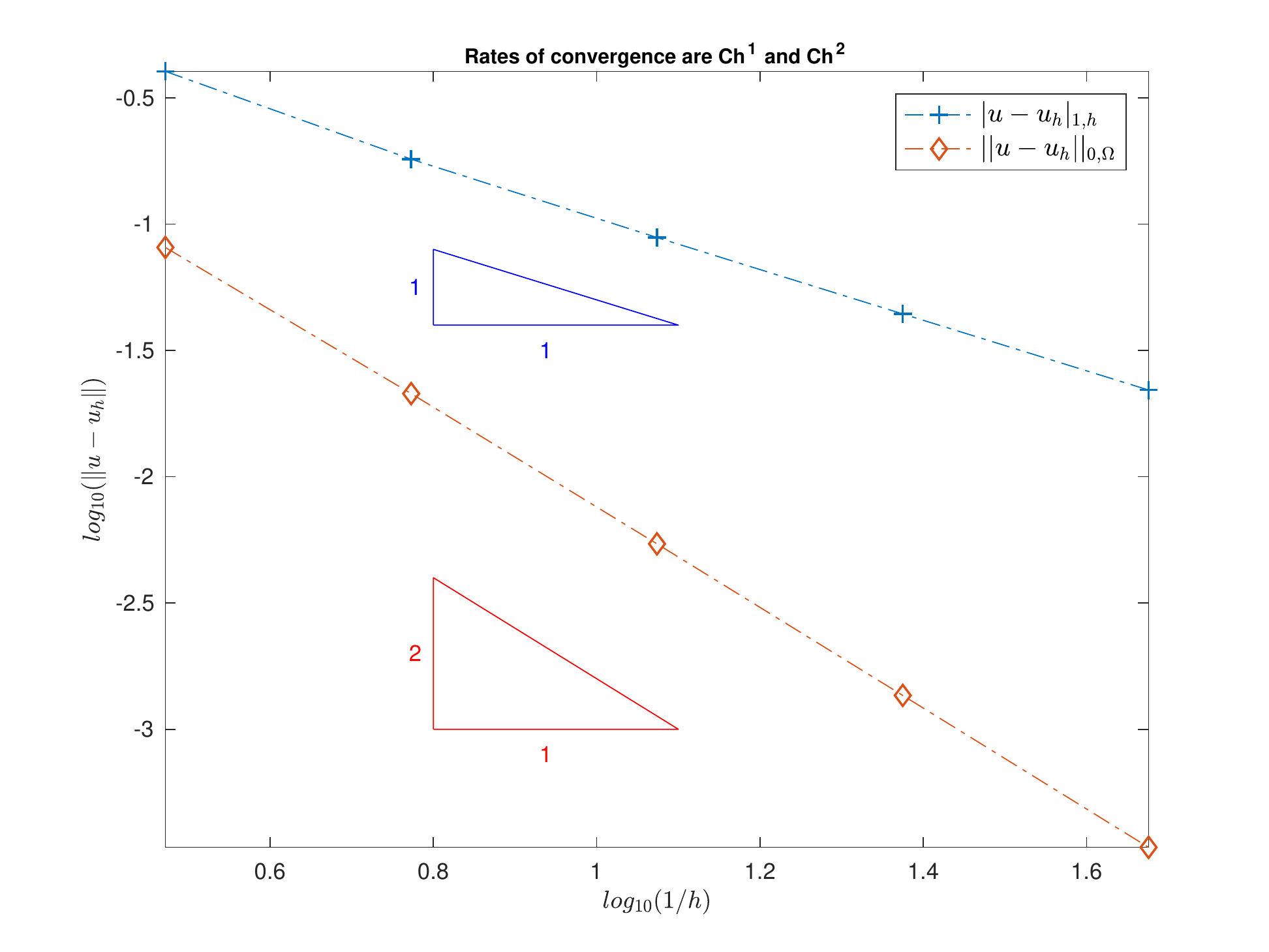}
\end{minipage}%
\begin{minipage}[t]{0.5\linewidth}
\centering
\includegraphics[width=3in]{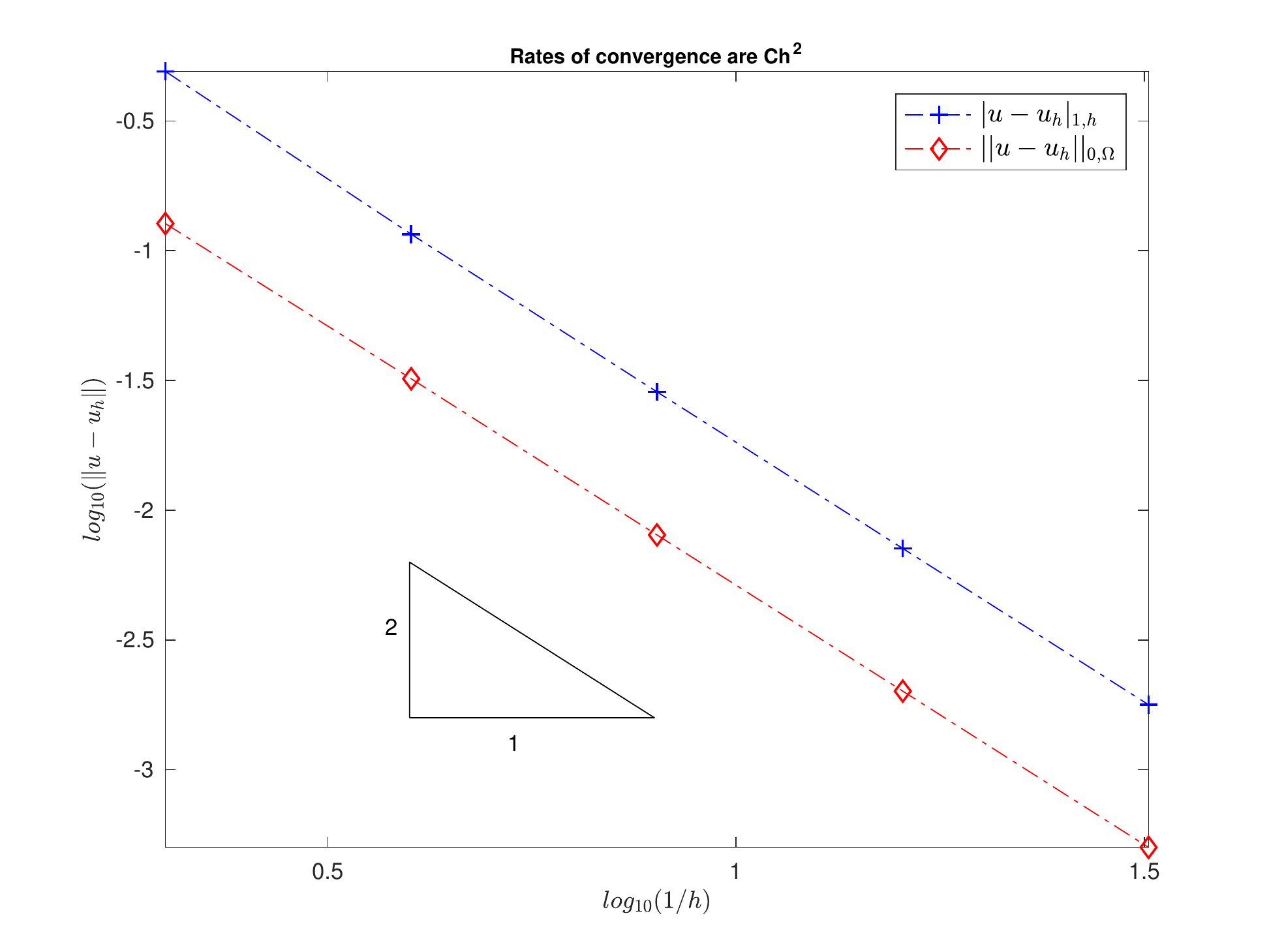}
\end{minipage}
\caption{Source problem on rectangle domain: Errors in the energy norm and $L^{2}$-norm with nonuniform (Left) or uniform (Right) subdivisions.}\label{fig:errorPoisson}
\end{figure}
%%%%%%

\noindent{\bf Example 2 for the eigenvalue problem.}  

\noindent Consider  \eqref{eq:eigen} with $\rho = 1$. Then we have the exact eigenfunctions $u_{k,l} = 2sin(k\pi x)sin(l\pi y)$ and eigenvalues $\lambda_{k,l} = (k^{2}+l^{2})\pi^{2} \ (k,l \in \mathbb{N}^{+}).$ Arrange them by increasing order. Apply \eqref{eq:eqformulE} to get the the smallest six discrete eigenvalues. From Figure \ref{fig:errEigen}, the convergence rates of eigenvalues almost reach $\mathcal{O}(h^{2})$ order in both nonuniform  and uniform cases. Moreover, from tables \ref{tab:nonuniformEigenvalues} and \ref{tab:uniformEigenvalues}, the eigenvalue approximations by the RRM element converge monotonically from below to the exact ones.
%%%%%%%%
\begin{figure}
\begin{minipage}[t]{0.5\linewidth}
\centering
\includegraphics[width=3in]{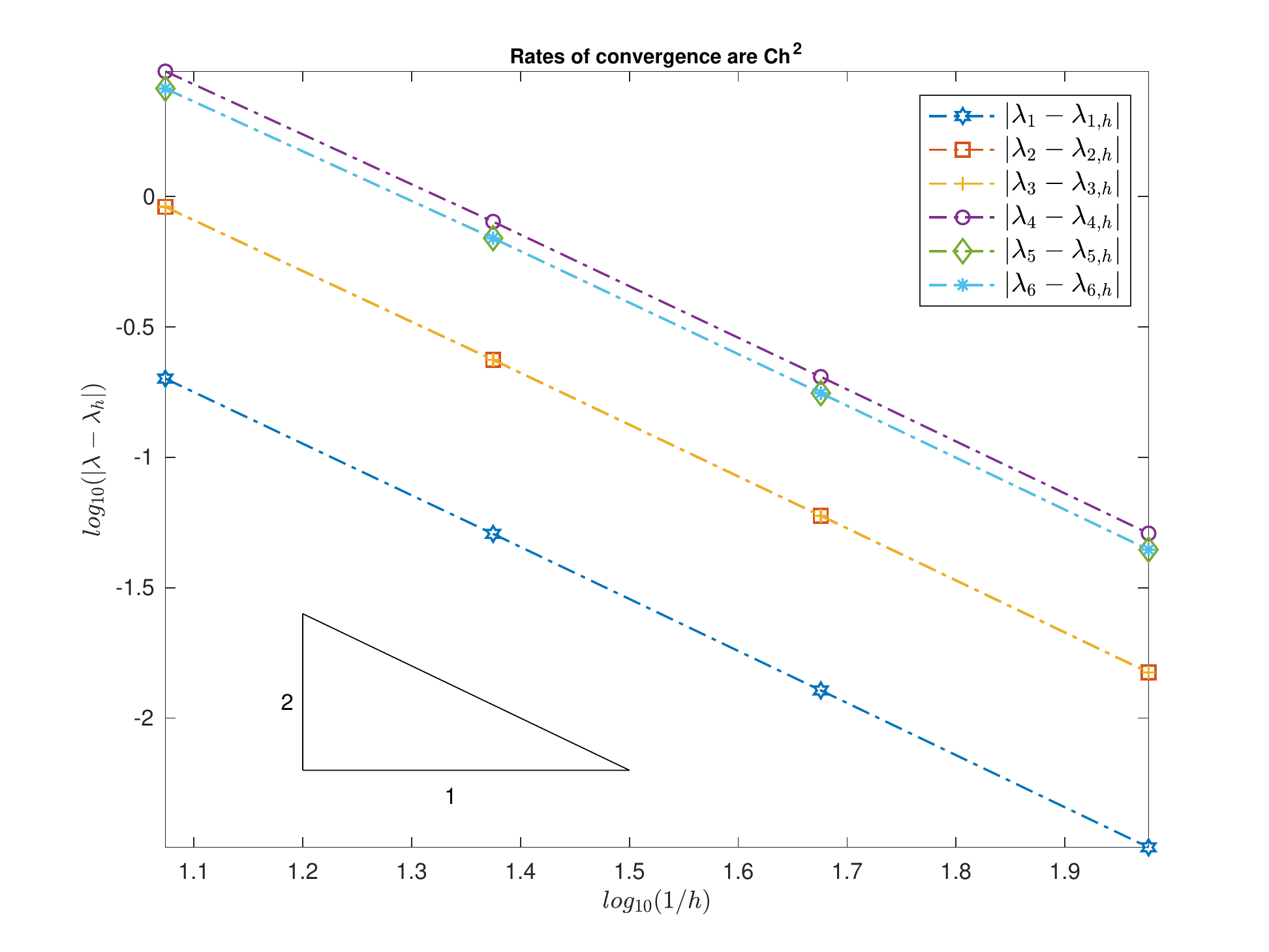}
\end{minipage}%
\begin{minipage}[t]{0.5\linewidth}
\centering
\includegraphics[width=3in]{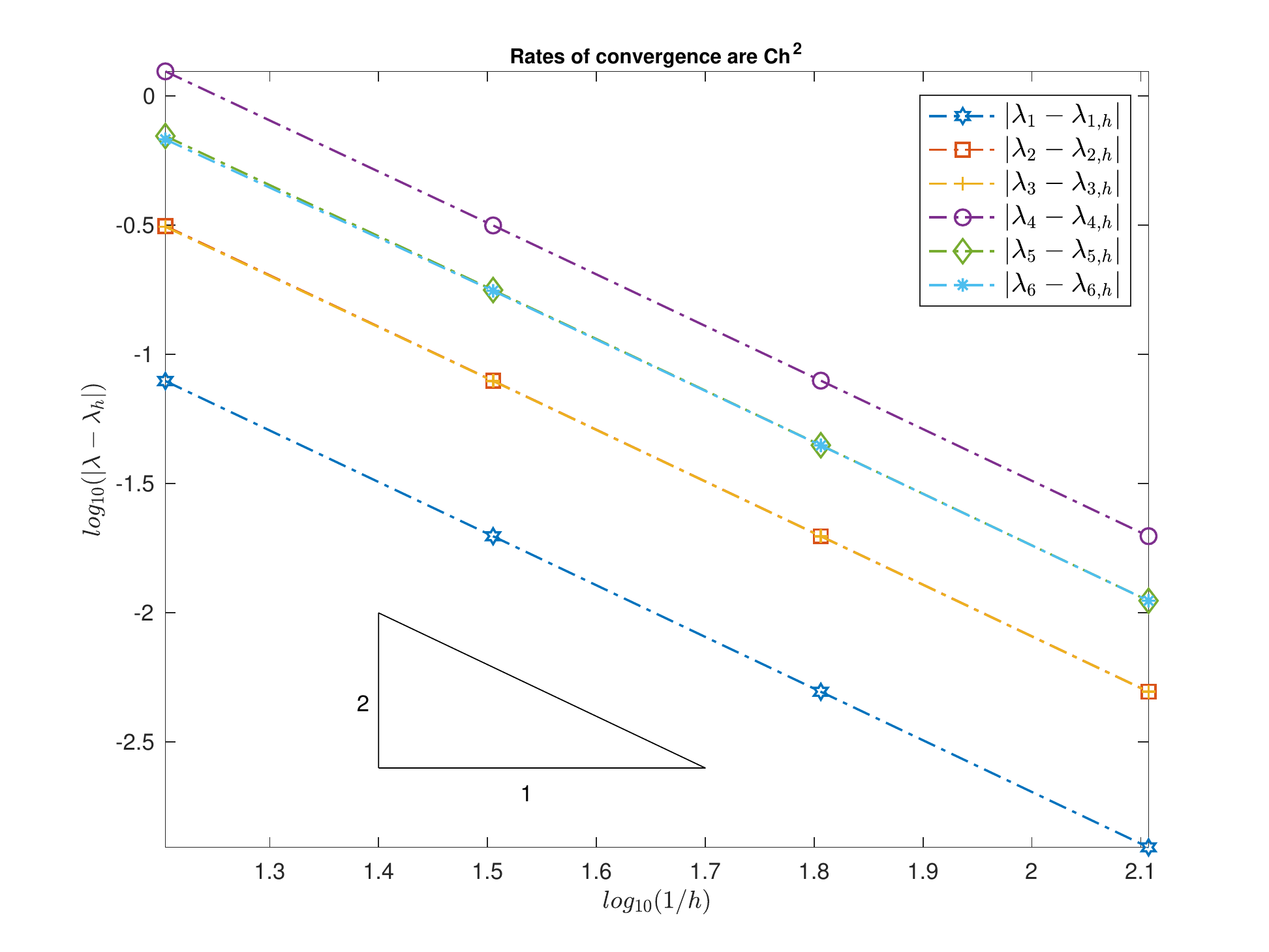}
\end{minipage}
\caption{Eigenvalue problem on rectangle domain: Errors of eigenvalues with nonuniform subdivisions (Left) and uniform subdivisions (Right).}\label{fig:errEigen}
\end{figure}
%%%%%%%%
\begin{table}[htbp]
\caption{Eigenvalues computed by the RRM element on nonuniform grids.}
\footnotesize
\centering
\begin{tabular}{c|cc|ccc|cc|ccc}
\hline
$h$ & $\lambda_{1,h}$ & order & $\lambda_{2,h}$ &  $\lambda_{3,h}$ & order  &  $\lambda_{4,h}$ & order  & $\lambda_{5,h}$ & $\lambda_{6,h}$ & order\\
\hline
$0.3375$ & 17.178 &  & 38.419 & 39.047 & & 47.604 &  & 65.324 & 78.711 & \\
$0.1688$ & 18.982 & 1.690 & 46.125 & 46.126 &1.696,1.599 & 68.932 & 1.564 & 90.210& 90.236& 1.966,1.181\\
$0.084375$ & 19.539 & 1.892 & 48.437 & 48.437 & 1.768, 1.768 & 75.941 & 1.662 & 96.105 & 96.106 &1.638,1.633\\
$0.0421875$ & 19.688 & 1.969 & 49.112 & 49.112 & 1.928, 1.928 & 78.157 & 1.884 & 98.004 & 98.004 &1.873,1.873\\
$0.02109375$ & 19.726 & 1.992 & 49.288 & 49.288 & 1.980,1.980 & 78.754 & 1.968 & 98.520 & 98.520 &1.964,1.964\\
$0.010546875$ & 19.736 & 1.998 & 49.333 & 49.333 & 1.995,1.995 & 78.906 & 1.992 & 98.652 & 98.652 &1.991,1.991 \\
Trend & $\nearrow$ & & $\nearrow$ & $\nearrow$ & & $\nearrow$ &  & $\nearrow$ & $\nearrow$&\\
Exact $\lambda_{j}$ & 19.739 &--  & 49.348 &
49.348 & & 78.957 & & 98.696 & 98.696 & \\
\hline
 \end{tabular}\label{tab:nonuniformEigenvalues}
\end{table}
%%%%%%%%
\begin{table}[htbp]
\caption{Eigenvalues computed by the RRM element on uniform grids.}
\footnotesize
\centering
\begin{tabular}{c|cc|ccc|cc|ccc}
\hline
$h \ (=h_{x})$ & $\lambda_{1,h}$ & order & $\lambda_{2,h}$ &  $\lambda_{3,h}$ & order  &  $\lambda_{4,h}$ & order  & $\lambda_{5,h}$ & $\lambda_{6,h}$ & order\\
\hline
$0.25$ & 18.559 & & 44.961 & 45.655 & & 63.427 &  & 90.249 & 95.913 & \\
$0.125$ & 19.428 & 1.896 & 48.127 & 48.163 &1.796,1.558 & 74.233 & 1.644 & 96.050 & 96.427 & 1.596,0.613\\
$0.0625$ & 19.660 & 1.973 & 49.034 & 49.036 & 1.944,1.899 & 77.711 &1.896 & 97.996 & 98.016 &1.890,1.670\\
$0.03125$ & 19.719 & 1.993 & 49.269 & 49.269 & 1.986, 1.975 & 78.641 &1.973 & 98.519 & 98.520 &1.972,1.928\\
$0.015625$ & 19.734 & 1.998 & 49.328 & 49.328 & 1.996,1.994 & 78.878 & 1.993 & 98.652 & 98.652 &1.993,1.983\\
$0.0078125$ & 19.738 & 2.000 & 49.343 & 49.343 & 2.000,1.996 & 78.937 & 1.998 & 98.685 & 98.685 &1.998,1.996\\
Trend & $\nearrow$ & & $\nearrow$ & $\nearrow$ & & $\nearrow$ &  & $\nearrow$ & $\nearrow$& \\
Exact $\lambda_{j}$ & 19.739 &  & 49.348 &
49.348 & & 78.957 &  & 98.696 & 98.696 & \\
\hline
 \end{tabular}\label{tab:uniformEigenvalues}
\end{table}
%%%%%%

%
%
%
\section{Conclusions and discussions}
\label{sec:concdisc}

\subsection{Concluding remarks}
In this paper, we present a reduced rectangular Morley element scheme for $H^{1}$ problems. Technically, the exactness relation between the RRM element and the PS element is figured out, and the approximation error estimate is established by an auxiliary Stokes problem. For the source problem, the convergence rate of this scheme is $\mathcal{O}(h)$ in the energy norm and $\mathcal{O}(h^{2})$ in $L^{2}$-norm, on general meshes. The error estimate in the energy norm reaches $\mathcal{O}(h^{2})$ order on uniform grids.  
Besides, a lower bound of the $L^{2}$-norm error is proved, and the best $L^{2}$-norm error estimate is at most  $\mathcal{O}(h^{2})$.
For the eigenvalue problem, the discrete eigenvalues by the RM element and the RRM element are both proved to be lower bounds of the exact ones. In fact, the inequality \eqref{eq:propertyPiRM}, reads $a_{h}(w-\Pi_{h}^{\rm M}w,\Pi_{h}^{\rm M}w) \geqslant \alpha h^{2}$, or \eqref{eq:property1PiRM}, reads $a_{h}(w-\Pi_{h}^{\rm M}w,w) \geqslant \alpha_{1} h^{2}$, is the dominant factor for the RRM element to yield these lower bounds, where $\Pi_{h}^{\rm M}$ is the interpolation operator for the RM element. 

Roughly speaking, for schemes which provide the lower bounds of the eigenvalues, a smaller space provides a better approximation. This can be viewed as a motivation for the optimal space.

\subsection{Further discussions}
In this paper, we mainly focus on the convex domain (rectangle domain) case. Also, for the eigenvalue problem, we pay special attention to the computation of eigenvalues. Some more numerical experiments illustrate that the schemes can perform even better than the theoretical description in this paper. These can stimulate further research, and we list part of them below. 

   (1) Consider the eigenvalue problem \eqref{eq:eigen} with $\rho = 1$ on $\Omega = (0,1)^{2}$. From Figure \ref{fig:1stEigenfunction}, right, the convergence rate of the first eigenfunction in $L^{2}$-norm reaches $\mathcal{O}(h^{3})$ order on uniform grids, while in Theorem \ref{thm:EerrorRRM} we can only derive $\big\|u_{j}-u_{j,h}^{\rm R}\big\|_{0,\rho} \leqslant \big|u_{j}-u_{j,h}^{\rm R}\big|_{1,h}\lesssim h^{2}$. Although,  there exists a lower bound of the $L^{2}$-norm error  for the source problem, it may not holds for the eigenvalue problem, and it is possible that the convergence rate of the eigenfunctions in $L^{2}$-norm may be higher than that in the energy norm.   
   
   (2) Consider the source problem  \eqref{eq:Poisson}  on L-shape domain: $\Omega = (0,2)^{2}\backslash [1,2]^{2}$. From Figure \ref{fig:LshapePoisson}, the convergence rates are consistent with the results derived on $\Omega = (0,1)^{2}$. Although Theorem~\ref{thm:errorRRM} for the source problem is based on the assumption of convex domain, this example implies that the RRM element may also be applicable to non-convex regions. 
   
    (3) Consider the eigenvalue problem \eqref{eq:eigen} with $\rho = 1$ on L-shape domain. The eigenfunctions and eigenvalues are unknown, and eigenfunctions may have singularities around the reentrant corner. From \cite{Babuska1991}, the third eigenfunction is analytic: $u_{3} = sin(\pi x)sin(\pi x)$. We present the errors of $u_{3}$ in Figure \ref{fig:3rdEigenfunction}, and observe that the convergence rates for error on L-shape domain are the same as that on a rectangle region. Moreover, the third eigenvalue computed by the RRM scheme satisfies $\lambda_{3,h}\leqslant \lambda_{3} = 2\pi^{2}$ and the convergence rate is $\mathcal{O}(h^{2})$. The perfermance of the smallest six eigenvalues with nonuniform and uniform subdivisions are listed in Tables \ref{tab:LshapenonuniformEigenvalues} and \ref{tab:LshapeuniformEigenvalues}. 			
%%%%%%%%
\begin{table}[htbp]
\footnotesize
\caption{Eigenvalues computed on L-shape domain with nonuniform subdivisions.}
\centering
\begin{tabular}{c|c|c|cc|c|c|c}
\hline
$h$ & $\lambda_{1,h}$ & $\lambda_{2,h}$ &  $\lambda_{3,h}$ & order  & $\lambda_{4,h}$ & $\lambda_{5,h}$ & $\lambda_{6,h}$\\
\hline
$0.675$ & 8.999 & 11.157 & 11.961 &  & 16.776 & 17.955  & 23.829\\
$0.3375$ & 9.637 & 13.948 & 17.219 & 1.626 & 25.337 & 27.892 & 36.616 \\
$0.16875$ & 9.708  & 14.840 & 18.984 & 1.739 & 28.190 & 30.865 & 40.131 \\
$0.084375$  & 9.691 & 15.104 & 19.539 & 1.916 & 29.160 & 31.713 & 41.140 \\
$0.0421875$ & 9.667 & 15.174 & 19.688 & 1.977 & 29.429 & 31.897 & 41.412 \\
$0.02109375$ & 9.652 & 15.191 & 19.726 & 1.994 & 29.498 & 31.923 & 41.469 \\
$0.010546875$ & 9.645 & 15.196 & 19.736 & 1.999 & 29.516 & 31.921 & 41.477 \\
Trend &  &$\nearrow$ & $\nearrow$ & & $\nearrow$ & & \\
 $\lambda_{j} \approx $ & 9.640 & 15.197  & 19.739 &  & 29.522 & 31.913 & 41.475\\
\hline
 \end{tabular}\label{tab:LshapenonuniformEigenvalues}
\end{table}
%%%%%%%%
\begin{table}[htbp]
\footnotesize
\caption{Eigenvalues computed on L-shape domain with uniform subdivisions.}
\centering
\begin{tabular}{c|c|c|cc|c|c|c}
\hline
$h_{x}=2h_{y}$ & $\lambda_{1,h}$ & $\lambda_{2,h}$ &  $\lambda_{3,h}$ & order  & $\lambda_{4,h}$ & $\lambda_{5,h}$ & $\lambda_{6,h}$\\
\hline
$0.5$ & 9.894 & 13.443 & 15.857 &  & 23.914 & 26.659 & 33.739\\
$0.25$ & 9.811 & 14.696 & 18.558 & 1.717 & 27.656 & 30.410 & 40.808 \\
$0.125$ & 9.743 & 15.068 & 19.428 & 1.923 & 29.015 & 31.687 & 41.262 \\
$0.0625$  & 9.691 & 15.165 & 19.660 & 1.980 & 29.392 & 31.921 & 41.456 \\
$0.03125$ & 9.663 & 15.189 & 19.719 & 1.995 & 29.489 & 31.941 & 41.488 \\
$0.015625 $ & 9.650 & 15.195 & 19.734 & 1.999 & 29.513 & 31.930 & 41.485 \\
Trend &  &$\nearrow$ & $\nearrow$ & & $\nearrow$ & & \\
 $\lambda_{j} \approx $ & 9.640  & 15.197  & 19.739 & & 29.522 & 31.913 & 41.475\\
\hline
 \end{tabular}\label{tab:LshapeuniformEigenvalues}
\end{table}
%%%%%%%%
These examples suggest that the RRM element may have better numerical applications, and these will be studied in our future work.
%%%%%%%
\begin{figure}
\begin{minipage}[t]{0.5\linewidth}
\centering
\includegraphics[width=3in]{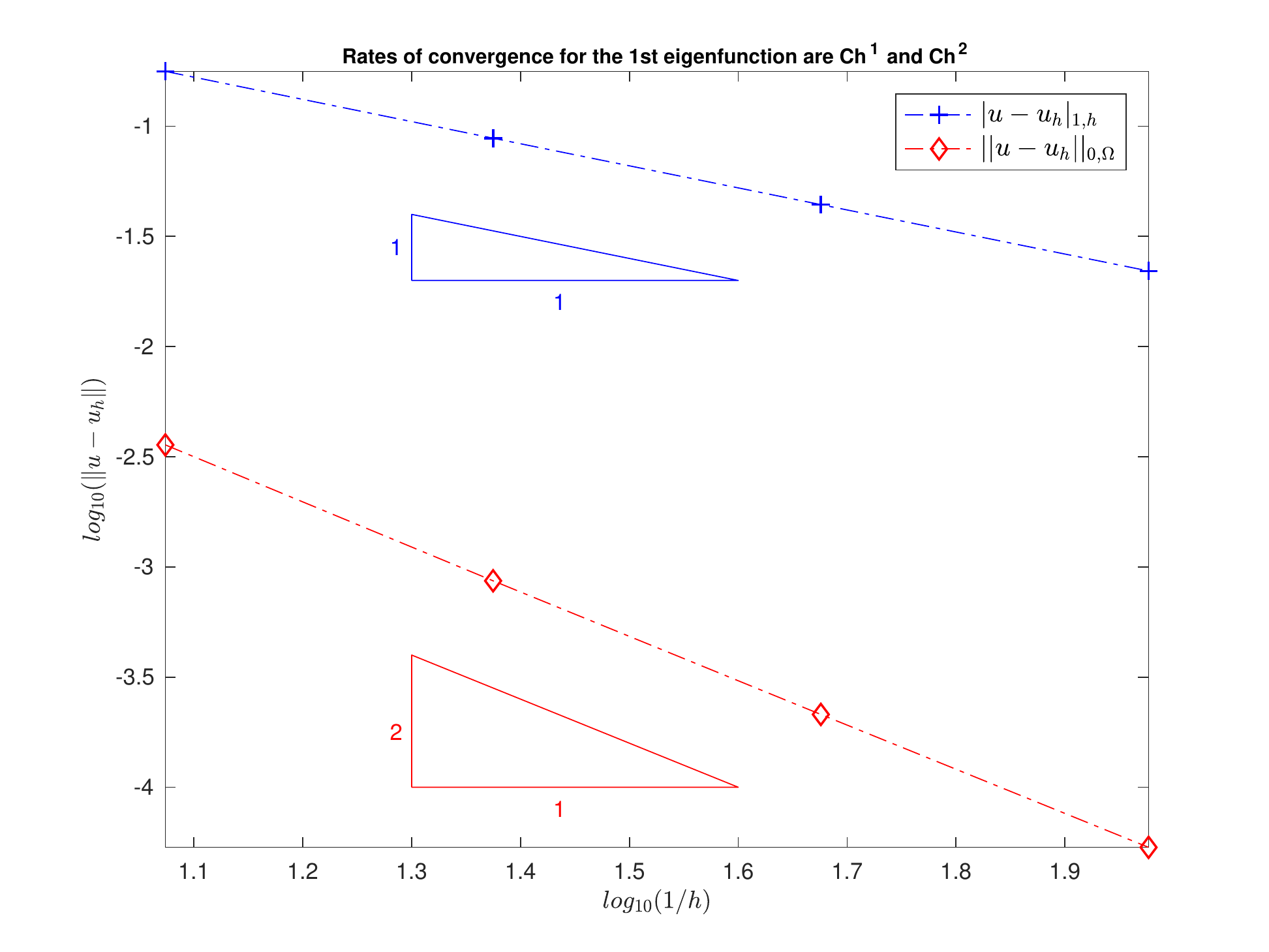}
\end{minipage}%
\begin{minipage}[t]{0.5\linewidth}
\centering
\includegraphics[width=3in]{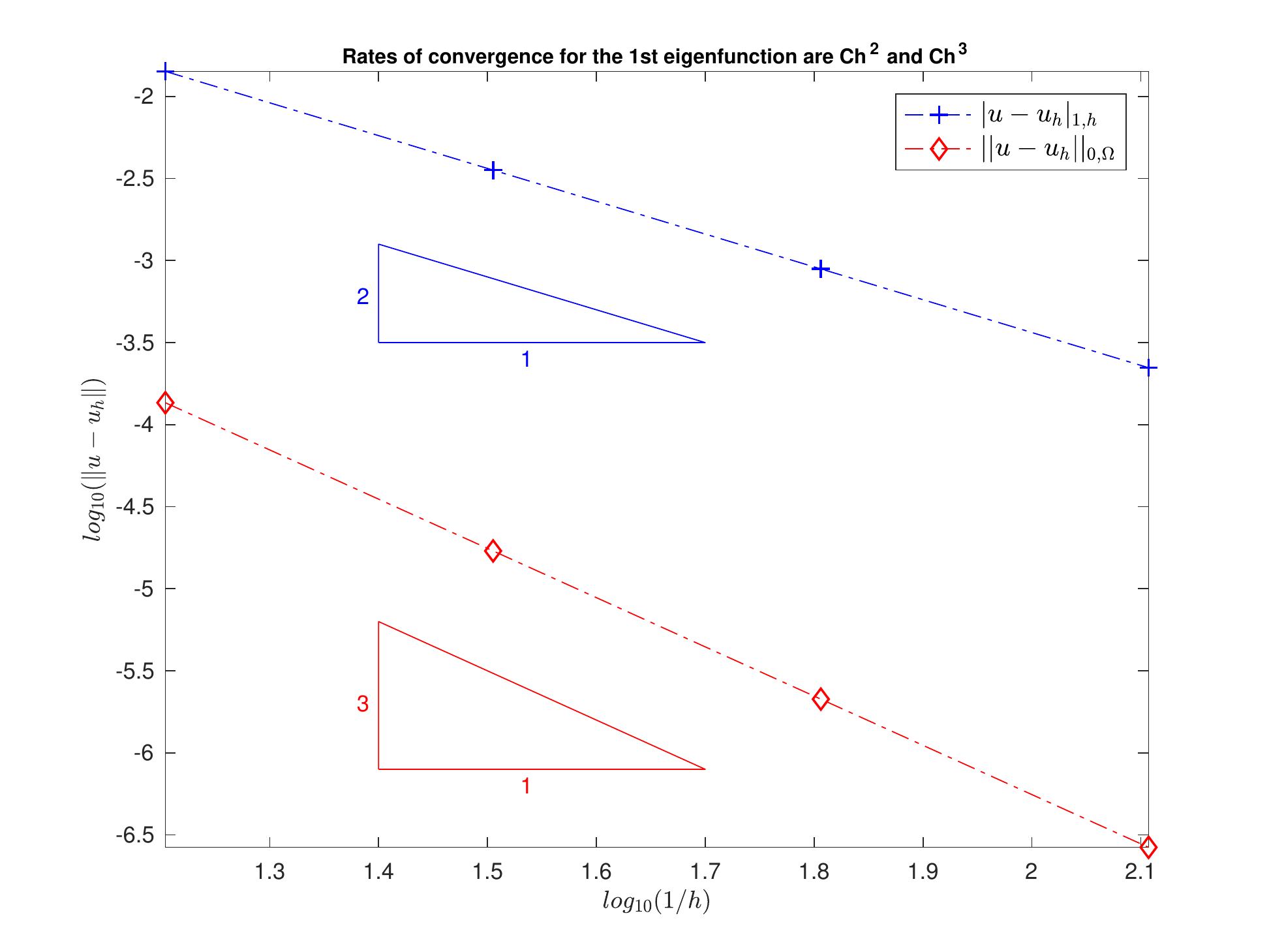}
\end{minipage}
\caption{Eigenvalue problem on rectangle domain: Convergence rates of the 1st eigenfunction with nonuniform (Left) and uniform subdivisions (Right).}\label{fig:1stEigenfunction}
\end{figure}
%%%%%%%%
\begin{figure}
\begin{minipage}[t]{0.5\linewidth}
\centering
\includegraphics[width=3in]{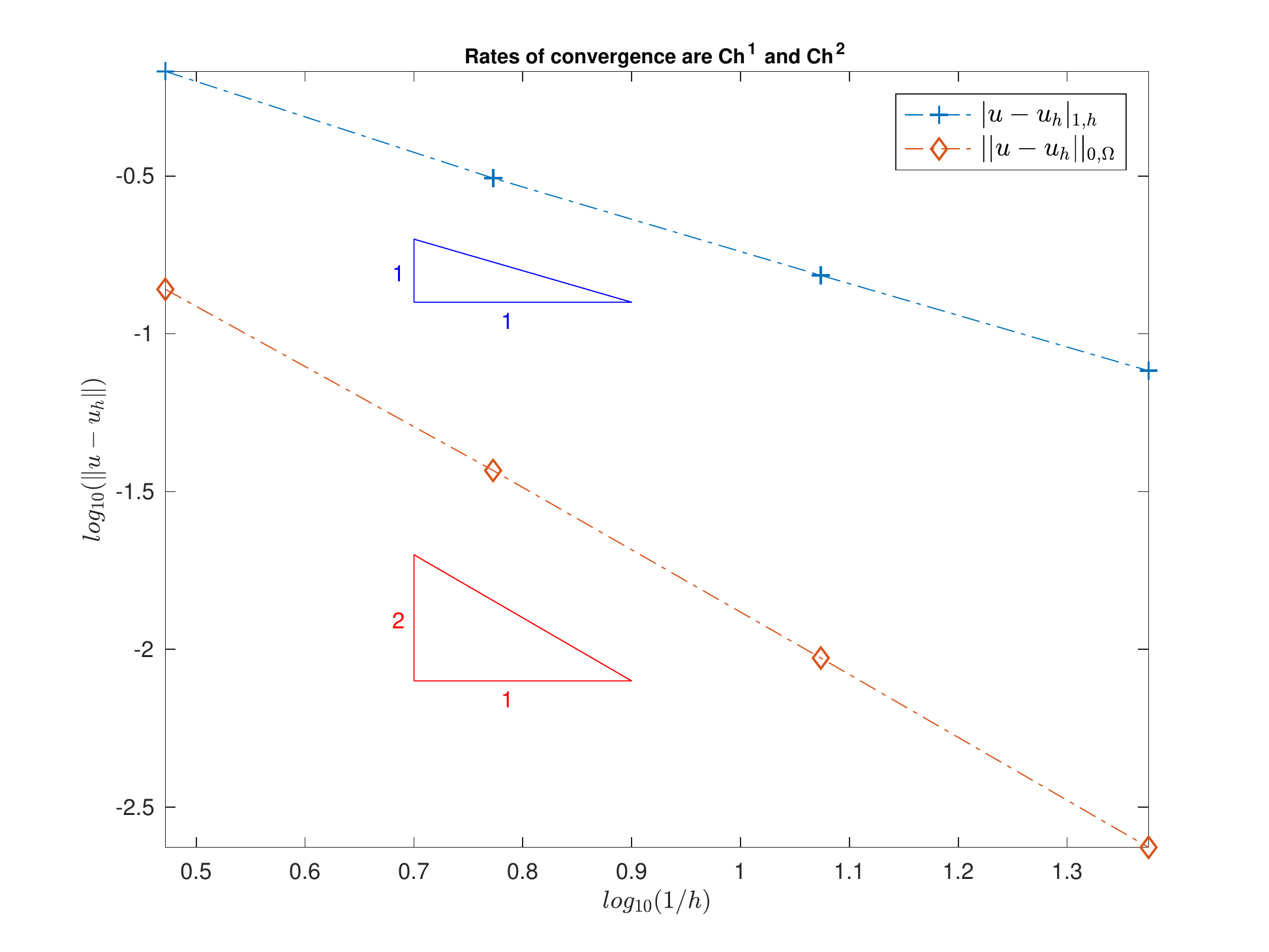}
\end{minipage}%
\begin{minipage}[t]{0.5\linewidth}
\centering
\includegraphics[width=3in]{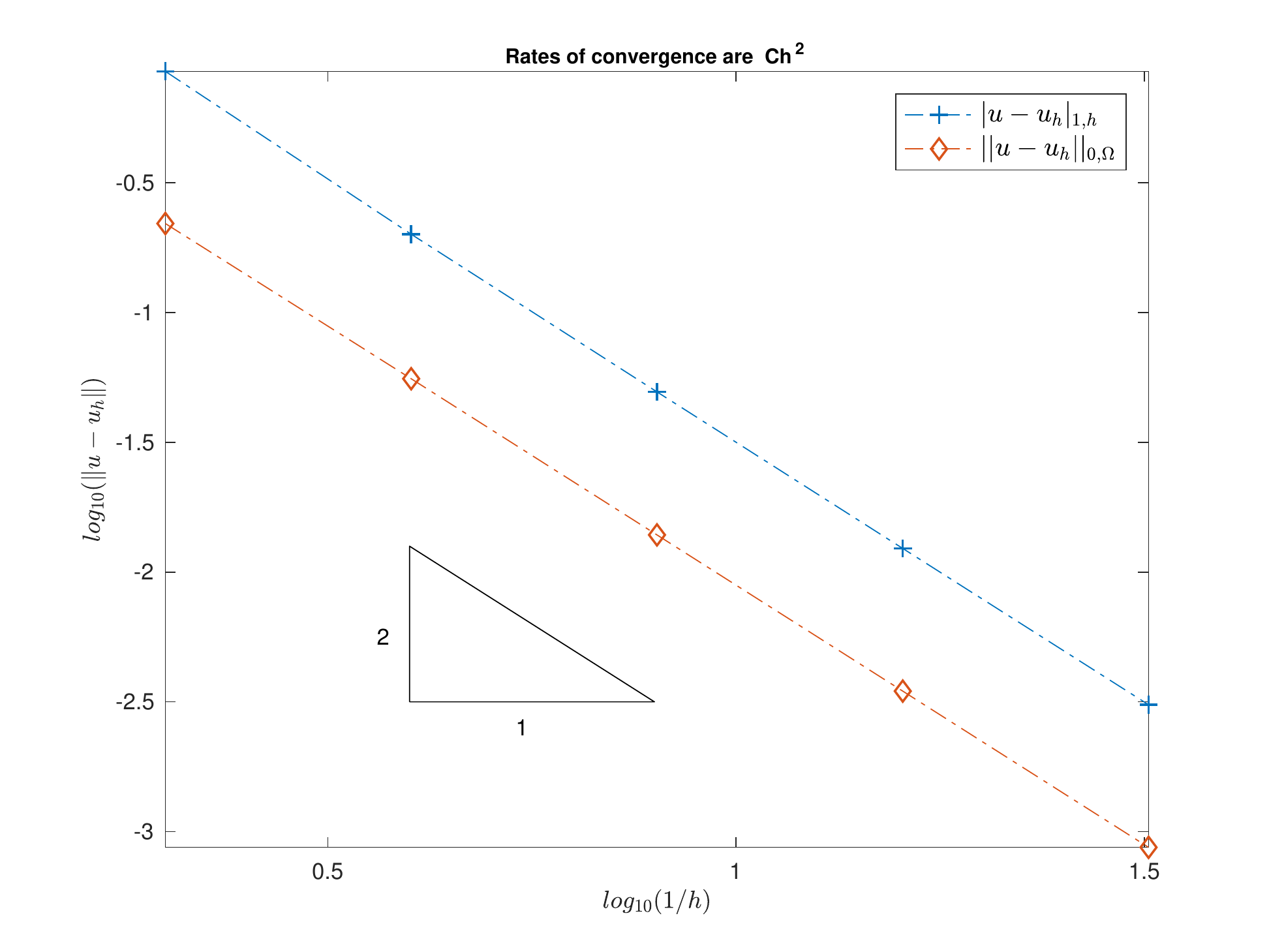}
\end{minipage}
\caption{Source problem on L-shape domain: Errors in the energy norm and $L^{2}$-norm with nonuniform (Left) and uniform subdivisions (Right).}\label{fig:LshapePoisson}
\end{figure}
%%%%%%%%
%%%%%%%%
\begin{figure}
\begin{minipage}[t]{0.5\linewidth}
\centering                            
\includegraphics[width=3in]{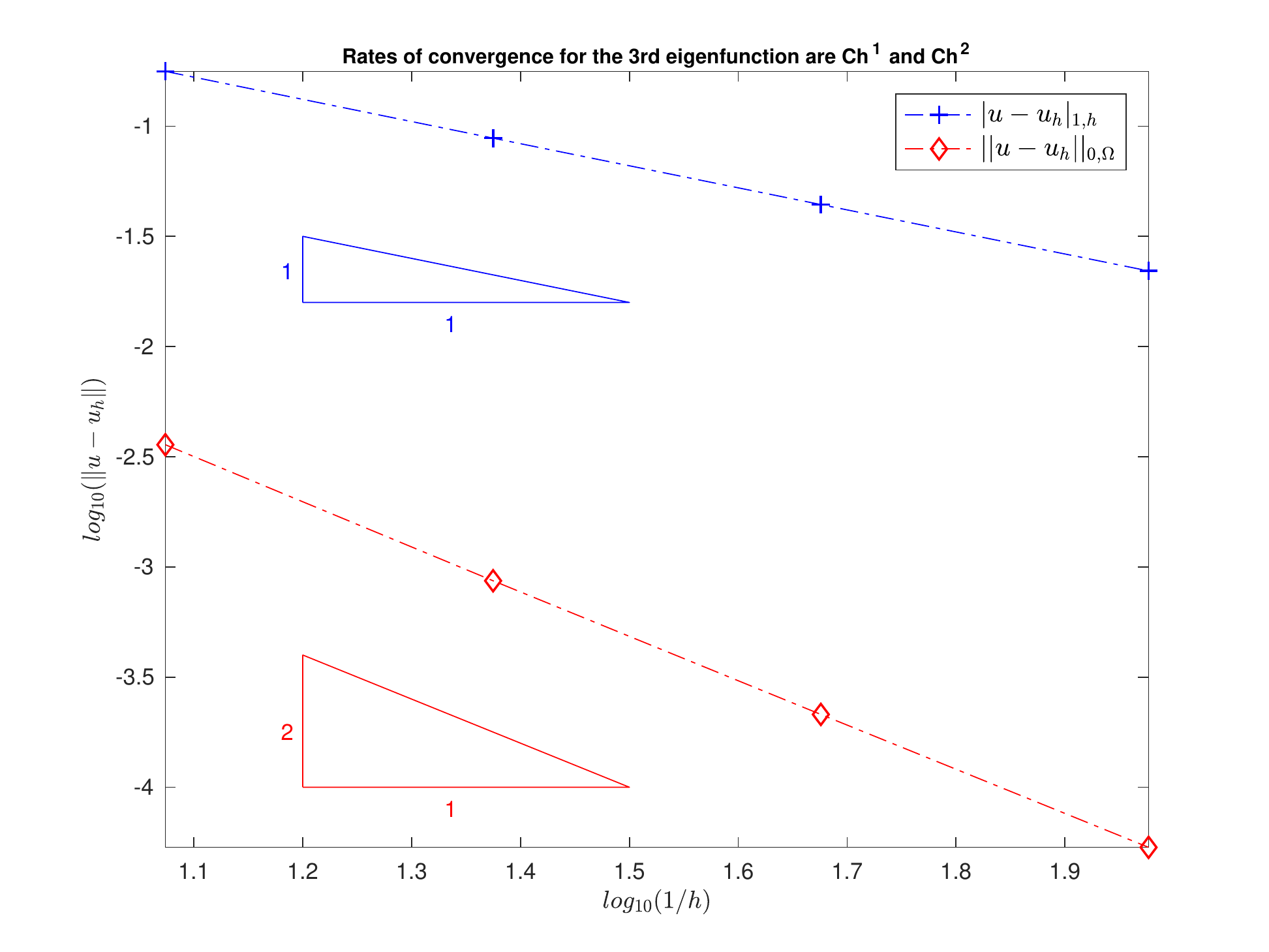}
\end{minipage}%
\begin{minipage}[t]{0.5\linewidth}
\centering
\includegraphics[width=3in]{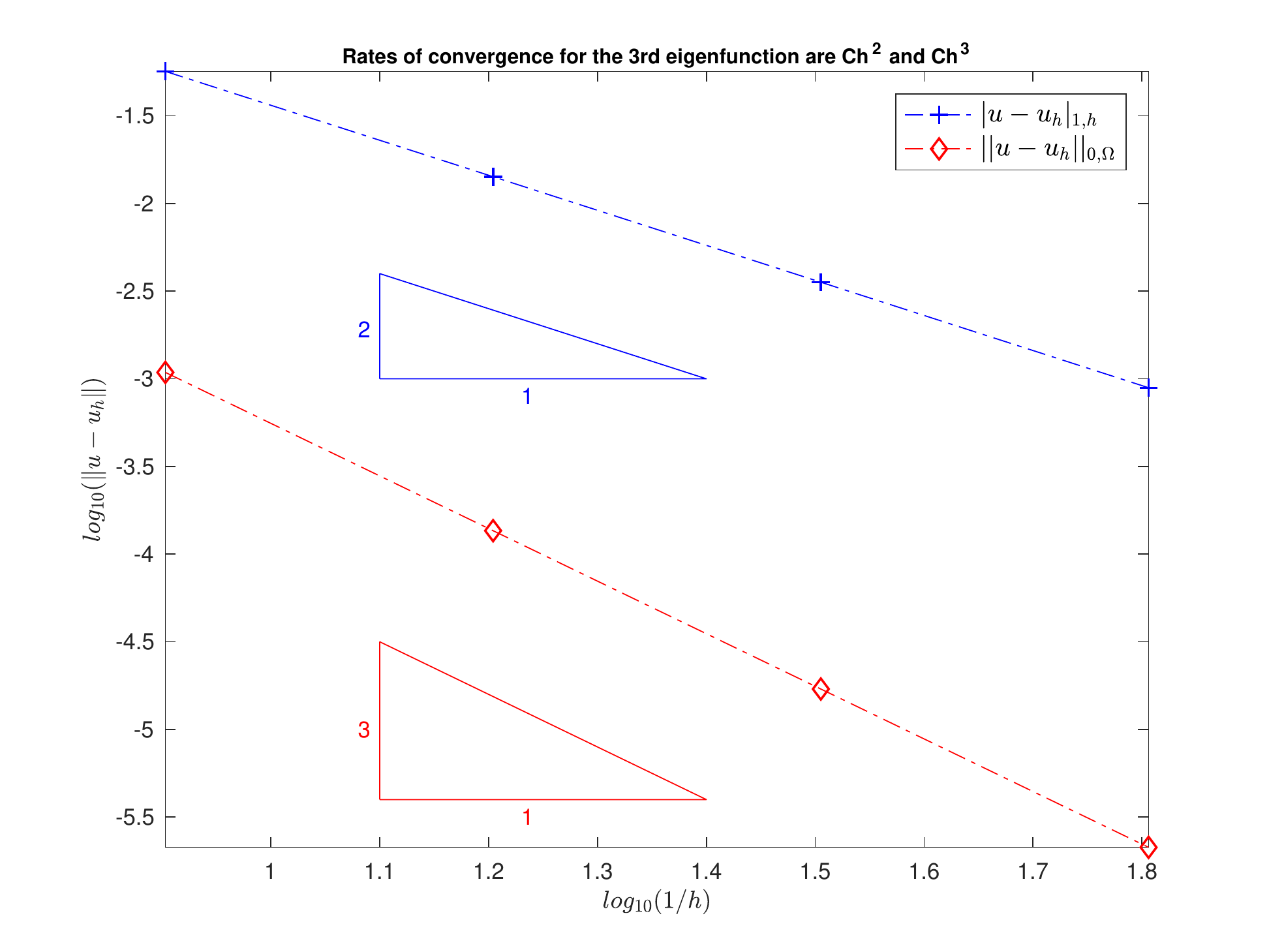}
\end{minipage}
\caption{Eigenvalue problem on L-shape domain: Convergence rates of the 3rd eigenfunction with nonuniform (Left) and uniform subdivisions (Right).}\label{fig:3rdEigenfunction}
\end{figure}

\bibliography{ZZZ_bib}

\begin{thebibliography}{10}

\bibitem{M.Armentano;R.Duran2004}
M.~G. Armentano and R.~G. Dur{\'a}n.
\newblock Asymptotic lower bounds for eigenvalues by nonconforming finite
  element methods.
\newblock {\em Electronic Transactions on Numerical Analysis}, 17(2):93--101,
  2004.

\bibitem{Babuska1991}
I.~Babu{\v{s}}ka and J.~Osborn.
\newblock Eigenvalue problems.
\newblock In {\em Finite Element Methods (Part I)}, volume~2 of {\em Handbook
  of Numerical Analysis}, pages 641--787. Elsevier, 1991.

\bibitem{Blum.H;Rannacher.R1980}
H.~Blum, R.~Rannacher, and R.~Leis.
\newblock On the boundary value problem of the biharmonic operator on domains
  with angular corners.
\newblock {\em Mathematical Methods in the Applied Sciences}, 2(4):556--581,
  1980.

\bibitem{Carstensen.C;Gallistl.D2014}
C.~Carstensen and D.~Gallistl.
\newblock Guaranteed lower eigenvalue bounds for the biharmonic equation.
\newblock {\em Numerische Mathematik}, 126(1):33--51, 2014.

\bibitem{Carstensen.C;Gedicke.J2014}
C.~Carstensen and J.~Gedicke.
\newblock Guaranteed lower bounds for eigenvalues.
\newblock {\em Mathematics of Computation}, 83(290):2605--2629, 2014.

\bibitem{C.Chen2001}
C.~Chen.
\newblock {\em {Finite element superconvergence structure theory}}.
\newblock Hunan Science and Technology Press, Hunan, 2001.

\bibitem{Fortin.M;Soulie.M1983}
M.~Fortin and M.~Soulie.
\newblock A non-conforming piecewise quadratic finite element on triangles.
\newblock {\em International Journal for Numerical Methods in Engineering},
  19(4):505--520, 1983.

\bibitem{GiraultRaviart1986}
V.~Girault and P.~A. Raviart.
\newblock {\em Finite element methods for Navier-Stokes equations: theory and
  algorithms}, volume~5.
\newblock Springer Science {\&}amp; Business Media, 2012.

\bibitem{J.Hu;Y.Huang;Q.Lin2014}
J.~Hu, Y.~Huang, and Q.~Lin.
\newblock Lower bounds for eigenvalues of elliptic operators: by nonconforming
  finite element methods.
\newblock {\em Journal of Scientific Computing}, 61(1):196--221, 2014.

\bibitem{Hu.J;Shi.Z-C2012}
J.~Hu and Z.~Shi.
\newblock The best {$L^{2}$} norm error estimate of lower order finite element
  methods for the fourth order problem.
\newblock {\em Journal of Computational Mathematics}, 30(5):449--460, 2012.

\bibitem{J.Hu;ZC.Shi2013}
J.~Hu and Z.~Shi.
\newblock A lower bound of the {$L^{2}$} norm error estimate for the {Adini}
  element of the biharmonic equation.
\newblock {\em Siam Journal on Numerical Analysis}, 51(5):2651--2659, 2013.

\bibitem{J.Hu;XQ.Yang;S.Zhang2015}
J.~Hu, X.~Yang, and S.~Zhang.
\newblock Capacity of the {Adini} element for biharmonic equations.
\newblock {\em Journal of Scientific Computing}, 69(3):1366--1383, 2016.

\bibitem{Hu.J;Zhang.S2013}
J.~Hu and S.~Zhang.
\newblock Nonconforming finite element methods on quadrilateral meshes.
\newblock {\em Science China Mathematics}, 56(12):2599--2614, 2013.

\bibitem{Hu.J;Zhang.Sy2015}
J.~Hu and S.~Zhang.
\newblock The minimal conforming {$H^k$} finite element spaces on
  {$\mathbb{R}^n$} rectangular grids.
\newblock {\em Mathematics of Computation}, 84(292):563--579, 2015.

\bibitem{SheenKimLuoMengNamPark2013}
I.~Kim, Z.~Luo, Z.~Meng, H.~Nam, C.~Park, and D.~Sheen.
\newblock A piecewise {$P_{2}$}-nonconforming quadrilateral finite element.
\newblock {\em ESAIM: Mathematical Modelling and Numerical Analysis},
  47(3):689--715, 2013.

\bibitem{Lee.H;Sheen.D2006}
H.~Lee and D.~Sheen.
\newblock A new quadratic nonconforming finite element on rectangles.
\newblock {\em Numerical Methods for Partial Differential Equations},
  22(4):954--970, 2006.

\bibitem{YA.Li2011}
Y.~Li.
\newblock The lower bounds of eigenvalues by the {W}ilson element in any
  dimension.
\newblock {\em Advances in Applied Mathematics and Mechanics}, 3(5):598--610,
  2011.

\bibitem{Q.Lin;JF.Lin2007}
Q.~Lin and J.~Lin.
\newblock {\em Finite Element Methods: Accuracy and Improvement}.
\newblock Science Press, Beijing, 2006.

\bibitem{Lin.Q;Tobiska.L;Aihui.Zhou2005}
Q.~Lin, L.~Tobiska, and A.~Zhou.
\newblock Superconvergence and extrapolation of non-conforming low order finite
  elements applied to the {Poisson} equation.
\newblock {\em IMA Journal of Numerical Analysis}, 25(1):160--181, 2005.

\bibitem{Lin.Q;Xie.H;Xu.J2014}
Q.~Lin, H.~Xie, and J.~Xu.
\newblock Lower bounds of the discretization error for piecewise polynomials.
\newblock {\em Mathematics of Computation}, 83(285):1--13, 2014.

\bibitem{FS.Luo;Q.Lin;Xie2012}
F.~Luo, Q.~Lin, and H.~Xie.
\newblock Computing the lower and upper bounds of {Laplace} eigenvalue problem:
  by combining conforming and nonconforming finite element methods.
\newblock {\em Science China Mathematics}, 55(5):1069--1082, 2012.

\bibitem{XY.Meng;XQ.Yang;S.Zhang2016}
X.~Meng, X.~Yang, and S.~Zhang.
\newblock Convergence analysis of the rectangular {Morley} element scheme for
  second order problem in arbitrary dimensions.
\newblock {\em Science China Mathematics}, 59(11):2245--2264, 2016.

\bibitem{Park.C;Sheen.D2003}
C.~Park and D.~Sheen.
\newblock {$P_{1}$}-nonconforming quadrilateral finite element methods for
  second-order elliptic problems.
\newblock {\em SIAM Journal on Numerical Analysis}, 41(2):624--640, 2003.

\bibitem{Rannacher.R;Turek.S1992}
R.~Rannacher and S.~Turek.
\newblock Simple nonconforming quadrilateral {Stokes} element.
\newblock {\em Numerical Methods for Partial Differential Equations},
  8(2):97--111, 1992.

\bibitem{ZC.Shi1986}
Z.~Shi.
\newblock A remark on the optimal order of convergence of {Wilson's}
  nonconforming element.
\newblock {\em Mathematica Numerica Sinica}, 8(2):159--163, 1986.

\bibitem{Wang.M;Shi.Z2013mono}
Z.~Shi and M.~Wang.
\newblock {\em Finite element methods}.
\newblock Science Press, Beijing, 2013.

\bibitem{Wang.M;Xu.J2012}
M.~Wang and J.~Xu.
\newblock Minimal finite element spaces for $2m$-th-order partial differential
  equations in $\mathbb{R}^n$.
\newblock {\em Mathematics of Computation}, 82(281):25--43, 2012.

\bibitem{Wilson.E;Taylor.R;Doherty.W;Ghaboussi.J1973}
E.~Wilson, R.~Taylor, W.~Doherty, and J.~Ghaboussi.
\newblock Incompatible displacement models.
\newblock In {\em Numerical and Computer Methods in Structural Mechanics},
  pages 43--57. Academic Press, New York, 1973.

\bibitem{Wu.S;Xu.J2019}
S.~Wu and J.~Xu.
\newblock Nonconforming finite element spaces for $2 m$-th order partial
  differential equations on $\mathbb{R}^n$ simplicial grids when $m= n+ 1$.
\newblock {\em Mathematics of Computation}, 88(316):531--551, 2019.

\bibitem{J.Xu1992}
J.~Xu.
\newblock Iterative methods by space decomposition and subspace correction.
\newblock {\em SIAM Review}, 34(4):581--613, 1992.

\bibitem{J.Xu;A.Zhou2001}
J.~Xu and A.~Zhou.
\newblock A two-grid discretization scheme for eigenvalue problems.
\newblock {\em Mathematics of Computation}, 70(233):17--25, 2001.

\bibitem{YiDU.Yang;Zhen.Chen2008}
Y.~Yang and Z.~Chen.
\newblock The order-preserving convergence for spectral approximation of
  self-adjoint completely continuous operators.
\newblock {\em Science in China Series A: Mathematics}, 51(7):1232--1242, 2008.

\bibitem{Y.Yang;J.Han;H.Bi;Y.Yu2015}
Y.~Yang, J.~Han, H.~Bi, and Y.~Yu.
\newblock The lower/upper bound property of the {Crouzeix--Raviart} element
  eigenvalues on adaptive meshes.
\newblock {\em Journal of Scientific Computing}, 62(1):284--299, 2015.

\bibitem{Y.Yang;Z.Zhang2010}
Y.~Yang, Z.~Zhang, and F.~Lin.
\newblock Eigenvalue approximation from below using non-conforming finite
  elements.
\newblock {\em Science in China Series A: Mathematics}, 53(1):137--150, 2010.

\bibitem{Zhang.S2016nm}
S.~Zhang.
\newblock Stable finite element pair for {Stokes} problem and discrete stokes
  complex on quadrilateral grids.
\newblock {\em Numerische Mathematik}, 133(2):371--408, 2016.

\bibitem{Shuo.Zhang2017}
S.~Zhang.
\newblock Minimal consistent finite element space for the biharmonic equation
  on quadrilateral grids.
\newblock {\em IMA Journal of Numerical Analysis}, 2019.

\bibitem{ZM.Zhang;YD.Yang2007}
Z.~Zhang, Y.~Yang, and C.~Zhen.
\newblock Eigenvalue approximation from below by {Wilson's} element.
\newblock {\em Mathematica Numerica Sinica}, 29(3):319--321, 2007.

\end{thebibliography}
\bibliographystyle{abbrv}

\appendix

\section{Construction of  basis functions for the moment-continuous (MC) element space}
\label{sec:appA}
\noindent Let $\Omega\subset\mathbb{R}^2$ be a polygonal region subdivided into a rectangular grid $\mathcal{G}_h$. Define the moment-continuous  (MC) element spaces as
\begin{align}
V_h^{\rm MC}:= \Big\{ w_h\in L^2(\Omega):  w_h|_K\in P_2(K), \text{ and } w_h\ \text{is\ moment-continuous}\ \text{on}\ \mathcal{G}_h \Big\};
\end{align}
%%%%%%%
\begin{align}
V_{h0}^{\rm MC}:= \Big\{w_h\in V_h^{\rm MC}:  w_h\ \text{is\ moment-homogeneous\ on}\ \mathcal{G}_h \Big\}.
\end{align}
They have the following equivalent definitions:
\begin{equation}
\begin{split}
V_h^{\rm MC}:= \Big\{ w_h\in L^2(\Omega): \ w_h|_K\in P_2(K), \text{ and } w_h\  & \text{is\ continuous at second Gauss--Legendre}
\\
 & \text{ points of any }\  e \in \mathcal{E}_h \Big\};
\end{split}
\end{equation}
%%%%%%
\begin{equation}
V_{h0}^{\rm MC}:= \Big\{w_h\in V_h^{\rm MC}:\ w_h\ \text{vanishes at second}  \text{ Gauss--Legendre points on any } e \in \mathcal{E}_h^{b}\Big\}.
\end{equation}
\noindent {In this section, we will present available sets of basis functions of $V_{h}^{\rm MC}$ and $V_{h0}^{\rm MC}$. }

\subsection{Compatibility conditions}
Let $K$ be a rectangle with $a_i$ the vertices and $g_{ij}$ the Gauss--Legendre points on the boundary (see Figure \ref{fig:APprect}), where $i=1,2,3,4$ and $j=1,2$. Let $\theta_1=\frac{1}{2}(1-\sqrt{\frac{1}{3}})$
and $\theta_2=\frac{1}{2}(1+\sqrt{\frac{1}{3}})$ be the coordinates of second-order Gauss-Legendre points on $[0,1]$. By a pure linear algebra argument, we have the following description of $P_2(K)$.

\begin{figure}[htbp]
\setlength{\unitlength}{0.75mm}

\begin{picture}(90,70)(0,-5)
\thicklines\color{black}

\put(0,0){\line(0,1){50}}
\put(0,0){\line(1,0){80}}
\put(0,50){\line(1,0){80}}
\put(80,0){\line(0,1){50}}

%50*0.788675134594813=39.433755
%80*0.788675134594813=63.094008
%50*0.211324865405187=10.566245
%80*0.211324865405187=16.905984

\put(-1,10.566245){$\bullet$}
\put(-1,39.433755){$\bullet$}
\put(79,10.566245){$\bullet$}
\put(79,39.433755){$\bullet$}
\put(16.905984,-1){$\bullet$}
\put(63.094008,-1){$\bullet$}
\put(16.905984,49){$\bullet$}
\put(63.094008,49){$\bullet$}

\put(-6,-5){{$a$}{${}_1$} {$(0,0)$}}
\put(80,-5){{$a$}{${}_2$} {$(\xi,0)$}}
\put(80,53){{$a$}{${}_3$} {$(\xi,\eta)$}}
\put(-6,53){{$a$}{${}_4$} {$(0,\eta)$}}

\put(-1,-1){$\bullet$}
\put(79,-1){$\bullet$}
\put(-1,49){$\bullet$}
\put(79,49){$\bullet$}

\put(-10.2,10.566245){{\Large$g_{42}$}}
\put(-10.2,39.433755){\Large$g_{41}$}
\put(82,10.566245){\Large$g_{21}$}
\put(82,39.433755){\Large$g_{22}$}
\put(16.905984,-5){\Large$g_{11}$}
\put(63.094008,-5){\Large$g_{12}$}
\put(16.905984,53){\Large$g_{32}$}
\put(63.094008,53){\Large$g_{31}$}

\dashline{1}(0,0)(80,50)

%\put(0,0){\line(1,0.625){5}}
%\put(0,0){\line(0.625){50}}
%\put(50,67){$\bullet$}

\put(16.906984,10.566245){$\bullet$}
\put(13.906984,14.566245){\large$g_{131}$}
\put(63.094008,39.433755){$\bullet$}
\put(65.094008,37.433755){\large$g_{132}$}
\end{picture}

\caption{Illustration of Gauss points on the boundary of a rectangle.}\label{fig:APprect}

\end{figure}
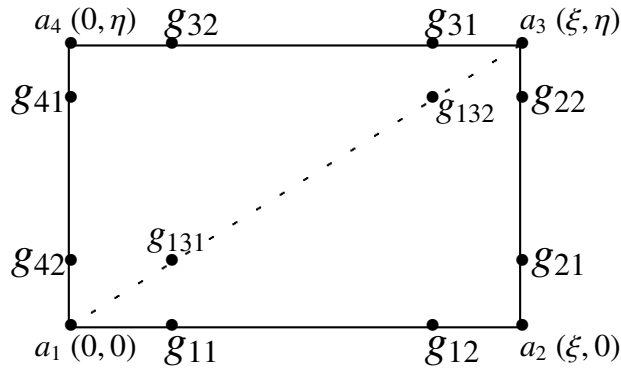

\begin{lemma}\label{lem:consisMC}
Given $\mathbb{g}_{ij}\in\mathbb{R}$, $i=1,2,3,4$, $j=1,2$. There exists $p\in P_2(K)$, such that $p(g_{ij})=\mathbb{g}_{ij}$, if and only if the following compatibility conditions are satisfied on $K$,
\begin{eqnarray}
\mathbb{g}_{11}-\mathbb{g}_{12}+\mathbb{g}_{21}-\mathbb{g}_{22}+\mathbb{g}_{31}-\mathbb{g}_{32}+\mathbb{g}_{41}-\mathbb{g}_{42}=0;\label{eq:cond1}\\
\theta_1(\mathbb{g}_{11}-\mathbb{g}_{32})+\theta_2(\mathbb{g}_{31}-\mathbb{g}_{12})+(\mathbb{g}_{21}-\mathbb{g}_{22})=0;\label{eq:cond2} \\
\theta_1(\mathbb{g}_{11}-\mathbb{g}_{12})+(\theta_2-\theta_1)(\mathbb{g}_{22}-\mathbb{g}_{41})+\theta_2(\mathbb{g}_{32}-\mathbb{g}_{31})=0.\label{eq:cond3}
\end{eqnarray}
\end{lemma}
\begin{proof}
To prove \eqref{eq:cond1}, we connect the vertices $a_1$ and $a_3$. Let $g_{131}$ and $g_{132}$ be the two Gauss points on $a_1a_3$. Then we obtain that (see \cite[(2)]{Fortin.M;Soulie.M1983})
$$
p(g_{11})-p(g_{12})+p(g_{21})-p(g_{22})+p(g_{132})-p(g_{131})=0
$$
and
$$
p(g_{31})-p(g_{32})+p(g_{41})-p(g_{42})+p(g_{131})-p(g_{132})=0.
$$
Thus \eqref{eq:cond1} follows. Since the two directional derivatives of $p$ belongs to $P_1(K)$, \eqref{eq:cond2} and \eqref{eq:cond3} hold.
\end{proof}

A set of basis functions of $P_2(K)$ are listed below. Note that with $\varphi_{a_i}$ being the bilinear functions and $\varphi_{xx}$ and $\varphi_{yy}$ being the bubbles in two directions, these six functions form a set of basis functions with respect to the Wilson element.

\noindent{\textbf{Local basis functions of $P_{2}(K)$:}} \begin{equation}\label{eq:basisW}
\left\{
\begin{array}{l}
\varphi_{a_1,K}= \frac{1}{\xi\eta}(\xi-x)(\eta-y),\\
\varphi_{a_2,K}= \frac{1}{\xi\eta}x(\eta-y),\\
\varphi_{a_3,K}= \frac{1}{\xi\eta}xy,\\
\varphi_{a_4,K}= \frac{1}{\xi\eta}(\xi-x)y,\\
\varphi_{xx,K}= x(\xi-x),\\
\varphi_{yy,K}= y(\eta-y).
\end{array}
\right.
\end{equation}

\begin{lemma}\label{lem:ght}
Let $p\in P_2(K)$. The following results hold.
\begin{enumerate}
\item If $p(g_{41})=p(g_{42})=0$, then $p\in {\rm span}\{\varphi_{a_2},\varphi_{a_3},\varphi_{xx},\varphi_{0,K}\};$ %The similar case holds when $p$ vanishes on the Gauss points of other edges.
\item If $p(g_{11})=p(g_{12})=p(g_{41})=p(g_{42})=0$, then $p\in {\rm span}\{\varphi_{a_3},\varphi_{0,K}\};$ %The similar case holds with respect to other vertices.
\item If $p(g_{11})=p(g_{12})=p(g_{21})=p(g_{22})=p(g_{41})=p(g_{42})=0$, then $p\in {\rm span}\{\varphi_{0,K}\}$.
\end{enumerate}
\end{lemma}

\begin{remark}\label{rem:P2bubble}{\rm
A polynomial $p\in P_2(K)$ is uniquely determined by its evaluations on $g_{ij}$ only up to $\varphi_{0,K}(x,y):= \frac{1}{\xi^{2}}\varphi_{xx,K} + \frac{1}{\eta^{2}}\varphi_{yy,K} - 2\theta_1\theta_2$ multiplicated by a constant.}
\end{remark}

\subsection{Patterns of supports of basis functions in $V_{h}^{\rm MC}$ and $V_{h0}^{\rm MC}$}
Suppose that the domain is divided into $m\times n$ rectangles; see Figure \ref{fig:domain and grid}. In $x$ direction, it is decomposed to $m$ rows, each being $\mathcal{G}_{i}$ ($1\leqslant i\leqslant m$), and in $y$ direction, it is decomposed to $n$ columns, each being $\mathcal{G}^{j}$ ($1\leqslant j\leqslant n$). The vertices are labeled by $X_i^j$, and the cells by $K^j_i$. That is, $K^j_i=\mathcal{G}_{i}\cap \mathcal{G}^{j}$, and it has four vertices as $X^{j-1}_{i-1}$, $X^{j-1}_{i}$, $X^{j}_{i-1}$, and $X^{j}_{i}$.
%%%%%%%%%%
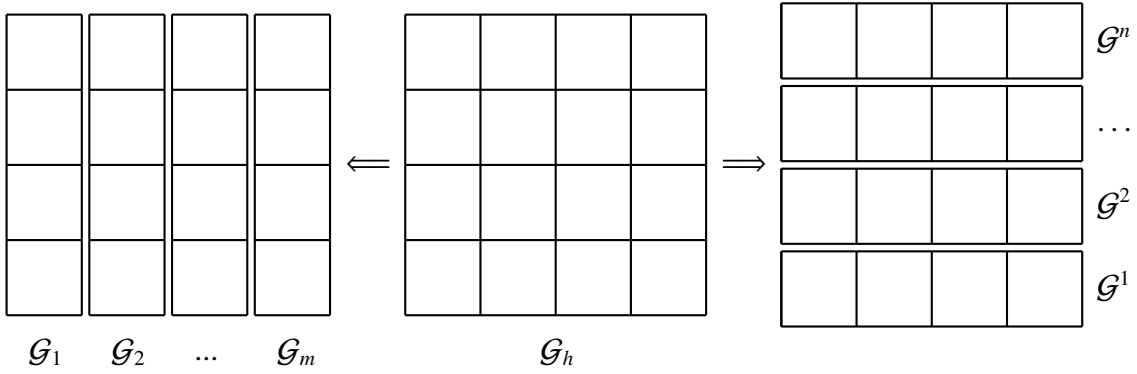
\begin{figure}[htbp]
\setlength{\unitlength}{1mm}

\begin{picture}(150,50)(-15,-5)
\thicklines\color{black}

\put(30,0){\line(0,1){40}}
\put(20,0){\line(0,1){40}}
\put(19,0){\line(0,1){40}}
\put(9,0){\line(0,1){40}}
\put(8,0){\line(0,1){40}}
\put(-2,0){\line(0,1){40}}
\put(-3,0){\line(0,1){40}}
\put(-13,0){\line(0,1){40}}
\put(-10,-6){$\mathcal{G}_1$}
\put(1,-6){$\mathcal{G}_2$}
\put(12,-6){$...$}
\put(23,-6){$\mathcal{G}_m$}

\put(-13,0){\line(1,0){10}}
\put(-2,0){\line(1,0){10}}
\put(9,0){\line(1,0){10}}
\put(20,0){\line(1,0){10}}

\put(-13,10){\line(1,0){10}}
\put(-2,10){\line(1,0){10}}
\put(9,10){\line(1,0){10}}
\put(20,10){\line(1,0){10}}

\put(-13,20){\line(1,0){10}}
\put(-2,20){\line(1,0){10}}
\put(9,20){\line(1,0){10}}
\put(20,20){\line(1,0){10}}

\put(-13,30){\line(1,0){10}}
\put(-2,30){\line(1,0){10}}
\put(9,30){\line(1,0){10}}
\put(20,30){\line(1,0){10}}

\put(-13,40){\line(1,0){10}}
\put(-2,40){\line(1,0){10}}
\put(9,40){\line(1,0){10}}
\put(20,40){\line(1,0){10}}

\put(32,19){$\Longleftarrow$}

%\put(36,14){$\Gamma_1$}
%\put(81,13.5){$\Gamma_3$}
%\put(65,-4){$\Gamma_2$}
%\put(65,41){$\Gamma_4$}
\put(40,0){\line(1,0){40}}
\put(40,10){\line(1,0){40}}
\put(40,20){\line(1,0){40}}
\put(40,30){\line(1,0){40}}
\put(40,40){\line(1,0){40}}
\put(40,0){\line(0,1){40}}
\put(50,0){\line(0,1){40}}
\put(60,0){\line(0,1){40}}
\put(70,0){\line(0,1){40}}
\put(80,0){\line(0,1){40}}

\put(58,-6){$\mathcal{G}_h$}

\put(82,19){$\Longrightarrow$}

\put(90,-1.5){\line(1,0){40}}
\put(90,8.5){\line(1,0){40}}
\put(90,9.5){\line(1,0){40}}
\put(90,19.5){\line(1,0){40}}
\put(90,20.5){\line(1,0){40}}
\put(90,30.5){\line(1,0){40}}
\put(90,31.5){\line(1,0){40}}
\put(90,41.5){\line(1,0){40}}

\put(132,2.5){$\mathcal{G}^1$}
\put(132,13.5){$\mathcal{G}^2$}
\put(132,24.5){$\dots$}
\put(132,35.5){$\mathcal{G}^n$}

\put(90,-1.5){\line(0,1){10}}
\put(90,9.5){\line(0,1){10}}
\put(90,20.5){\line(0,1){10}}
\put(90,31.5){\line(0,1){10}}

\put(90,-1.5){\line(0,1){10}}
\put(90,9.5){\line(0,1){10}}
\put(90,20.5){\line(0,1){10}}
\put(90,31.5){\line(0,1){10}}

\put(100,-1.5){\line(0,1){10}}
\put(100,9.5){\line(0,1){10}}
\put(100,20.5){\line(0,1){10}}
\put(100,31.5){\line(0,1){10}}

\put(110,-1.5){\line(0,1){10}}
\put(110,9.5){\line(0,1){10}}
\put(110,20.5){\line(0,1){10}}
\put(110,31.5){\line(0,1){10}}

\put(120,-1.5){\line(0,1){10}}
\put(120,9.5){\line(0,1){10}}
\put(120,20.5){\line(0,1){10}}
\put(120,31.5){\line(0,1){10}}

\put(130,-1.5){\line(0,1){10}}
\put(130,9.5){\line(0,1){10}}
\put(130,20.5){\line(0,1){10}}
\put(130,31.5){\line(0,1){10}}

\end{picture}%

\caption{Illustration of the domain and the grid.} \label{fig:domain and grid}
%\end{center}
\end{figure}

We call the support set of a basis function a pattern. Below we present four kinds of patterns sequentially, namely, cell patterns in Lemma \ref{lem:cellpatMC}, vertex patterns in Lemma \ref{lem:verpatMC}, column patterns and row patterns in Lemma \ref{lem:colrowpatMC}. 
\begin{lemma}\label{lem:cellpatMC}
Let $\varphi_{0,K_{i}^{j}}$ be defined in Remark  \ref{rem:P2bubble} on $K_{i}^{j}$, $1\leqslant i\leqslant m$, $1\leqslant j\leqslant n$. Then,  $\varphi_{0,K_{i}^{j}} \in V_{h0}^{MC}$.
\end{lemma}
\begin{lemma}\label{lem:verpatMC}
 Let $\varphi_{X_{i}^{j}}$ denote a function defined on the support patch $\omega$ associated with $X_{i}^{j}$, which is bilinear on every element in $\omega$. Then, $\varphi_{X_{i}^{j}} \in V_{h}^{MC}$, $\forall X_{i}^{j} \in \mathcal{N}_{h}$, and $\varphi_{X_{i}^{j}} \in V_{h0}^{MC}$, $\forall X_{i}^{j} \in \mathcal{N}_{h}^{i}$.
 \end{lemma}
\begin{lemma}\label{lem:colrowpatMC}
 Let $\omega$ be a patch with boundaries $\Gamma_{l}$ $(l=1:4)$, anticlockwise; see Figure \ref{fig:APglobalMC}. Let $V_{h}^{\rm MC}(\omega)$ denote a moment-continuous space restricted on $\omega$. 

\noindent{(a)} Let $\omega = \mathcal{G}_{i} $ be the union of elements in the $i$-th column; see Figure \ref{fig:APglobalMC} (Left).  We define 
\begin{equation*}
 V_{{\rm col},i}^{\rm MC} := \Big\{ v_{h} \in V_{h}^{\rm MC}(\omega); \ v_{h} \text{ vanishes at } \text{all Gauss-Legendre points on } e \subset \Gamma_{l}\ (l = 1,3)\Big\}.
\end{equation*}
Let $\varphi_{i}^{x}$ be a function defined on $\omega$, which is equal to $\varphi_{xx,K_{i}^{j}}$ on every $K_{i}^{j}\in \omega$. Then $\varphi_{i}^{x} \in V_{{\rm col},i}^{\rm MC}$, and furthermore,  $V_{{\rm col},i}^{\rm MC} = {\rm{span}} \big\{\varphi_{i}^{x}\big\} \oplus {\rm{span}}\big
\{\varphi_{0,K_{i}^{j}}\big\}_{j=1}^{n}$.

\noindent{(b)} Let $\omega : = \mathcal{G}^{j} $ be the union of elements in the $j$-th row; see Figure \ref{fig:APglobalMC} (Right). We define 
\begin{equation*}
 V_{{\rm row},j}^{\rm MC} := \Big\{ v_{h} \in V_{h}^{\rm MC}(\omega); \ v_{h} \text{ vanishes at } \text{all Gauss-Legendre points on } e \subset \Gamma_{l} \ (l = 2,4)\Big\}
\end{equation*}

\noindent Let $\varphi_{j}^{y}$ be a function defined on $\omega$, which is equal to $\varphi_{yy,K_{i}^{j}}$ on every $K_{i}^{j}\in \omega$. Then $\varphi_{j}^{y} \in V_{{\rm row},j}^{\rm MC}$, and furthermore, $V_{{\rm row},j}^{\rm MC} = {\rm{span}} \big\{\varphi_{j}^{y} \big\} \oplus {\rm{span}} \big\{ \varphi_{0,K_{i}^{j}} \big\}_{i=1}^{m}$.
\end{lemma}
\begin{proof}
We present the proof of (a), and omit the proof of (b), which can be obtained similarly. Suppose that $V_{{\rm col},i}^{\rm MC} \ne \varnothing $. Let $v_{h}$ be a function in $V_{{\rm col},i}^{\rm MC}$. Denote by $\delta_{i}$ the value of $v_{h}$ on a Gauss point  in Figure \ref{fig:APglobalMC}. According to Lemma \ref{lem:consisMC}, it holds on the element $K_{i}^{1}$ that
\begin{align}
\delta_{1} - \delta_{2} +\delta_{4} - \delta_{3} = 0;\label{eq:A1}
\\ 
\theta_{1}(\delta_{1} - \delta_{3}) + \theta_{2}(\delta_{4} - \delta_{2}) = 0;\label{eq:A2}
\\ 
\theta_{1}(\delta_{1} - \delta_{2}) + \theta_{2}(\delta_{3} - \delta_{4}) = 0. \label{eq:A3}
\end{align}
It follows from \eqref{eq:A1} - \eqref{eq:A3} that $ \delta_{1} = \delta_{2} = \delta_{3} = \delta_{4}.$
Apply Lemma \ref{lem:consisMC}  to $v_{h}$ on $K_{i}^{2}$, and we obtain
$\delta_{3} = \delta_{4} = \delta_{5} = \delta_{6}.$
Repeat the process for the whole row, and we have $\delta_{1} = \delta_{2} = \cdots = \delta_{2n+1} = \delta_{2n+2}.$ By definition, it is obvious that $\varphi_{i}^{x} \in V_{{\rm col},i}^{\rm MC}$. From Remark \ref{rem:P2bubble}, we derive that
$V_{{\rm col},i}^{\rm MC} = \text{span}  \big\{\varphi_{i}^{x} \big\} \oplus \text{span} \big\{ \varphi_{0,K_{i}^{j}} \big\}_{j=1}^{n}$.
 \end{proof}
%%%%%%%%
\begin{figure}[!htbp]
\setlength{\unitlength}{1.4mm}
\begin{picture}(90,46)(10,-5)
\thicklines\color{black}
\put(20,0){\line(1,0){10}}
\put(20,10){\line(1,0){10}}
\put(20,20){\line(1,0){10}}
\put(20,30){\line(1,0){10}}
\put(20,40){\line(1,0){10}}
%\put(20,0){\line(0,1){40}}
%\put(30,0){\line(0,1){40}}
\dashline{1}(20,0)(20,40)
\dashline{1}(30,0)(30,40)
%\put(23.5,-3){$\mathcal{G}_i$}
\put(24,-4){$\Gamma_2$}
\put(24,44){$\Gamma_4$}
\put(15,20){$\Gamma_1$}
\put(32,20){$\Gamma_3$}
%\put(31.5,4){$H_{1}$}
%\put(31.5,14){$H_{2}$}
%\put(31.5,25){$...$}
%\put(31.5,34){$H_{n}$}

% Gauss points
\put(22,-0.6){$\bullet$}
\put(27,-0.6){$\bullet$}
\put(22,9.4){$\bullet$}
\put(27,9.4){$\bullet$}
\put(22,19.4){$\bullet$}
\put(27,19.4){$\bullet$}
\put(22,29.4){$\bullet$}
\put(27,29.4){$\bullet$}
\put(22,39.4){$\bullet$}
\put(27,39.4){$\bullet$}
% K_{i}^{j}
\put(24,4){$K_{i}^{1}$}
\put(24,14){$K_{i}^{2}$}
\put(24,24){$...$}
\put(24,34){$K_{i}^{n}$}

\put(22,1){$\delta_{1}$}
\put(27,1){$\delta_{2}$}
\put(22,11){$\delta_{3}$}
\put(27,11){$\delta_{4}$}
\put(22,21){$\delta_{5}$}
\put(27,21){$\delta_{6}$}
\put(20,41.5){$\delta_{2n+1}$}
\put(27,41.5){$\delta_{2n+2}$}

\dashline{1}(50,10)(90,10)
\dashline{1}(50,20)(90,20)
\put(50,10){\line(0,1){10}}
\put(60,10){\line(0,1){10}}
\put(70,10){\line(0,1){10}}
\put(80,10){\line(0,1){10}}
\put(90,10){\line(0,1){10}}

% Gauss points
\put(49.5,12){$\bullet$}
\put(49.5,17){$\bullet$}
\put(59.5,12){$\bullet$}
\put(59.5,17){$\bullet$}
\put(69.5,12){$\bullet$}
\put(69.5,17){$\bullet$}
\put(79.5,12){$\bullet$}
\put(79.5,17){$\bullet$}
\put(89.5,12){$\bullet$}
\put(89.5,17){$\bullet$}
\put(54,14){$K_{1}^{j}$}
\put(64,14){$K_{2}^{j}$}
\put(74,14){$...$}
\put(84,14){$K_{m}^{j}$}

%\put(53.5,6){$L_{1}$}
%\put(63.5,6){$L_{2}$}
%\put(73.5,6){$...$}
%\put(83.5,6){$L_{m}$}
%\put(91.5,14){$G^{j}$}
\put(46,14){$\Gamma_{1}$}
\put(92,14){$\Gamma_{3}$}
\put(69,22){$\Gamma_{4}$}
\put(69,6){$\Gamma_{2}$}
\end{picture}
\caption{Illustration of the column patterns and row patterns of basis functions in $V_{h}^{\text{MC}}$. These functions vanishes at the Gauss-Legendre points at the dotted lines.}\label{fig:APglobalMC}
\end{figure}
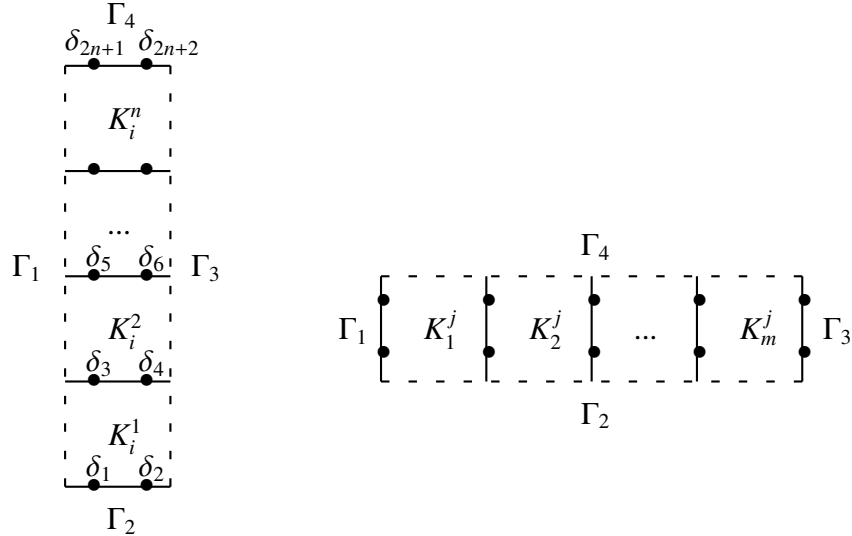
%%%%%%
Here and throughout this paper, we do not distinct $\varphi_{0,K_{i}^{j}}$, $\varphi_{X_{i}^{j}}$, $\varphi_{i}^{x}$, $\varphi_{j}^{y}$ and their respective extension onto the whole domain by zero. Thus we also obtain $\varphi_{i}^{x}\in V_{h}^{\rm MC}$ and $\varphi_{j}^{y} \in V_{h}^{\rm MC}$.

\subsection{Structure of the MC element space}
Here we will present the construction of basis functions in spaces $V_{h0}^{\rm MC}$ and $V_{h}^{\rm MC}$.

\begin{theorem}\label{thm:hombcMC}
Let $\mathcal{G}_{h}$ be a $m \times n$ rectangular subdivision of $\Omega$. Then, $V_{h0}^{\rm MC}=V^{\rm BL}_{h0}\oplus {\rm{span}}\{\varphi_{0,K}\}_{K\in\mathcal{G}_h}.$
\end{theorem}
\begin{proof}
It is obvious that $V^{\rm BL}_{h0}\cap \text{span}\{\varphi_{0,K}\}_{K\in\mathcal{G}_h}= 0$. We only have to show that $V_{h0}^{\rm MC}\subset V^{\rm BL}_{h0}\oplus \text{span}\{\varphi_{0,K}\}_{K\in\mathcal{G}_h}$, i.e., any function in the former is the combination of functions in the latter.
%%%%%%%%%%%%
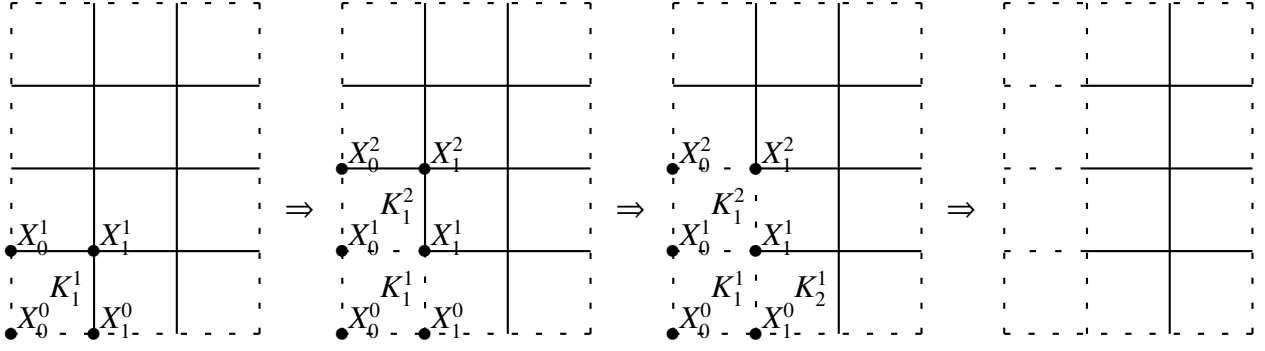
\begin{figure}[htbp]
\setlength{\unitlength}{1.1mm}

\begin{picture}(150,45)(-1,-1)%???????x?y???
\thicklines\color{black}%?????0.8pt

%\put(0,0){\line(1,0){30}}
\dashline{1}(0,0)(0,40)
\dashline{1}(0,40)(30,40)
\dashline{1}(30,0)(30,40)
\put(0,10){\line(1,0){30}}
\put(0,30){\line(1,0){30}}
\put(0,20){\line(1,0){30}}
%\put(0,40){\line(1,0){30}}
%\put(0,0){\line(0,1){40}}
\dashline{1}(0,0)(30,0)
\put(10,0){\line(0,1){40}}
%\put(30,0){\line(0,1){40}}
\put(20,0){\line(0,1){40}}
\put(-1,-1){$\bullet$}
\put(9,-1){$\bullet$}
\put(-1,9){$\bullet$}
\put(9,9){$\bullet$}
\put(0.7,0.7){$X_0^0$}
\put(10.7,0.7){$X^0_1$}
\put(0.7,10.7){$X^1_0$}
\put(10.7,10.7){$X_1^1$}
\put(4.5,4.5){$K_1^1$}

\put(33,14){$\Rightarrow$}

\dashline{1}(40,0)(70,0)
\dashline{1}(40,0)(40,40)
\dashline{1}(50,10)(50,0)
\dashline{1}(40,10)(50,10)
\dashline{1}(40,40)(70,40)
\dashline{1}(70,0)(70,40)
%\put(50,0){\line(1,0){20}}
\put(50,10){\line(1,0){20}}
%\put(40,10){\line(0,1){30}}
\put(40,20){\line(1,0){30}}
\put(40,30){\line(1,0){30}}
%\put(40,40){\line(1,0){30}}
\put(50,10){\line(0,1){30}}
\put(60,0){\line(0,1){40}}
%\put(70,0){\line(0,1){40}}
\put(44.5,4.5){$K_1^1$}
\put(44.5,14.5){$K_1^2$}
\put(39,-1){$\bullet$}
\put(49,-1){$\bullet$}
\put(39,9){$\bullet$}
\put(49,9){$\bullet$}
\put(39,19){$\bullet$}
\put(49,19){$\bullet$}
\put(40.7,0.7){$X_0^0$}
\put(50.7,0.7){$X^0_1$}
\put(40.7,10.7){$X^1_0$}
\put(50.7,10.7){$X_1^1$}
\put(40.7,20.7){$X_0^2$}
\put(50.7,20.7){$X^2_1$}

\put(73,14){$\Rightarrow$}

\dashline{1}(80,0)(80,40)
\dashline{1}(90,0)(90,20)
\dashline{1}(80,0)(110,0)
\dashline{1}(80,10)(90,10)
\dashline{1}(80,20)(90,20)
\dashline{1}(80,40)(110,40)
\dashline{1}(110,0)(110,40)
%\dashline{0.9}(80,30)(90,30)
%\put(80,0){\line(1,0){30}}
%\put(80,10){\line(1,0){30}}
%\put(80,30){\line(1,0){30}}
%\put(80,20){\line(1,0){30}}
%\put(80,0){\line(0,1){30}}
%\put(90,0){\line(1,0){20}}
\put(90,10){\line(1,0){20}}
\put(90,20){\line(1,0){20}}
%\put(80,20){\line(0,1){20}}
\put(80,30){\line(1,0){30}}
%\put(80,40){\line(1,0){30}}
\put(90,20){\line(0,1){20}}
\put(100,0){\line(0,1){40}}
%\put(110,0){\line(0,1){40}}
\put(84.5,4.5){$K_1^1$}
\put(94.5,4.5){$K_2^1$}
%\put(44.5,4.5){$T_1^1$}
\put(84.5,14.5){$K_1^2$}
%\put(84.5,24.5){$...$}
\put(79,-1){$\bullet$}
\put(89,-1){$\bullet$}
\put(79,9){$\bullet$}
\put(89,9){$\bullet$}
\put(79,19){$\bullet$}
\put(89,19){$\bullet$}
\put(80.7,0.7){$X_0^0$}
\put(90.7,0.7){$X^0_1$}
\put(80.7,10.7){$X_0^1$}
\put(90.7,10.7){$X_1^1$}
\put(80.7,20.7){$X^2_0$}
\put(90.7,20.7){$X^2_1$}

\put(113,14){$\Rightarrow$}

\dashline{1}(120,0)(150,0)
\dashline{1}(120,0)(120,40)
\dashline{1}(130,0)(130,40)
\dashline{1}(120,10)(130,10)
\dashline{1}(120,20)(130,20)
\dashline{1}(120,30)(130,30)
\dashline{1}(120,40)(150,40)
\dashline{1}(150,0)(150,40)
%\dashline{1}(140,0)(140,10)
%\dashline{1}(150,0)(150,10)
%\put(120,0){\line(1,0){30}}
%\put(120,10){\line(1,0){30}}
\put(130,30){\line(1,0){20}}
%\put(120,20){\line(1,0){30}}
\put(140,0){\line(0,1){40}}
\put(130,20){\line(1,0){20}}
%\put(150,0){\line(0,1){40}}
%\put(140,0){\dashline{30}(0,1)}
%\put(130,0){\line(1,0){20}}
\put(130,10){\line(1,0){20}}
%\put(130,40){\line(1,0){20}}
%\put(139,18.59){$\mathcal{G}_h$}

\end{picture}%$\Longrightarrow$

\caption{ Elimination of one column.} %\label{tu:DESlunhanshu}
%\end{center}
\end{figure}\label{fig:APhalfbc}
%%%%%%%%%%%%%%%%%%%%%%%%%%%%%
Here we use a sweeping procedure.
Let $v_h\in V_{h0}^{\rm MC}$. First, by Lemma \ref{lem:ght}, we have that $v_h|_{K_1^1}=\alpha_1^1\cdot \varphi_{X_1^1}|_{K_1^1}+\gamma_1^1\cdot\varphi_{0,K_1^1}$ with some constants $\alpha_1^1$ and $\gamma_1^1$. Therefore, $v_h=v_h^{1,1}+\alpha_1^1\cdot\varphi_{X_1^1}|_{K_1^1}+\gamma_1^1\cdot\varphi_{0,K_1^1}$ with $v_h^{1,1}\in V_{h0}^{\rm MC}$ and $v_h^{1,1}$ vanishing on $K_1^1$. Second, $v_h^{1,1}=v_h^{1,2}+\alpha_1^2\cdot\varphi_{X_1^2}+\gamma_1^2\varphi_{0,K_1^2}$ with $v_h^{1,2}\in V_{h0}^{\rm MC}$ and $v_h^{1,2}$ vanishing on $K_1^1$ and $K_1^2$. Furthermore, repeat this process on all the cells of the first column, and we obtain that
$$
v_h=v_h^1+\sum_{j=1}^{n-1} \alpha_1^j\varphi_{X_1^j}+\sum_{j=1}^n \gamma_1^j\varphi_{0,K_1^j},
$$
where $v_h^1\in V_{h0}^{\rm MC}$ and $v_h^1$ vanishes on the whole column $\mathcal{G}_1$.
Finally, we repeat the process from $\mathcal{G}_1$ to $\mathcal{G}_n$, and obtain that
$$
v_h=\sum_{i=1}^{m-1}\sum_{j=1}^{n-1}\alpha_i^j\varphi_{X_i^j}+\sum_{i=1}^m\sum_{j=1}^n\gamma_i^j\varphi_{0,K_i^j}.
$$
Hence the result.
\end{proof}

\begin{theorem}
Let $\mathcal{G}_{h}$ be a $m \times n$ rectangular subdivision of $\Omega$. Then, we have
$$V_h^{\rm MC}=V^{\rm BL}_h+ ({\rm{span}}\{\varphi_{0,K}\}_{K\in\mathcal{G}_h}\oplus {\rm{span}}\{\varphi_{i}^{x}\}_{i=1}^m\oplus {\rm{span}}\{\varphi_{j}^{y}\}_{j=1}^n).$$
\end{theorem}

\begin{proof}
It is obvious that $V_h^{\rm MC} \supset V^{\rm BL}_h+ (\text{span}\{\varphi_{0,K}\}_{K\in\mathcal{G}_h}\oplus \text{span}\{\varphi_{i}^x\}_{i=1}^m\oplus \text{span}\{\varphi_{j}^y\}_{j=1}^n)$. Here we have noted that, if $\varphi_1+\varphi_2+\varphi_3=0$ with $\varphi_1\in {\rm{span}}\{\varphi_{0,K}\}_{K\in\mathcal{G}_h}$, $\varphi_2\in \text{span}\{\varphi_{i}^x\}_{i=1}^m$, and $\varphi_3\in \text{span}\{\varphi_{j}^y\}_{j=1}^n$, then $\varphi_1=\varphi_2=\varphi_3=0$. We only have to show the other direction. Let $v_h\in V_h^{\rm MC}$.

 First, by \eqref{eq:basisW}, there exists unique constants $\alpha_0^0, \alpha_0^1,\alpha_1^0,\alpha_1^1,\kappa_1$, and $\sigma_1$, such that
%$$
%v_h|_{K_1^1}=\alpha_0^0\varphi_{X_0^0}|_{K_1^1}+\alpha_0^1\varphi_{X_0^1}|_{K_1^1}+\alpha_1^0\varphi_{X_1^0}|_{K_1^1}+\alpha_1^1\varphi_{X_1^1}|_{K_1^1}+\gamma_1^1\varphi_{0,K_1^1}+\kappa_1\varphi_1^x|_{K_1^1}+\sigma_1\varphi_1^y|_{K_1^1},
%$$ remove \gamma_1^1\varphi_{0,K_1^1}
$$
v_h|_{K_1^1}=\alpha_0^0\varphi_{X_0^0}|_{K_1^1}+\alpha_0^1\varphi_{X_0^1}|_{K_1^1}+\alpha_1^0\varphi_{X_1^0}|_{K_1^1}+\alpha_1^1\varphi_{X_1^1}|_{K_1^1}+\kappa_1\varphi_1^x|_{K_1^1}+\sigma_1\varphi_1^y|_{K_1^1}.
$$
%Note that these constants are not necessarily unique with respect to $v_h$. Thus,
%$$ remove \gamma_1^1\varphi_{0,K_1^1}
%v_h=v_h^{1,1}+\alpha_0^0\varphi_{X_0^0}+\alpha_0^1\varphi_{X_0^1}+\alpha_1^0\varphi_{X_1^0}+\alpha_1^1\varphi_{X_1^1}+\gamma_1^1\varphi_{0,K_1^1}+\kappa_1\varphi_1^x+\sigma_1\varphi_1^y,
%$$
Thus, we have
$
v_h=v_h^{1,1}+\alpha_0^0\varphi_{X_0^0}+\alpha_0^1\varphi_{X_0^1}+\alpha_1^0\varphi_{X_1^0}+\alpha_1^1\varphi_{X_1^1}+
\kappa_1\varphi_1^x+\sigma_1\varphi_1^y
$
with $v_h^{1,1}\in V_h^{\rm MC}$ and $v_h^{1,1}|_{K_{1}^1}=0$.
Second, by Lemma \ref{lem:ght}, we have
$
v_h^{1,1}|_{K_1^2}=\alpha_0^2\varphi_{X_0^2}|_{K_1^2}+\alpha_1^2\varphi_{X_1^2}|_{K_1^2}+\gamma_1^2\varphi_{0,K_1^2}+\sigma_2\varphi_2^y|_{K_1^2}.
$
Therefore, we obtain
$
v_h^{1,1}=v_h^{1,2}+\alpha_0^2\varphi_{X_0^2}+\alpha_1^2\varphi_{X_1^2}+\gamma_1^2\varphi_{0,K_1^2}+\sigma_2\varphi_2^y
$
with $v_h^{1,2}\in V_h^{\rm MC}$ and $v_h^{1,2}|_{K_{1}^1\cup K_1^2}=0$.
%%%%%%%%%%%%%
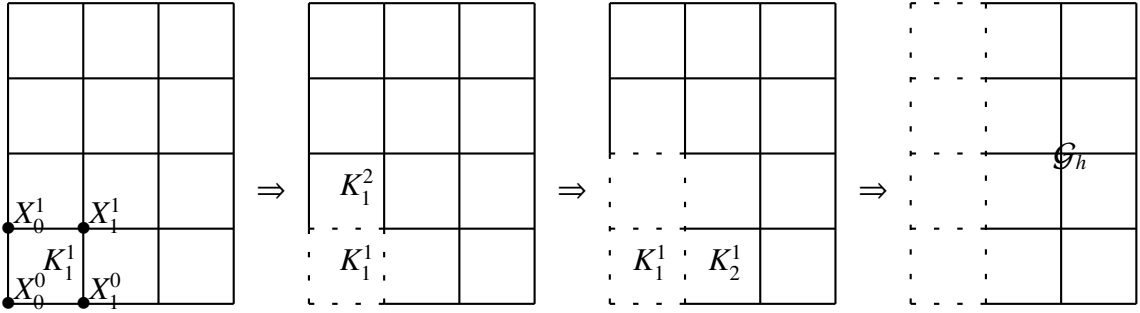
\begin{figure}[htbp]
\setlength{\unitlength}{1mm}

\begin{picture}(150,43)(-1,-1)%???????x?y???
\thicklines\color{black}%?????0.8pt

\put(0,0){\line(1,0){30}}
\put(0,10){\line(1,0){30}}
\put(0,30){\line(1,0){30}}
\put(0,20){\line(1,0){30}}
\put(0,40){\line(1,0){30}}
\put(0,0){\line(0,1){40}}
\put(10,0){\line(0,1){40}}
\put(30,0){\line(0,1){40}}
\put(20,0){\line(0,1){40}}
\put(-1,-1){$\bullet$}
\put(9,-1){$\bullet$}
\put(-1,9){$\bullet$}
\put(9,9){$\bullet$}
\put(0.7,0.7){$X_{0}^{0}$}
\put(10.7,0.7){$X_{1}^{0}$}
\put(0.7,10.7){$X_{0}^{1}$}
\put(10.7,10.7){$X_{1}^{1}$}
\put(4.5,4.5){$K_1^1$}

\put(33,14){$\Rightarrow$}

\dashline{1}(40,0)(50,0)
\dashline{1}(40,0)(40,10)
\dashline{1}(50,10)(50,0)
\dashline{1}(40,10)(50,10)
\put(50,0){\line(1,0){20}}
\put(50,10){\line(1,0){20}}
\put(40,20){\line(1,0){30}}
\put(40,30){\line(1,0){30}}
\put(40,40){\line(1,0){30}}
\put(40,10){\line(0,1){30}}
\put(50,10){\line(0,1){30}}
\put(60,0){\line(0,1){40}}
\put(70,0){\line(0,1){40}}
\put(44,4.5){$K_1^1$}
\put(44,14.5){$K_1^2$}

\put(73,14){$\Rightarrow$}

\dashline{1}(80,0)(80,20)
\dashline{1}(90,0)(90,20)
\dashline{1}(80,0)(90,0)
\dashline{1}(80,10)(90,10)
\dashline{1}(80,20)(90,20)
%\dashline{0.9}(80,30)(90,30)
%\put(80,0){\line(1,0){30}}
%\put(80,10){\line(1,0){30}}
%\put(80,30){\line(1,0){30}}
%\put(80,20){\line(1,0){30}}
%\put(80,0){\line(0,1){30}}
\put(90,0){\line(1,0){20}}
\put(90,10){\line(1,0){20}}
\put(90,20){\line(1,0){20}}
\put(80,30){\line(1,0){30}}
\put(80,40){\line(1,0){30}}
\put(80,20){\line(0,1){20}}
\put(90,20){\line(0,1){20}}
\put(100,0){\line(0,1){40}}
\put(110,0){\line(0,1){40}}
\put(83,4.5){$K_1^1$}
\put(93,4.5){$K_2^1$}

\put(113,14){$\Rightarrow$}

\dashline{1}(120,0)(130,0)
\dashline{1}(120,0)(120,40)
\dashline{1}(130,0)(130,40)
\dashline{1}(120,10)(130,10)
\dashline{1}(120,20)(130,20)
\dashline{1}(120,30)(130,30)
\dashline{1}(120,40)(130,40)
%\dashline{1}(140,0)(140,10)
%\dashline{1}(150,0)(150,10)
%\put(120,0){\line(1,0){30}}
%\put(120,10){\line(1,0){30}}
\put(130,30){\line(1,0){20}}
%\put(120,20){\line(1,0){30}}
\put(140,0){\line(0,1){40}}
\put(130,20){\line(1,0){20}}
\put(150,0){\line(0,1){40}}
%\put(140,0){\dashline{30}(0,1)}
\put(130,0){\line(1,0){20}}
\put(130,10){\line(1,0){20}}
\put(130,40){\line(1,0){20}}
\put(139,18.59){$\mathcal{G}_h$}

\end{picture}%$\Longrightarrow$

\caption{ Elimination of left column.} %\label{tu:DESlunhanshu}
%\end{center}
\end{figure}
%%%%%%%%%%%%%%%%%%%%%%%%%%%%%

\noindent Furthermore, repeat this process on the column $\mathcal{G}_1$, and we obtain
\begin{align}\label{eq:A12}
v_h=v_h^1+\sum_{j=0}^n(\alpha_0^j\varphi_{X_0^j}+\alpha_1^j\varphi_{X_1^j})+(\sum_{j=1}^n \gamma_1^j\varphi_{0,K_1^j}) - \varphi_{0,K_1^1}+\sum_{j=1}^n \sigma_j\varphi_j^y+\kappa_1\varphi_1^x,
\end{align}
where $v_{h}^1\in V_h^{\rm MC}$ and $v_{h}^1|_{\mathcal{G}_{1}}=0$.

%%%%%%%%%%%%%%%%%%%%%%%%%%%%%%%%
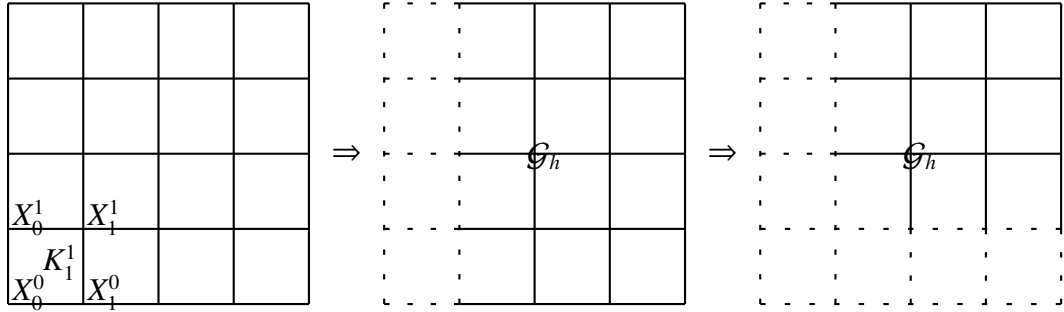
\begin{figure}[htbp]
\setlength{\unitlength}{1mm}

\begin{picture}(150,45)(-1,-1)%???????x?y???
\thicklines\color{black}%?????0.8pt

\put(0,0){\line(1,0){40}}
\put(0,10){\line(1,0){40}}
\put(0,30){\line(1,0){40}}
\put(0,20){\line(1,0){40}}
\put(0,40){\line(1,0){40}}
\put(0,0){\line(0,1){40}}
\put(10,0){\line(0,1){40}}
\put(30,0){\line(0,1){40}}
\put(20,0){\line(0,1){40}}
\put(40,0){\line(0,1){40}}
\put(0.5,0.5){$X_{0}^{0}$}
\put(10.51,0.5){$X_{1}^{0}$}
\put(0.51,10.51){$X_{0}^{1}$}
\put(10.51,10.51){$X_{1}^{1}$}
\put(4.5,4.5){$K_1^1$}

\put(43,19){$\Rightarrow$}

\dashline{1}(50,0)(60,0)
\dashline{1}(50,0)(50,40)
\dashline{1}(60,0)(60,40)
\dashline{1}(50,10)(60,10)
\dashline{1}(50,20)(60,20)
\dashline{1}(50,30)(60,30)
\dashline{1}(50,40)(60,40)
%\dashline{1}(140,0)(140,10)
%\dashline{1}(150,0)(150,10)
%\put(120,0){\line(1,0){30}}
%\put(120,10){\line(1,0){30}}
\put(60,30){\line(1,0){30}}
%\put(120,20){\line(1,0){30}}
\put(70,0){\line(0,1){40}}
\put(60,20){\line(1,0){30}}
\put(80,0){\line(0,1){40}}
\put(90,0){\line(0,1){40}}
%\put(140,0){\dashline{30}(0,1)}
\put(60,0){\line(1,0){30}}
\put(60,10){\line(1,0){30}}
\put(60,40){\line(1,0){30}}
\put(69,18.59){$\mathcal{G}_h$}

\put(93,19){$\Rightarrow$}

\dashline{1}(100,0)(140,0)
\dashline{1}(100,0)(100,40)
\dashline{1}(110,0)(110,40)
\dashline{1}(100,10)(140,10)
\dashline{1}(100,20)(110,20)
\dashline{1}(100,30)(110,30)
\dashline{1}(100,40)(110,40)
\dashline{1}(120,0)(120,10)
\dashline{1}(130,0)(130,10)
\dashline{1}(140,0)(140,10)
%\dashline{1}(140,0)(140,10)
%\dashline{1}(150,0)(150,10)
%\put(120,0){\line(1,0){30}}
%\put(120,10){\line(1,0){30}}
\put(110,30){\line(1,0){30}}
\put(110,20){\line(1,0){30}}
\put(120,10){\line(0,1){30}}
%\put(110,20){\line(1,0){30}}
\put(130,10){\line(0,1){30}}
\put(140,10){\line(0,1){30}}
%\put(140,0){\dashline{30}(0,1)}
%\put(110,0){\line(1,0){30}}
%\put(110,10){\line(1,0){30}}
\put(110,40){\line(1,0){30}}
\put(119,18.59){$\mathcal{G}_h$}

\end{picture}%$\Longrightarrow$

\caption{Elimination of left column and bottom row.} %\label{tu:DESlunhanshu}
%\end{center}
\end{figure}
%%%%%%%%%%%%%%

\noindent Similarly, repeat this process on the row $\mathcal{G}^1$, and we have
\begin{align}\label{eq:A13}
v_h^1=\tilde{v}_h^{1,1}+\sum_{i=2}^m(\alpha_i^0\varphi_{X_i^0}+\alpha_i^1\varphi_{X_i^1})+\sum_{i=2}^m \gamma_i^1\varphi_{0,K_i^1}+\sum_{i=2}^m \kappa_i\varphi_i^x,
\end{align}
with $\tilde{v}_h^{1,1}\in V_{h}^{\rm MC}$, and $v_h^{1,1}|_{\mathcal{G}_1\cup \mathcal{G}^1}=0$.
Finally, by the same technique as used in the proof of Theorem~\ref{thm:hombcMC}, we can prove that
\begin{align}\label{eq:A14}
\tilde{v}_h^{1,1}=\sum_{\substack{2\leqslant i\leqslant m\\ 2\leqslant j\leqslant n}}\alpha_i^j\varphi_{X_i^j}+\sum_{\substack{2\leqslant i\leqslant m\\ 2\leqslant j\leqslant n}}\gamma_i^j\varphi_{0,K_i^j}.
\end{align}
%$$
%v_h^{\mathcal{G}_1\cup \mathcal{G}^1}=\sum c_{i,j}\cdot\varphi_{a_{i,j}} + \sum d_i^j\varphi_{0,T_i^j}.
%$$
A combination of \eqref{eq:A12},  \eqref{eq:A13}, and \eqref{eq:A14} leads to 
%$$
%v_h=\sum_{\substack{0\leqslant i\leqslant m \\ 0\leqslant j\leqslant n}}\alpha_i^j\varphi_{X_i^j}+\sum_{\substack{1\leqslant i\leqslant m \\ 1\leqslant j\leqslant n}}\gamma_i^j\varphi_{0,K_i^j}+\sum_{i=1}^m\kappa_i\varphi_i^x+\sum_{j=1}^n\sigma_j\varphi_j^y.
%$$
$$
v_h=\sum_{\substack{0\leqslant i\leqslant m \\ 0\leqslant j\leqslant n}}\alpha_i^j\varphi_{X_i^j}+(\sum_{\substack{1\leqslant i\leqslant m \\ 1\leqslant j\leqslant n}}\gamma_i^j\varphi_{0,K_i^j}) - \gamma_1^1\varphi_{0,K_1^1}
+\sum_{i=1}^m\kappa_i\varphi_i^x+\sum_{j=1}^n\sigma_j\varphi_j^y.
$$
Hence the result.
\end{proof}

\begin{remark}{\rm
From the above two theorems, it holds that ${\rm dim} (V_{h0}^{\rm MC} )= 2mn-m-n+1$ and ${\rm dim} (V_{h}^{\rm MC} )= 2mn+2m+2n$.}
\end{remark}
\begin{proposition}\label{comparisionMC}
Define $V^{(2)}_h:=\Big\{v_h\in H^1(\Omega): v_h|_K \in  P_2(K), \forall K \in\mathcal{G}_h\Big\}$, and $V^{(2)}_{h0}:=V^{(2)}_h\cap H^1_0(\Omega)$. That is, $V^{(2)}_h$ and $V^{(2)}_{h0}$ are conforming $P_2$ element spaces. Then, 
\begin{enumerate}
\item $V^{(2)}_h=V^{\rm BL}_h+{\rm span}\{\varphi_i^x,\varphi^y_j\}$, and $\ V^{(2)}_{h0}=V^{\rm BL}_{h0};$
\item $V_h^{\rm MC}=V^{(2)}_h+{\rm span}\{\varphi_{0,K}\}_{K\in\mathcal{G}_h}$, and $ V_{h0}^{\rm MC}=V^{(2)}_{h0}+{\rm span}\{\varphi_{0,K}\}_{K\in\mathcal{G}_{h}}.$
\end{enumerate}
\end{proposition}

\begin{proposition}\label{pro:MCWIL}
Let $V_h^{\rm W}$ be the Wilson element space and $V_{h0}^{\rm W}$ be its homogeneous subspace. Then 

\begin{enumerate}
\item $\nabla_hV_h^{\rm MC}\subset \nabla_hV_h^{\rm W}$, and $ \nabla_hV_{h0}^{\rm MC}\subset\nabla_hV_{h0}^{\rm W};$
\item $V_h^{\rm MC}\setminus V_h^{\rm W}\neq \varnothing$, $V_h^{\rm W}\setminus V_h^{\rm MC}\neq \varnothing$,\ $V_{h0}^{\rm MC}\setminus V_{h0}^{\rm W}\neq \varnothing$, and $V_{h0}^{\rm W}\setminus V_{h0}^{\rm MC}\neq \varnothing$.
\end{enumerate}
\end{proposition}

\section{Construction of basis functions for the reduced rectangular Morley (RRM) element space}
\label{sec:appB}
\noindent Let $\Omega \in \mathbb{R}^{2}$ be rectangle domain divided by $m\times n$ rectangles. Define the reduced rectangular Morley element spaces as
\begin{align}
\begin{split}
V_{h}^{\rm{R}} := \Big\{w\in L^{2}(\Omega) : w|_{K} \in P_{2}(K), & \ w_{h}(a)\text{ is continuous at } a\in \mathcal{N}_{h}^{i}, \text{and}\\
& \fint_{e}\partial_{n_{e}} w_{h} \ud s \text{ is continuous across } e \in \mathcal{E}_{h}^{i}\Big\};
\end{split}
\end{align}
%%%%%%%
\begin{align}
V_{hs}^{\rm{R}} :=\Big\{w_{h}\in V_{h}^{\rm{R}} :  w_{h}(a)\text{ vanishes at } a\in \mathcal{N}_{h}^{b} \Big\}.
\end{align}
\noindent In this section, we will present an available set of  basis functions of $V_{hs}^{\rm{R}}$.

\subsection{Compatibility conditions}
By a pure linear algebra argument, the following description holds for $P_{2}(K)$.
\begin{lemma}{\rm(\cite[Lemma 15]{Shuo.Zhang2017})}
Let $K$ be a rectangle with vertices $a_{i}$ and edges $\Gamma_{l}$ ($i,l = 1:4$). Denote its length and width by $L$ and $H$, respectively; see Figure \ref{fig:subdivision} (Left). Then, given $\alpha_{i}, \beta_{i} \in \mathbb{R}$, there exists a uniquely $p\in P_{2}(K)$ satisfying
$$
p(a_{i}) = \alpha_{i}, \fint_{\Gamma_{1}}\partial_{x}p=\beta_{1}, \fint_{\Gamma_{2}}\partial_{y}p=\beta_{2}, \fint_{\Gamma_{3}}\partial_{x}p=\beta_{3}, \fint_{\Gamma_{4}}\partial_{y}p=\beta_{4}.
$$
if and only if the following compatibility conditions are satisfied on $K$,
\begin{align}
\frac{\alpha_{3}-\alpha_{4}}{L} + \frac{\alpha_{2}-\alpha_{1}}{L} = \beta_{1}+\beta_{3};\label{eq:c1}\\
\frac{\alpha_{3}-\alpha_{2}}{H} + \frac{\alpha_{4}-\alpha_{1}}{H} = \beta_{2}+\beta_{4}.\label{eq:c2}
\end{align}
\end{lemma}

Recall the definitions of $\mathcal{G}_{h}$,  $\mathcal{G}_{i}$, and $\mathcal{G}^j$ in Appendix A. Also, the vertices are labeled by $X_i^j$, the midpoints on any edge by $Y^j_i$ and $Z^j_i$, and the cells by $K^j_i$; see Figure \ref{fig:subdivision} (Right). Next we present some local or global functions in $V_{hs}^{\rm{R}}$ by giving their value on $X_{i}^{j}\in \mathcal{E}_{h}^{i}$ and derivative on the midpoint of $e\in \mathcal{E}_{h}$.
%%%%%%%%%
\begin{figure}[h]
\setlength{\unitlength}{0.75mm}
\centering

\begin{picture}(140,90)(-60,-5)%
\thicklines\color{black}%
\put(-60,0){\line(1,0){35}}
\put(-60,30){\line(1,0){35}}
\put(-60,0){\line(0,1){30}}
\put(-25,0){\line(0,1){30}}
\put(-60.8,-0.8){$\bullet$}
\put(-60.8,29.2){$\bullet$}
\put(-25.8,-0.8){$\bullet$}
\put(-25.8,29.2){$\bullet$}

\put(-59,1){$a_1$}
\put(-59,31){$a_4$}
\put(-24,1){$a_2$}
\put(-24,31){$a_3$}

\put(-59,14){$\Gamma_1$}
\put(-24,14){$\Gamma_3$}
\put(-41,1){$\Gamma_2$}
\put(-41,31){$\Gamma_4$}
\put(-42,13){$K$}

\put(0,0){\line(1,0){80}}
\put(0,20){\line(1,0){80}}
\put(0,40){\line(1,0){80}}
\put(0,60){\line(1,0){80}}
\put(0,80){\line(1,0){80}}

\put(0,0){\line(0,1){80}}
\put(20,0){\line(0,1){80}}
\put(40,0){\line(0,1){80}}
\put(60,0){\line(0,1){80}}
\put(80,0){\line(0,1){80}}

\put(-7.5,10){$H_{1}$}
\put(-7.5,30){$H_{2}$}
\put(-7.5,50){$\cdots$}
\put(-7.5,70){$H_{n}$}

\put(10,83){$L_{1}$}
\put(30,83){$L_{2}$}
\put(50,83){$\cdots$}
\put(70,83){$L_{m}$}

\put(19,8.5){$\bullet$}
\put(9,19){$\bullet$}
\put(19,19){$\bullet$}
\put(-1,8.5){$\bullet$}
\put(9,-1){$\bullet$}
\put(19,39){$\bullet$}
\put(39,39){$\bullet$}
\put(39,19){$\bullet$}

\put(22,10){$Z^1_2$}
\put(2,10){$Z^1_1$}
\put(12,21){$Y^2_1$}
\put(12,1){$Y^1_1$}
\put(21,21){$X^1_1$}
\put(21,41){$X^2_1$}
\put(41,21){$X^1_2$}
\put(41,41){$X^2_2$}

\put(83,9){$\mathcal{G}^1$}
\put(83,29){$\mathcal{G}^2$}
\put(83,49){$\cdots$}
\put(83,69){$\mathcal{G}^n$}

\put(9,-6){$\mathcal{G}_1$}
\put(29,-6){$\mathcal{G}_2$}
\put(49,-6){$\cdots$}
\put(69,-6){$\mathcal{G}_m$}

\put(9,9){$K^1_1$}
\put(29,9){$K^1_2$}
\put(9,29){$K^2_1$}
\put(29,29){$K^2_2$}

\end{picture}
\vskip0mm
\caption{Left: Illustration of a rectangle $K$ with width $L$ and height $H$. Right: Illustration of the grid $\mathcal{G}_h$: $X^j_i$ denotes the vertices, $Y^j_i$ and $Z^j_i$ denote the midpoints and $K^j_i$ denotes the cells. $L_{i}$ and $H_{i}$  denote the widths (heights) of the cells in the same column (row).}\label{fig:subdivision}
\end{figure}
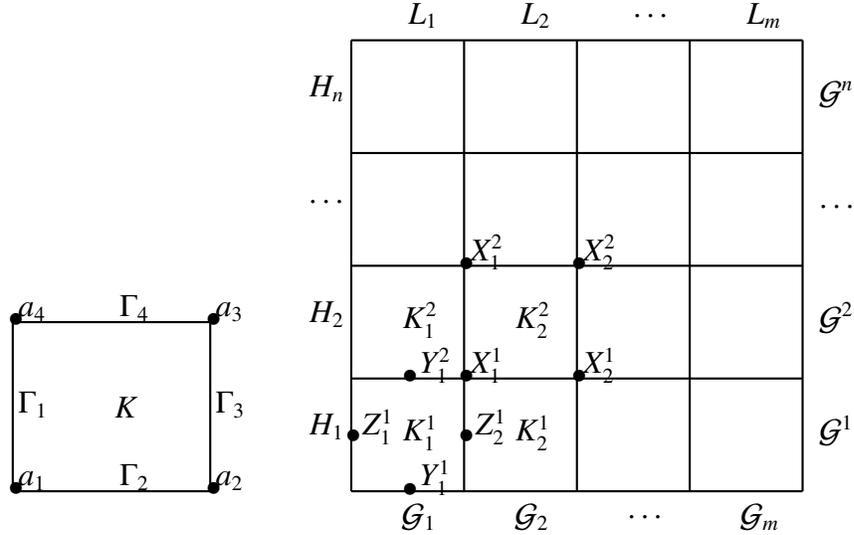

\subsection{Patterns of supports of basis functions in $V_{hs}^{\rm{R}}$}
Associated with $\mathcal{G}_{h}$, we present some patterns, i.e., the support sets of basis functions in the RRM element space. To begin with, we introduce some notations. An interior edge, denoted by $e_{{\rm bot},i}$ $(2\leqslant i \leqslant m-1)$, is called a bottom interior edge if its two endpoints are interior points, and its lower opposite edge is on the bottom of $\mathcal{G}_{h}$. A top interior edge $e_{{\rm top},i}$, a left interior edge $e_{{\rm lef},j}$, and a right interior edge $e_{{\rm rig},j}$ are defined in a similar way $(2\leqslant j \leqslant n-1)$. 
 
 In the following lemmas, we always denote, by $\omega$, a generic patch with boundaries $\Gamma_{l}$ ($l=1:4$), anticlockwise. 
 
\begin{lemma}\label{lem:basis1} 
Let $e_{{\rm bot},i}$ ($2\leqslant i \leqslant m-1$) be a bottom interior edge with endpoints $X_{i-1}^{1}$ and $X_{i}^{1}$ . Let $\omega$ be a $3\times 2$ cells patch as shown in Figure \ref{fig:APfunction11}, left. Define 
$$
V_{{\rm bot},i}^{\rm R} :=\Big \{v_{h}\in V_{hs}^{\rm{R}}(\omega) :  \fint_{e} \partial_{n_{e}}v_{h} \ud s =0, \ \forall e \subset \Gamma_{l}, \ l = 1,3,4 \Big \}.
$$
Likewise, we can define $V_{{\rm top},i}^{\rm R}$ (see Figure \ref{fig:APfunction11}, right), $V_{{\rm lef},j}^{\rm R}$, and  $V_{{\rm rig},j}^{\rm R}$ (see Figure \ref{fig:APfunction12}, left and right), where $2\leqslant i \leqslant m-1$ and $2\leqslant j \leqslant n-1$. Then we have ${\rm dim}(V_{{\rm bot},i}^{\rm R})= {\rm dim}(V_{{\rm top},i}^{\rm R}) = {\rm dim}(V_{{\rm lef},j}^{\rm R}) ={\rm dim}(V_{{\rm rig},j}^{\rm R}) = 1$.
\end{lemma}
%%%%%%
\begin{proof}
First, we consider $V_{{\rm bot},i}^{\rm R}$. Let the geometric features of $\omega$ be represented as in Figure \ref{fig:APfunction11}, left. Given $\varphi \in V_{{\rm bot},i}^{\rm R}$, denote by $x_{i}^{j}:=\varphi(X_{i}^{j})$, $y_{i}^{j}:=\partial_y\varphi(Y_{i}^{j})$, and $z_{i}^{j}:=\partial_x\varphi(Z_{i}^{j})$. Apply conditions \eqref{eq:c1} and \eqref{eq:c2} on every element, we have, row by row,
\begin{align}
& \frac{x_{i-1}^1}{L_{i-1}} = z_{i}^1,\  \frac{x_{i-1}^1}{H_{1}} = y_{i-1}^1+y_{i-1}^2,\  \frac{x_{i}^1-x_{i-1}^1}{L_{i}} = z_{i}^1+z_{i+1}^1,\ \frac{x_{i}^1+x_{i-1}^1}{H_{1}} = y_{i}^1+y_{i}^2,
\\
& -\frac{x_{i}^1}{L_{i+1}} = z_{i+1}^1,\  \frac{x_{i}^1}{H_{1}} = y_{i+1}^1+y_{i+1}^2,\  \frac{x_{i-1}^1}{L_{i-1}} = z_{i}^2,\  -\frac{x_{i-1}^1}{H_{2}} = y_{i-1}^2,
\\
& \frac{x_{i}^1-x_{i-1}^1}{L_{i}} = z_{i}^2+z_{i+1}^2,\  -\frac{x_{i-1}^1+x_{i}^1}{H_{2}} = y_{i}^2,\  -\frac{x_{i}^1}{L_{i+1}} = z_{i+1}^2,\  -\frac{x_{i}^1}{H_{2}} = y_{i+1}^2.
\end{align}
Rewrite the system after adjusting the order,\begin{align}
 \frac{1}{L_{i-1}}x_{i-1}^1 - z_{i}^1 =0,\label{eq:1}
\end{align}
\begin{align}
 \frac{1}{H_{1}}x_{i-1}^1 - y_{i-1}^1 - y_{i-1}^2   =0,
\end{align}
\begin{align}
 -\frac{1}{L_{i}}x_{i-1}^1 + \frac{1}{L_{i}}x_{i}^1 - z_{i}^1  - z_{i+1}^1  =0,\label{eq:3}
\end{align}
\begin{align}
\frac{1}{H_{1}}x_{i-1}^1 + \frac{1}{H_{1}}x_{i}^1  - y_{i}^1 - y_{i}^2   =0,
\end{align}
\begin{align}
- \frac{1}{L_{i+1}}x_{i}^1 - z_{i+1}^1 =0,\label{eq:5}
\end{align}
\begin{align}
\frac{1}{H_{1}}x_{i}^1 - y_{i+1}^1 - y_{i+1}^2   =0,
\end{align}
\begin{align}
\frac{1}{L_{i-1}}x_{i-1}^1 - z_{i}^2 = 0,\label{eq:7}
\end{align}
\begin{align}
-\frac{1}{H_{2}}x_{i-1}^1 - y_{i-1}^2 = 0,
\end{align}
\begin{align}
-\frac{1}{L_{i}}x_{i-1}^1 + \frac{1}{L_{i}}x_{i}^1 - z_{i}^2 -z_{i+1}^2= 0,\label{eq:9}
\end{align}
\begin{align}
-\frac{1}{H_{2}}x_{i-1}^1 -\frac{1}{H_{2}}x_{i}^1- y_{i}^2= 0,
\end{align}
\begin{align}
-\frac{1}{L_{i+1}}x_{i}^1 - z_{i+1}^2 = 0,\label{eq:11}
\end{align}
\begin{align}
-\frac{1}{H_{2}}x_{i}^1 - y_{i+1}^2 = 0.
\end{align}

\noindent It is straightforward to verify that $\eqref{eq:3} - \eqref{eq:1} -\eqref{eq:5} -\big[\eqref{eq:9} - \eqref{eq:7} -\eqref{eq:11}\big ] =0.$ The system admits a one-dimension solution space. Thus we obtain ${\rm dim} (V_{{\rm bot},i}^{\rm R}) =1$. Likewise, we have ${\rm dim} (V_{{\rm top},i}^{\rm R}) = {\rm dim} (V_{{\rm lef},j}^{\rm R} )= {\rm dim} (V_{{\rm rig},j}^{\rm R}) =1 \ (\forall \ 2\leqslant i \leqslant m-1, 2\leqslant j \leqslant n-1$).
\end{proof}

\begin{figure}[!htbp]
\setlength{\unitlength}{0.9mm}
\centering
\begin{picture}(150,50)(-90,-5)%
\thicklines\color{black}%
\dashline{1}(-85,40)(-25,40)
\dashline{1}(-85,0)(-85,40)
\dashline{1}(-25,0)(-25,40)
\put(-85,0){\line(1,0){60}}
\put(-85,20){\line(1,0){60}}
\put(-65,0){\line(0,1){40}}
\put(-45,0){\line(0,1){40}}

\put(-77,-5){$L_{i-1}$}
\put(-58,-5){$L_{i}$}
\put(-38,-5){$L_{i+1}$}
\put(-92,9){$H_{1}$}
\put(-92,29){$H_{2}$}

\put(-77,-1){$\bullet$}
\put(-77,19){$\bullet$}
\put(-76,1.5){$Y_{i-1}^{1}$}
\put(-76,21.5){$Y_{i-1}^{2}$}

\put(-56,-1){$\bullet$}
\put(-56,19){$\bullet$}
\put(-55,1.5){$Y_{i}^{1}$}
\put(-55,21.5){$Y_{i}^{2}$}
\put(-55,15){$e_{{\rm bot},i}$}

\put(-35,-1){$\bullet$}
\put(-35,19){$\bullet$}
\put(-33.5,1.5){$Y_{i+1}^{1}$}
\put(-33.5,21.5){$Y_{i+1}^{2}$}

\put(-66,9){$\bullet$}
\put(-66,19){$\bullet$}
\put(-66,29){$\bullet$}
\put(-64,9){$Z_{i}^{1}$}
\put(-64,29){$Z_{i}^{2}$}

\put(-46,9){$\bullet$}
\put(-46,19){$\bullet$}
\put(-46,29){$\bullet$}
\put(-44,9){$Z_{i+1}^{1}$}
\put(-44,29){$Z_{i+1}^{2}$}

\put(-64,21.5){$X_{i-1}^{1}$}
\put(-44,21.5){$X_{i}^{1}$}

\dashline{1}(0,0)(60,0)
\dashline{1}(0,0)(0,40)
\dashline{1}(60,0)(60,40)
\put(0,20){\line(1,0){60}}
\put(0,40){\line(1,0){60}}
\put(20,0){\line(0,1){40}}
\put(40,0){\line(0,1){40}}

\put(8,-5){$L_{i-1}$}
\put(27,-5){$L_{i}$}
\put(46,-5){$L_{i+1}$}
\put(-8,9){$H_{n-1}$}
\put(-7,29){$H_{n}$}

\put(8,19){$\bullet$}
\put(8,39){$\bullet$}
\put(9,21.5){$Y_{i-1}^{n}$}
\put(9,41.5){$Y_{i-1}^{n+1}$}

\put(29,19){$\bullet$}
\put(29,39){$\bullet$}
\put(31.5,21.5){$Y_{i}^{n}$}
\put(30,41.5){$Y_{i}^{n+1}$}
\put(30,15){$e_{{\rm top},i}$}

\put(50,19){$\bullet$}
\put(50,39){$\bullet$}
\put(52,21.5){$Y_{i+1}^{n}$}
\put(51.5,41.5){$Y_{i+1}^{n+1}$}

\put(19,9){$\bullet$}
\put(19,19){$\bullet$}
\put(19,29){$\bullet$}
\put(21,9){$Z_{i}^{n-1}$}
\put(21,29){$Z_{i}^{n}$}

\put(39,9){$\bullet$}
\put(39,19){$\bullet$}
\put(39,29){$\bullet$}
\put(41,9){$Z_{i+1}^{n-1}$}
\put(41,29){$Z_{i+1}^{n}$}

\put(20,21.5){$X_{i-1}^{n-1}$}
\put(40.5,21.5){$X_{i}^{n-1}$}

\end{picture}
\vskip0mm

Illustration of the column patterns and row patterns of basis functions

\caption{Left: Illustration of a $3\times2$ pattern associated with $e_{\text{bot}}^{i}$. Right: Illustration of a $3\times 2$ pattern associated with $e_{\text{top},i}$. The values or derivatives vanish at edges on the dotted lines.}
\label{fig:APfunction11}
\end{figure}

\begin{figure}[!htbp]
\setlength{\unitlength}{0.9mm}
\centering
\begin{picture}(135,70)(-80,-5)%
\thicklines\color{black}%
\dashline{1}(-75,0)(-25,0)
\dashline{1}(-75,60)(-25,60)
\dashline{1}(-25,0)(-25,60)
\put(-75,0){\line(0,1){60}}
\put(-50,0){\line(0,1){60}}
\put(-75,20){\line(1,0){50}}
\put(-75,40){\line(1,0){50}}

\put(-51,19){$\bullet$}
\put(-51,39){$\bullet$}
\put(-49,21.5){$X_{1}^{j-1}$}
\put(-49,41.5){$X_{1}^{j}$}
\put(-76,9){$\bullet$}
\put(-76,29){$\bullet$}
\put(-76,49){$\bullet$}
\put(-74,9){$Z_{1}^{j-1}$}
\put(-74,29){$Z_{1}^{j}$}
\put(-74,49){$Z_{1}^{j+1}$}
\put(-51,9){$\bullet$}
\put(-51,29){$\bullet$}
\put(-51,49){$\bullet$}
\put(-49,9){$Z_{2}^{j-1}$}
\put(-49,29){$Z_{2}^{j}$}
\put(-59,29){$e_{{\rm lef},j}$}

\put(-49,49){$Z_{2}^{j+1}$}
\put(-63.5,19){$\bullet$}
\put(-38.5,19){$\bullet$}
\put(-62.5,21.5){$Y_{1}^{j}$}
\put(-37.5,21.5){$Y_{2}^{j}$}
\put(-63.5,39){$\bullet$}
\put(-38.5,39){$\bullet$}
\put(-62.5,41.5){$Y_{1}^{j+1}$}
\put(-37.5,41.5){$Y_{2}^{j+1}$}

\put(-62.5,-5){$L_{1}$}
\put(-37.5,-5){$L_{2}$}
\put(-85,9){$H_{j-1}$}
\put(-85,29){$H_{j}$}
\put(-85,49){$H_{j+1}$}

\dashline{1}(0,0)(50,0)
\dashline{1}(0,60)(50,60)
\dashline{1}(0,0)(0,60)
\put(50,0){\line(0,1){60}}
\put(25,0){\line(0,1){60}}
\put(0,20){\line(1,0){50}}
\put(0,40){\line(1,0){50}}

\put(24,19){$\bullet$}
\put(24,39){$\bullet$}
\put(26,22){$X_{m-1}^{j-1}$}
\put(26,42){$X_{m-1}^{j}$}
\put(24,9){$\bullet$}
\put(24,29){$\bullet$}
\put(24,49){$\bullet$}
\put(26,9){$Z_{m}^{j-1}$}
\put(26,29){$Z_{m}^{j}$}
\put(16,29){$e_{{\rm rig},j}$}
\put(26,49){$Z_{m}^{j+1}$}
\put(49,9){$\bullet$}
\put(49,29){$\bullet$}
\put(49,49){$\bullet$}
\put(51,9){$Z_{m+1}^{j-1}$}
\put(51,29){$Z_{m+1}^{j}$}
\put(51,49){$Z_{m+1}^{j+1}$}
\put(11.5,19){$\bullet$}
\put(37.5,19){$\bullet$}
\put(10.5,22){$Y_{m-1}^{j}$}
\put(36.5,22){$Y_{m}^{j}$}
\put(12.5,39){$\bullet$}
\put(37.5,39){$\bullet$}
\put(10.5,42.5){$Y_{m-1}^{j+1}$}
\put(36.5,42.5){$Y_{m}^{j+1}$}

\put(11.5,-5){$L_{m-1}$}
\put(36.5,-5){$L_{m}$}
\put(-8.5,9){$H_{j-1}$}
\put(-8.5,29){$H_{j}$}
\put(-8.5,49){$H_{j+1}$}
\end{picture}
\vskip0mm
\caption{Left: Illustration of a $2\times3$ pattern associated with $e_{\text{lef},j}$. Right: Illustration of a $2\times3$ pattern associated with $e_{\text{rig},j}$. The values or derivatives vanish at edges on the dotted lines.
}\label{fig:APfunction12}
\end{figure} 
%%%%%
\noindent With similar procedures in Lemma \ref{lem:basis2}, the following two lemmas can be obtained.
\begin{lemma}\label{lem:basis2}
Let $X_{m-1}^{n-1}$ be an interior node in the northeast corner. Let $\omega$ be a $2\times 2$ patch as shown in Figure \ref{fig:APfunction3and4}, left. Define $V_{X}^{\rm R} :=\Big\{v_{h}\in V_{hs}^{\rm{R}}(\omega): \fint_{e} \partial_{n_{e}}v_{h} \ud s =0,\  \forall e \subset \Gamma_{l}\ (l = 1,2)\Big\}$. Then ${\rm dim}(V_{X}^{\rm R}) =1$. 
\end{lemma}
\begin{lemma}{\rm(\cite[Lemma 15]{Shuo.Zhang2017})}\label{lem:basis3}
Let $\omega$ be a $3 \times 3 $ patch as shown in Figure \ref{fig:APfunction3and4}, right. Define $V_{K_{i}^{j}}^{\rm R} :=\Big \{v_{h}\in V_{hs}^{\rm{R}}(\omega): \fint_{e} \partial_{n_{e}}v_{h} \ud s =0, \ \forall e \subset \Gamma_{i} \ (1 \leqslant l \leqslant 4 )\Big \}$, where $2\leqslant i \leqslant m-1$ and $ \ 2\leqslant j \leqslant n-1$.
Then ${\rm dim} (V_{K_{i}^{j}}^{\rm R} )=1$.
\end{lemma}
%%%%%%%%
\begin{figure}[!htbp]
\setlength{\unitlength}{0.9mm}
\centering
\begin{picture}(160,85)(-80,-5)%
\thicklines\color{black}%
\dashline{1}(-75,0)(-75,50)
\dashline{1}(-75,0)(-25,0)
\put(-75,50){\line(1,0){50}}
\put(-25,0){\line(0,1){50}}
\put(-75,25){\line(1,0){50}}
\put(-50,0){\line(0,1){50}}
\put(-51,24){$\bullet$}
\put(-49.9,27){$X_{m-1}^{n-1}$}
\put(-49,14.5){$Z_{m}^{n-1}$}
\put(-49,39.5){$Z_{m}^{n}$}
\put(-63.5,24){$\bullet$}
\put(-63.5,49){$\bullet$}
\put(-38.5,24){$\bullet$}
\put(-38.5,49){$\bullet$}
\put(-62.5,26.5){$Y_{m-1}^{n}$}
\put(-62.5,51.5){$Y_{m-1}^{n+1}$}
\put(-36.5,26.5){$Y_{m}^{n}$}
\put(-37.5,51.5){$Y_{m}^{n+1}$}

\put(-51,11.5){$\bullet$}
\put(-51,36.5){$\bullet$}
\put(-26,11.5){$\bullet$}
\put(-26,36.5){$\bullet$}
\put(-24,14.5){$Z_{m+1}^{n-1}$}
\put(-24,39.5){$Z_{m+1}^{n}$}

\dashline{1}(0,0)(0,75)
\dashline{1}(0,0)(75,0)
\dashline{1}(0,75)(75,75)
\dashline{1}(75,0)(75,75)
\put(0,25){\line(1,0){75}}
\put(0,50){\line(1,0){75}}
\put(25,0){\line(0,1){75}}
\put(50,0){\line(0,1){75}}

\put(24,24){$\bullet$}
\put(24,49){$\bullet$}
\put(49,24){$\bullet$}
\put(49,49){$\bullet$}
\put(26,26.5){$X_{i-1}^{j-1}$}
\put(26,51.5){$X_{i-1}^{j}$}
\put(51,26.5){$X_{i}^{j-1}$}
\put(51,51.5){$X_{i}^{j}$}
\put(36.5,37.5){$K^j_i$}

\put(12.5,24){$\bullet$}
\put(37.5,24){$\bullet$}
\put(62.5,24){$\bullet$}
\put(14,26.5){$Y_{i-1}^{j}$}
\put(38.5,26.5){$Y_{i}^{j}$}
\put(63.5,26.5){$Y_{i+1}^{j}$}

\put(12.5,49){$\bullet$}
\put(37.5,49){$\bullet$}
\put(62.5,49){$\bullet$}
\put(14,51.5){$Y_{i-1}^{j+1}$}
\put(38.5,51.5){$Y_{i}^{j+1}$}
\put(63.5,51.5){$Y_{i+1}^{j+1}$}

\put(24,12.5){$\bullet$}
\put(24,37.5){$\bullet$}
\put(24,62.5){$\bullet$}
\put(26,12.5){$Z_{i}^{j-1}$}
\put(26,37.5){$Z_{i}^{j}$}
\put(26,62.5){$Z_{i}^{j+1}$}

\put(51,12.5){$Z_{i+1}^{j-1}$}
\put(51,37.5){$Z_{i+1}^{j}$}
\put(51,62.5){$Z_{i+1}^{j+1}$}

\put(-63,-5){$L_{m-1}$}
\put(-38,-5){$L_{m}$}
\put(12,-5){$L_{i-1}$}
\put(37,-5){$L_{i}$}
\put(62,-5){$L_{i+1}$}
\put(-80,12){$H_{n-1}$}
\put(-80,37){$H_{n}$}
\put(-5,12){$H_{j-1}$}
\put(-5,37){$H_{j}$}
\put(-5,62){$H_{j+1}$}
\end{picture}
\vskip0mm
\caption{Left: Illustration of a $2\times2$ pattern associated with $X_{m-1}^{n-1}$. Right: Illustration of a $3\times3$ pattern associated with $K_{i}^{j}$. The values or derivatives vanish at edges on the dotted lines.}\label{fig:APfunction3and4}
\end{figure}
\begin{lemma}\label{lem:1times3}
Let $\omega$ be a $3\times 1$ or $1\times 3$ patch in Figure \ref{fig:AP1times3}, left and right, respectively. Define 
$$V_{\omega} := \Big\{v_{h}\in V_{hs}^{\rm R}(\omega): \fint_{e}\partial_{n_{e}}v_{h} \ud s =0, \ \forall e \subset \Gamma_{l} \ ( l = 1,2,3)\Big\}.$$ 
Then, ${\rm dim}(V_{\omega}) =1$ and $v_{h}(X_{1}) = 0$ implies $v_{h} \equiv 0$ on $\omega.$
\end{lemma}
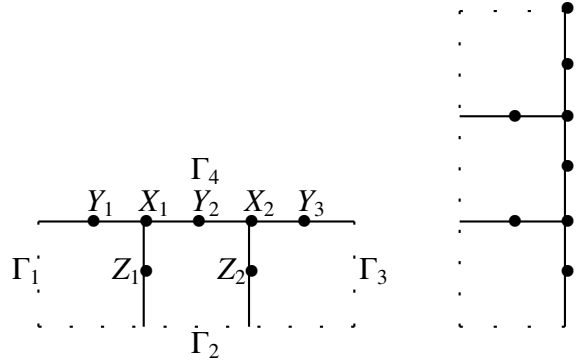
\begin{figure}[!htbp]
\setlength{\unitlength}{0.7mm}
\begin{picture}(110,65)(-85,-5)
\thicklines\color{black}
\put(-80,20){\line(1,0){60}}
\put(-60,0){\line(0,1){20}}
\put(-40,0){\line(0,1){20}}
\dashline{1}(-80,0)(-20,0)
\dashline{1}(-80,0)(-80,20)
\dashline{1}(-20,0)(-20,20)

\put(-61,18.5){$\bullet$}
\put(-41,18.5){$\bullet$}
\put(-61,22){$X_{1}$}
\put(-41,22){$X_{2}$}
\put(-71,18.5){$\bullet$}
\put(-51,18.5){$\bullet$}
\put(-31,18.5){$\bullet$}
\put(-71,22){$Y_{1}$}
\put(-51,22){$Y_{2}$}
\put(-31,22){$Y_{3}$}
\put(-61,9){$\bullet$}
\put(-41,9){$\bullet$}
\put(-66,9){$Z_{1}$}
\put(-46,9){$Z_{2}$}

\put(-51,-5){$\Gamma_{2}$}
\put(-85,9){$\Gamma_{1}$}
\put(-19,9){$\Gamma_{3}$}
\put(-51,28){$\Gamma_{4}$}

\put(20,0){\line(0,1){60}}
\put(0,20){\line(1,0){20}}
\put(0,40){\line(1,0){20}}
\dashline{1}(0,0)(20,0)
\dashline{1}(0,60)(20,60)
\dashline{1}(0,0)(0,60)

\put(19,18.5){$\bullet$}
\put(19,38.5){$\bullet$}
\put(19,59){$\bullet$}
\put(19,9){$\bullet$}
\put(19,29){$\bullet$}
\put(19,48.5){$\bullet$}
\put(9,18.5){$\bullet$}
\put(9,38.5){$\bullet$}

%\put(10,62){$\Gamma_{1}$}
%\put(-7,29){$\Gamma_{2}$}
%\put(10,-5){$\Gamma_{3}$}
%\put(23,29){$\Gamma_{4}$}

\end{picture}
\caption{Illustration of $3\times 1$ or $1\times 3$ patch $\omega$. The values or derivatives vanish at edges on the dotted lines.
}\label{fig:AP1times3}
\end{figure}
%%%%%%%%
\begin{lemma}\label{lem:basis4}
If we denote $\omega = \mathcal{G}_{i}$ the union of elements in the $i$-th column with boundaries $\bigcup_{l=1}^{4}\Gamma_{l}$ (see Figure \ref{fig:APfunction4}, left), and define 
$$
V_{{\rm col},i}^{\rm R} := \Big\{v_{h} \in V_{hs}^{\rm{R}}(\omega): \fint_{e}  \partial_{n_{e}}v_{h} \ud s =0, \ \forall e \subset \Gamma_{l} \ (l = 1,3) \Big\},
$$ 
then ${\rm dim}(V_{{\rm col},i}^{\rm R})  =1$. If $\omega = \mathcal{G}^{j}$ (see Figure \ref{fig:APfunction4}, right),  and define 
$$V_{{\rm row},j}^{\rm R} :=\Big\{v_{h} \in V_{hs}^{\rm{R}}(\omega): \fint_{e}  \partial_{n_{e}}v_{h} \ud s =0,\ \forall e \subset \Gamma_{l} \ (l = 2,4) \Big\},$$
then  ${\rm dim}(V_{{\rm row},j}^{\rm R}) =1$. 
\end{lemma}
\begin{proof}
Obviously, $V_{{\rm col},i}^{\rm R} \ne \varnothing$. Assume that $v_{h} \in V_{{\rm col},i}^{\rm R}$. We denote $y_{i}^{j}: = v_{h}(Y_{i}^{j})$, where $1\leqslant j \leqslant n+1$; see Figure \ref{fig:APfunction4} (Left). Apply  \eqref{eq:c1} and \eqref{eq:c2} to $v_{h}$ on each element in $\mathcal{G}_{i}$, and we obtain $y_{i}^{j} = -y_{i}^{j+1}\ (1\leqslant j \leqslant n)$. It proves that ${\rm dim}(V_{{\rm col},i}^{\rm R})  =1$.
Similarly, assume that $w_{h} \in V_{{\rm col},i}^{\rm R}$ and denote $z_{i}^{j}: = w_{h}(Z_{i}^{j})$ $(1\leqslant i \leqslant m+1)$. We have $z_{i}^{j} = -z_{i+1}^{j}\ (1\leqslant i \leqslant m)$; see Figure \ref{fig:APfunction4} (Right). Thus we  obtain ${\rm dim}(V_{{\rm row},j}^{\rm R}) =1$. 
\end{proof}
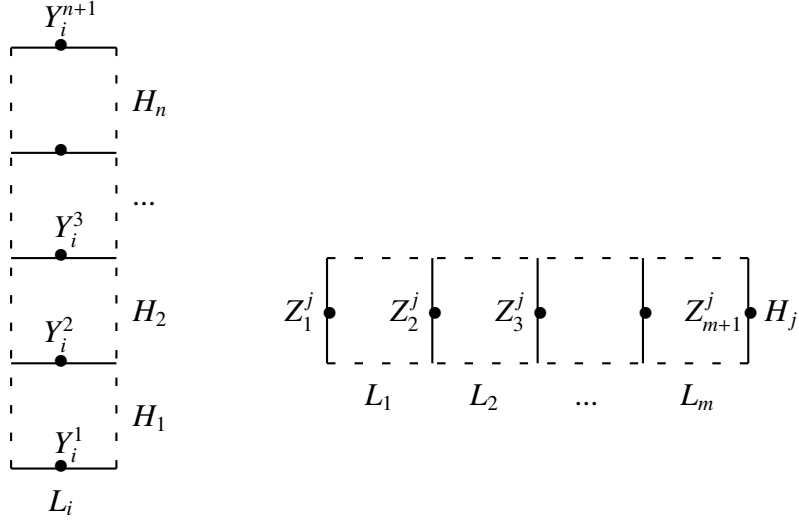
\begin{figure}[!htbp]
\setlength{\unitlength}{1.4mm}
\begin{picture}(90,50)(10,-5)
\thicklines\color{black}
\put(20,0){\line(1,0){10}}
\put(20,10){\line(1,0){10}}
\put(20,20){\line(1,0){10}}
\put(20,30){\line(1,0){10}}
\put(20,40){\line(1,0){10}}
%\put(20,0){\line(0,1){40}}
%\put(30,0){\line(0,1){40}}
\dashline{1}(20,0)(20,40)
\dashline{1}(30,0)(30,40)
\put(23.5,-4){$L_i$}
\put(31.5,4){$H_{1}$}
\put(31.5,14){$H_{2}$}
\put(31.5,25){$...$}
\put(31.5,34){$H_{n}$}
\put(24,1.5){$ Y_{i}^{1}$}
\put(23,12){$Y_{i}^{2}$}
\put(24,22){$Y_{i}^{3}$}
\put(23,42){$ Y_{i}^{n+1}$}
\put(24,-0.5){$\bullet$}
\put(24,9.5){$\bullet$}
\put(24,19.5){$\bullet$}
\put(24,29.5){$\bullet$}
\put(24,39.5){$\bullet$}

\dashline{1}(50,10)(90,10)
\dashline{1}(50,20)(90,20)
\put(50,10){\line(0,1){10}}
\put(60,10){\line(0,1){10}}
\put(70,10){\line(0,1){10}}
\put(80,10){\line(0,1){10}}
\put(90,10){\line(0,1){10}}

\put(46,14){$Z_{1}^{j}$}
\put(56,14){$Z_{2}^{j}$}
\put(66,14){$Z_{3}^{j}$}
\put(84,14){$ Z_{m+1}^{j}$}
\put(49.5,14){$\bullet$}
\put(59.5,14){$\bullet$}
\put(69.5,14){$\bullet$}
\put(79.5,14){$\bullet$}
\put(89.5,14){$\bullet$}

\put(53.5,6){$L_{1}$}
\put(63.5,6){$L_{2}$}
\put(73.5,6){$...$}
\put(83.5,6){$L_{m}$}
\put(91.5,14){$H_{j}$}

\end{picture}
\caption{Left: Illustration of a column pattern associated with $\mathcal{G}_{i}$. Right: Illustration of a row pattern associated with $\mathcal{G}^{j}$. The values or derivatives vanish at edges on the dotted lines.}\label{fig:APfunction4}

\end{figure}
%%%%%
\begin{remark}{\rm Here and throughout this paper, we do not distinct between $V_{{\rm bot},i}^{\rm R},$ $V_{{\rm top},i}^{\rm R},$ $V_{{\rm lef},j}^{\rm R},$ $V_{{\rm rig},j}^{\rm R},$ $V_{X}^{\rm R},$$V_{K_{i}^{j}}^{\rm R},$ $V_{{\rm col},i}^{\rm R},$ $V_{{\rm row},j}^{\rm R}$ and their respective extension onto the whole domain by zero. Each of them is a subspace in $V_{hs}^{\rm{R}}$.}
\end{remark}

\subsection{Structure of the RRM element space}
\begin{theorem}\label{thm:APbasisofRRM}
Let $\mathcal{G}_{h}$ be a $m\times n$ rectangular subdivision of $\Omega$. Define $V_{btlr}^{\rm R} : = \oplus_{i,j} \Big(V_{{\rm bot},i}^{\rm R} \oplus V_{{\rm top},i}^{\rm R}\oplus V_{{\rm lef},j}^{\rm R} \oplus V_{{\rm rig},j}^{\rm R} \Big)$, $V_{K}^{\rm R} :=  \oplus_{i,j} V_{K_{i}^{j}}^{\rm R} $, and $V_{{\rm glob}}^{\rm R} := V_{{\rm col},m-1}^{\rm R} \oplus V_{{\rm col},m}^{\rm R} \oplus V_{{\rm row},n-1}^{\rm R}\oplus V_{{\rm row},n}^{\rm R}$, where $2\leqslant i \leqslant m-1$, and $2\leqslant j \leqslant n-1$. Then we have $V_{hs}^{\rm{R}} = V_{btlr}^{\rm R} \oplus V_{{\rm glob}}^{\rm R} \oplus V_{K}^{\rm R} \oplus V_{X}^{\rm R}$.
\end{theorem}
\begin{proof}
Here we utilize the sweeping procedure again, and  divide the proof process into four steps.

\noindent{\bf Step 1.} Given $v \in V_{hs}^{\rm{R}}$. We begin with the boundaries of $\mathcal{G}_{h}$; see Figure \ref{fig:APElim12}. Recall the definition of  $e_{{\rm bot},i}$, $e_{{\rm top},i}$, $e_{{\rm lef},j}$, and $e_{{\rm rig},j}$. Associated with $e_{{\rm bot},2}$, $V_{{\rm bot},2}^{\rm R}$ is a subspace of $V$ defined in Lemma \ref{lem:basis1}. There exists a unique function $\varphi_{b,1}\in V_{{\rm bot},2}^{\rm R}$, such that $\varphi_{b,1} (Y_{1}^{1})= v(Y_{1}^{1})$. Then, $v_{1} = v-\varphi_{b,1} $ with $v_{1}\equiv 0$ on the bottom edge of $K_{1}^{1}$. Repeat the procedure for $Y_{2}^{1}$, ..., $Y_{m-2}^{1}$, and we obtain $v_{m-2} = v-\sum_{i=2}^{m-1}\varphi_{b,i}$ with $v_{m-2}$ vanishing on the first $m-2$ edges on the bottom boundary of $\mathcal{G}_{h}$. Similarly, we obtain functions $\varphi_{t,i}$, $\varphi_{l,j}$, and $\varphi_{r,j}$. Therefore, $w_{1}=v-\sum_{i=2}^{m-1}\varphi_{b,i} -\sum_{i=2}^{m-1}\varphi_{t,i}-\sum_{j=2}^{n-1}\varphi_{l,j}-\sum_{j=2}^{n-1}\varphi_{r,j}$ with $w_{1}$ vanishing on the dotted edges; see Figure \ref{fig:APElim12} (Middle).

\noindent{\bf Step 2.} There exists uniquely $\varphi_{{\rm col},i}$ in $V_{{\rm col},i}^{\rm R}$ and $\varphi_{{\rm row},j}$ in $V_{{\rm row},j}^{\rm R}$, which satisfy $\varphi_{{\rm col},i}(Y_{i}^{1}) = w_{1}(Y_{i}^{1})$, and $\varphi_{{\rm row},j}(Z_{1}^{j}) = w_{1}(Z_{1}^{j})$, respectively.
Then $w_{2} = w_{1}-\sum_{i = m-1}^{m} \varphi_{{\rm col},i} - \sum_{j = n-1}^{n} \varphi_{{\rm row},j}$ with $w_{2}$ vanishing on the dotted edges; see Figure \ref{fig:APElim12} (Right).

\noindent{\bf Step 3.} Consider elements in the first column $\mathcal{G}_{1}$; see Figure \ref{fig:AP1times3}. Note that $K_{2}^{2}$ is the only interior element whose $3\times 3$ patch contains $K_{1}^{1}$. There exists uniquely $\varphi_{1}^{1}\in V_{K_{2}^{2}}^{\rm R}$, such that $\varphi_{1}^{1}|_{K_{1}^{1}} = w_{2}|_{K_{1}^{1}}$. Therefore, $w_{2}-\varphi_{1}^{1}$ vanishes on $K_{1}^{1}$. Next, there exists a unique $\varphi_{1}^{2} \in V_{K_{2}^{3}}^{\rm R} $, such that $w_{2}-\varphi_{1}^{1}-\varphi_{1}^{2}$ vanishes on on $K_{1}^{1}\cup K_{1}^{2}$. Repeating the procedure for $\mathcal{G}_{1}$, we obtain $w_{2}-\sum_{j=1}^{n-2}\varphi^j_{1}$, which vanishes on $\cup_{j=1}^{n-2} K_1^j$. Notice that $\cup_{j=n-2}^{n}K_1^j$ is a $1\times 3$ patch $\omega$ stated in Lemma \ref{lem:1times3}; see Figure \ref{fig:APElim34} (Left). Since $w_{2}-\sum_{j=1}^{n-2}\varphi^j_{1}$ vanishes on $X_{1}^{n-2}$, it vanishes on $\omega$ and further $\mathcal{G}_{1}$.

 By repeating the procedure along the column $\mathcal{G}_{i}\ (i=2,\ \ldots,\ m-2)$, we derive $w_{3}:= w_{2}-\sum_{i=1}^{m-2}\sum_{j=1}^{n-2}\varphi^j_{i}$, which vanishes on $\bigcup_{i=1}^{m-2}\mathcal{G}_{i}$. Especially, $w_{3}(X_{m-2}^{j}) = 0$ ($1\leqslant j \leqslant n-2$). Since $\cup_{i=m-2}^{m}K_i^1$ forms a $1\times 3$ patch $\omega$ (see Figure \ref{fig:APElim34}, middle) and $w_{3}(X_{m-2}^{1}) = 0$, thus it vanishes on $\omega$. Find other $1\times 3$ or $3\times 1$ patch $\omega$, which satisfies the conditions in Lemma \ref{lem:1times3}, in columns $\mathcal{G}_{m-2}$, $\mathcal{G}_{m-1}$, and $\mathcal{G}_{m}$. Therefore, we derive that $w_{3}$ vanishes on $\mathcal{G}_{h}$ except for a $2\times 2 $ patch in the northeast corner; see Figure \ref{fig:APElim34} (Right).
 
 \begin{figure}[!htbp]
\setlength{\unitlength}{1mm}

\begin{picture}(160,50)(-10,-5)
\thicklines\color{black}
\put(0,0){\line(1,0){40}}
\put(0,10){\line(1,0){40}}
\put(0,20){\line(1,0){40}}
\put(0,30){\line(1,0){40}}
\put(0,40){\line(1,0){40}}
\put(0,0){\line(0,1){40}}
\put(10,0){\line(0,1){40}}
\put(20,0){\line(0,1){40}}
\put(30,0){\line(0,1){40}}
\put(40,0){\line(0,1){40}}
\put(-7,4){$\mathcal{G}^{1}$}
\put(-7,14){$...$}
\put(-8,24){$\mathcal{G}^{n-1}$}
\put(-7,34){$\mathcal{G}^{n}$}
\put(4,-5){$\mathcal{G}_{1}$}
\put(14,-5){$...$}
\put(22,-5){$\mathcal{G}_{m-1}$}
\put(33,-5){$\mathcal{G}_{m}$}

\put(42,19){$\Longrightarrow$}
\dashline{1}(50,0)(70,0)
\dashline{1}(50,0)(50,20)
\dashline{1}(50,40)(70,40)
\dashline{1}(90,0)(90,20)
\put(70,0){\line(1,0){20}}
\put(70,40){\line(1,0){20}}
\put(50,20){\line(0,1){20}}
\put(90,20){\line(0,1){20}}
\put(50,10){\line(1,0){40}}
\put(50,20){\line(1,0){40}}
\put(50,30){\line(1,0){40}}
\put(60,0){\line(0,1){40}}
\put(70,0){\line(0,1){40}}
\put(80,0){\line(0,1){40}}

\put(92,19){$\Longrightarrow$}
\dashline{1}(100,0)(140,0)
\dashline{1}(100,0)(100,40)
\dashline{1}(100,40)(120,40)
\dashline{1}(140,0)(140,20)
\put(100,10){\line(1,0){40}}
\put(100,20){\line(1,0){40}}
\put(100,30){\line(1,0){40}}
\put(110,0){\line(0,1){40}}
\put(120,0){\line(0,1){40}}
\put(130,0){\line(0,1){40}}
\put(120,40){\line(1,0){20}}
\put(140,20){\line(0,1){20}}
\end{picture}%$\Longrightarrow$
\caption{Elimination process of Step 1 and 2.} \label{fig:APElim12}
\end{figure}
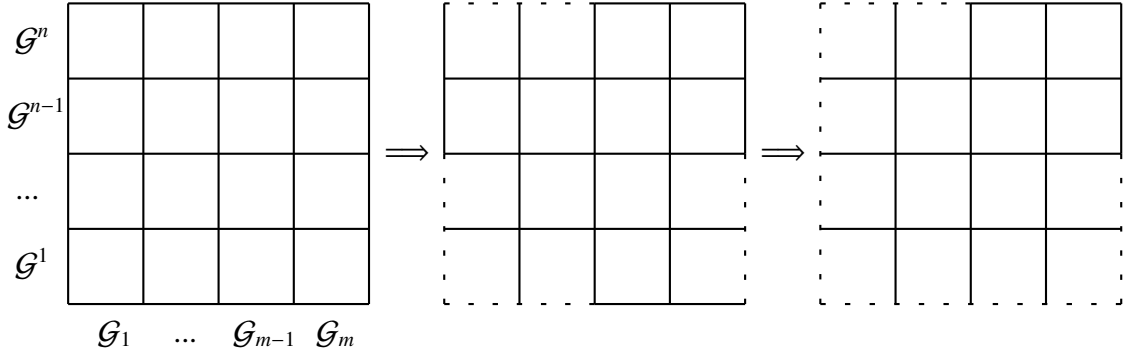
%%%%%%%%%
\begin{figure}[!htbp]
\setlength{\unitlength}{1mm}
\begin{picture}(160,50)(-10,-5)
\thicklines\color{black}
\dashline{1}(0,0)(0,40)
\dashline{1}(0,0)(40,0)
\dashline{1}(0,10)(10,10)
\dashline{1}(10,0)(10,10)
\dashline{1}(40,0)(40,20)
\dashline{1}(0,40)(20,40)

\put(10,10){\line(1,0){30}}
\put(0,20){\line(1,0){40}}
\put(0,30){\line(1,0){40}}
\put(10,10){\line(0,1){30}}
\put(20,0){\line(0,1){40}}
\put(30,0){\line(0,1){40}}
\put(20,40){\line(1,0){20}}
\put(40,20){\line(0,1){20}}

\put(9,19){$\bullet$}
\put(4,25){$\omega$}

\put(-7,4){$\mathcal{G}^{1}$}
\put(-7,14){$...$}
\put(-8,24){$\mathcal{G}^{n-1}$}
\put(-7,34){$\mathcal{G}^{n}$}
\put(4,-5){$\mathcal{G}_{1}$}
\put(14,-5){$...$}
\put(22,-5){$\mathcal{G}_{m-1}$}
\put(33,-5){$\mathcal{G}_{m}$}

\put(42,19){$\Longrightarrow$}
\dashline{1}(50,0)(70,0)
\dashline{1}(50,0)(50,40)
\dashline{1}(50,40)(70,40)
\dashline{1}(90,0)(90,20)
\dashline{1}(60,0)(60,40)
\dashline{1}(50,10)(60,10)
\dashline{1}(50,20)(60,20)
\dashline{1}(50,30)(60,30)
\dashline{1}(70,0)(90,0)

\put(70,40){\line(1,0){20}}
\put(90,20){\line(0,1){20}}
\put(60,10){\line(1,0){30}}
\put(60,20){\line(1,0){30}}
\put(60,30){\line(1,0){30}}
\put(70,0){\line(0,1){40}}
\put(80,0){\line(0,1){40}}

\put(69,9){$\bullet$}
\put(69,19){$\bullet$}
\put(69,29){$\bullet$}

\put(74,4){$\omega$}

\put(92,19){$\Longrightarrow$}
\dashline{1}(100,0)(140,0)
\dashline{1}(100,10)(140,10)
\dashline{1}(100,20)(140,20)
\dashline{1}(100,30)(120,30)
\dashline{1}(110,0)(110,40)
\dashline{1}(120,0)(120,40)
\dashline{1}(130,0)(130,20)

\dashline{1}(100,0)(100,40)
\dashline{1}(100,40)(120,40)
\dashline{1}(140,0)(140,20)
\put(120,30){\line(1,0){20}}
\put(130,20){\line(0,1){20}}
\put(120,40){\line(1,0){20}}
\put(140,20){\line(0,1){20}}
\put(124,4){$\omega$}
\put(124,14){$\omega$}
\put(114,24){$\omega$}
\end{picture}
\caption{ Elimination process of Step 3 and 4.}\label{fig:APElim34}
\end{figure}
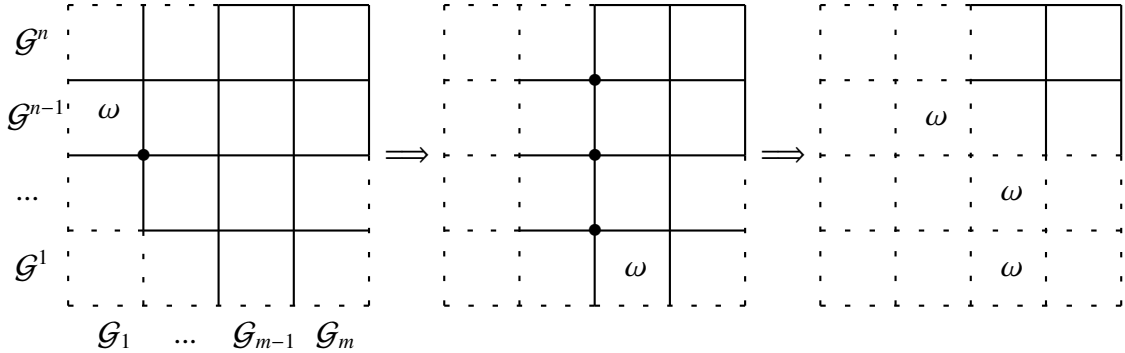

\noindent{\bf Step 4.} From  Lemma \ref{lem:basis2}, there exists a unique function $\varphi_{m-1}^{n-1}\in V_{X}^{\rm R}$ and $w_{3}-\varphi_{m-1}^{n-1}=0$.
This way, we have verified that $v = (\sum_{i=2}^{m-1}\varphi_{b,i} + \sum_{i=2}^{m-1}\varphi_{t,i} + \sum_{j=2}^{n-1}\varphi_{l,j} + \sum_{j=2}^{n-1}\varphi_{r,j}) + (\sum_{i = m-1}^{m} \varphi_{{\rm col},i} + \sum_{j = n-1}^{n} \varphi_{{\rm row},j}) + \sum_{i=1}^{m-2}\sum_{j=1}^{n-2}\varphi^j_{i} + \varphi_{m-1}^{n-1}$, and this representation by these $mn+1$ functions is unique. 
\end{proof}
%%%%%
\begin{remark}{\rm 
From Theorem \ref{thm:APbasisofRRM}, it holds that ${\rm dim} V_{hs}^{\rm{R}} =  mn+1$.}
\end{remark}
%%%%%%%%
\begin{proposition}\label{pro:relationof3}{\rm It holds for $V_{h}^{\rm{W}}$, $V_{h0}^{\rm{W}}$, $V_{h}^{\rm{M}}$, $V_{hs}^{\rm{M}}$, $V_{h}^{\rm{R}}$, and $V_{hs}^{\rm{R}}$ that
\begin{align}
V_{h}^{\rm{R}} & = V_{h}^{\rm{M}} \cap  V_{h}^{\rm{W}};\\
V_{hs}^{\rm{R}} & = V_{hs}^{\rm{M}} \cap  V_{h0}^{\rm{W}}.
\end{align}}
\end{proposition}

\end{document}